\documentclass{mscs}
\usepackage{amsmath, amssymb, amscd, amsfonts, amstext,
  amsbsy, enumerate, mathrsfs, hyperref, upgreek, mathtools, stmaryrd}
\usepackage[shadow]{todonotes}

\newcommand{\F}{\mathcal{F}}
\newcommand{\G}{\mathcal{G}}
\newcommand{\RR}{\mathbb{R}}
\newcommand{\QQ}{\mathbb{Q}}

\newcommand{\id}{\operatorname{id}}

\newcommand{\pow}{\mathscr{P}}

\newcommand{\Int}{\operatorname{int}}
\newcommand{\rk}{\operatorname{rk}}
\newcommand{\diam}{\operatorname{diam}}
\newcommand{\dom}{\operatorname{dom}}
\newcommand{\range}{\operatorname{range}}

\newcommand{\ZF}{{\sf ZF}}

\newtheorem{theorem}{Theorem}[section]
\newtheorem{lemma}[theorem]{Lemma}
\newtheorem{corollary}[theorem]{Corollary}
\newtheorem{proposition}[theorem]{Proposition}

\newtheorem{question}[theorem]{Question}

\newtheorem{claim}{Claim}[theorem]

\newtheorem{definition}[theorem]{Definition}

\newtheorem{example}[theorem]{Example}
\newtheorem{remark}[theorem]{Remark}

\begin{document}

\title{Wadge-like reducibilities on arbitrary quasi-Polish spaces}

\date{\today}

\author[L.~Motto Ros, P.~Schlicht and V.~Selivanov]{L\ls U\ls C\ls A\ns M\ls O\ls T\ls T\ls O\ns R\ls O\ls S$^1$, P\ls H\ls I\ls L\ls I\ls P\ls P\ns S\ls C\ls H\ls L\ls I\ls C\ls H\ls T$^2$ \ns and \ns V\ls I\ls C\ls T\ls O\ls R\ns  S\ls E\ls L\ls I\ls V\ls A\ls N\ls O\ls V$^3$%
\thanks{Victor Selivanov has been supported by a Marie Curie International Research Staff Exchange Scheme Fellowship within the 7th European Community Framework Programme and by DFG-RFBR Grant  436 RUS 113/1002/01.}\\
$^1$ Abteilung f\"ur Mathematische Logik, Mathematisches Institut, Albert-Ludwigs-Universit\"at Freiburg,  \addressbreak
Eckerstra\ss e 1,
 D-79104 Freiburg im Breisgau,
Germany. \addressbreak
\texttt{luca.motto.ros@math.uni-freiburg.de}
\addressbreak
$^2$ Mathematisches Institut,  Universit\"at Bonn,
Endenicher Allee 60, 53115 Bonn, Germany.
\addressbreak
\texttt{schlicht@math.uni-bonn.de}
\addressbreak
$^3$ A.~P.~Ershov Institute of
Informatics Systems, Siberian Division Russian Academy of Sciences,
\addressbreak Novosibirsk, Russia.
\addressbreak
\texttt{vseliv@iis.nsk.su}
}


\maketitle

\begin{abstract}
The structure of the Wadge degrees on
zero-dimensional spaces is very simple  (almost well-ordered), but
for many other natural non-zero-dimensional spaces (including
the space of reals)   this structure  is much more complicated. We
consider weaker notions of reducibility, including the so-called \( \mathbf{\Delta}^0_\alpha \)-reductions, and try to find for various natural topological spaces \( X \) the least ordinal \( \alpha_X \) such
that for every \( \alpha_X \leq \beta < \omega_1 \) the degree-structure induced on \( X \) by the \( \mathbf{\Delta}^0_\beta \)-reductions is
simple (i.e.\ similar to the Wadge hierarchy on the
Baire space). We show that \( \alpha_X \leq \omega \) for every quasi-Polish
space \( X \), that \( \alpha_X \leq 3 \) for
quasi-Polish spaces of dimension \( \neq \infty \), and that this last bound is in fact optimal for many (quasi-)Polish spaces, including the real line and its powers.
\end{abstract}

\tableofcontents

\section{Introduction}\label{sectionintro}

A subset \( A \) of the Baire space
\( \mathcal{N}=\omega^\omega \) is \emph{Wadge reducible} to a subset \( B \) if and only if
\( A=f^{-1}(B) \) for some continuous function \( f \colon \mathcal{N} \to \mathcal{N} \).
The structure of Wadge degrees (i.e.\ the quotient-structure of
\( (\mathscr{P}(\mathcal{N}),\leq_\mathsf{W}) \)) is fairly well understood and turns our
to be rather simple.  In particular, the structure
\( (\mathbf{B}(\mathcal{N}),\leq_\mathsf{W}) \) of the Wadge degrees of  Borel sets is semi-well-ordered by~\cite{wad84}, i.e.\ it has no infinite descending chain and for every
$A,B \in \mathbf{B}(\mathcal{N})$ we have $A \leq_\mathsf{W}B$ or
$\mathcal{N} \setminus B \leq_\mathsf{W} A$, which implies that the antichains have size at most two.
More generally: if all the Boolean combinations of sets in a pointclass \( \boldsymbol{\Gamma} \subseteq \mathscr{P}(\mathcal{N}) \) (closed under continuous preimages) are determined, then the Wadge structure restricted to \( \boldsymbol{\Gamma} \) is semi-well-ordered. For example, under the Axiom of Projective Determinacy \( \mathsf{PD} \) the structure of the Wadge degrees of projective sets is semi-well-ordered, and under the Axiom of Determinacy \( \mathsf{AD} \), the whole Wadge degree-structure remains semi-well-ordered.

The Wadge degree-structure refines the structure of levels (more
precisely, of the Wadge complete sets in those levels) of several important
hierarchies, like the stratification of the Borel sets in \( \mathbf{\Sigma}^0_\alpha \) and \(\mathbf{\Pi}^0_\alpha \) sets, or the Hausdorff-Kuratowski difference hierarchies, and serves as a nice tool to measure the topological complexity
of many problems of interest in descriptive set theory (DST)~\cite{ke94}, automata theory~\cite{pp04,s08a}, and computable analysis (CA)~\cite{wei00}.

There are several reasons and several ways to generalize the
Wadge reducibility $\leq_\mathsf{W}$  on the Baire space. For example, one can consider

\begin{enumerate}[(1)]
\item other natural classes of reducing functions
in place of the continuous functions;

\item more complicated topological spaces instead
of \( \mathcal{N} \) (the notion of Wadge reducibility makes
sense for arbitrary topological spaces);

\item reducibility between functions rather than reducibility between sets (the sets may be
identified with their characteristic functions);

\item more complicated reductions than the many-one reductions by continuous functions.
\end{enumerate}

In any of the mentioned directions a certain progress has been
achieved, although in many cases the situation typically becomes more
complicated than in the classical case.

For what concerns the possibility of using other sets of functions as reducibilities between subsets of \( \mathcal{N} \), in a series of papers, A.~Andretta, D.~A.~Martin, and L.~Motto Ros
 considered the degree-structures obtained by replacing continuous functions with one of the following classes:
\begin{enumerate}[(a)]
\item
the class of Borel functions, i.e.\ of those \( f \colon \mathcal{N} \to \mathcal{N} \) such that \( f^{-1}(U) \) is Borel for every open (equivalently, Borel) set \( U \) (see~\cite{andmar});
\item
the class \( \mathsf{D}_\alpha \) of \( \mathbf{\Delta}^0_\alpha \)-functions (for \( \alpha < \omega_1 \)), i.e.\ of those \( f \colon \mathcal{N} \to \mathcal{N} \) such that \( f^{-1}(D) \in \mathbf{\Delta}^0_\alpha \) for every \( D \in \mathbf{\Delta}^0_\alpha \) (see~\cite{an06} for the case \( \alpha = 2 \), and~\cite{ros09} for the general case);
\item
for \( \gamma < \omega_1 \) an additively closed ordinal, the collection \( \mathcal{B}_\gamma \) of all functions of Baire class \( < \gamma \), i.e.\ of those \( f \colon \mathcal{N} \to \mathcal{N} \) for which there is \( \alpha < \gamma \) such that \( f^{-1}(U) \in \mathbf{\Sigma}^0_\alpha \) for every open set \( U \) (see~\cite{motbaire})\footnote{Notice that we cannot take a single level of the Baire stratification because in general it is not closed under composition, and hence does not give a preorder when used as reducibility between sets of reals.};
\item
the class of \( \mathbf{\Sigma}^1_n \) functions (for \( n \in \omega \)), i.e.\ the class of those \( f \colon \mathcal{N} \to \mathcal{N} \) such that \( f^{-1}(U) \in \mathbf{\Sigma}^1_n \) for every open (equivalently, \( \mathbf{\Sigma}^1_n \)) set \( U \) (see~\cite{motsuperamenable}).
\end{enumerate}
It turns out that the degree-structures resulting from (a)--(c), as for the Wadge degrees case, are all semi-well-ordered when restricted to the class of Borel sets or, provided that all Boolean combinations of sets in \( \boldsymbol{\Gamma} \) are determined, to any pointclass \( \boldsymbol{\Gamma} \subseteq \pow (\mathcal{N}) \) closed under continuous preimages (hence, in particular, to the entire \( \pow(\mathcal{N}) \) when \( \mathsf{AD} \) is assumed),\footnote{In fact, for the cases of Borel functions and \( \mathbf{\Delta}^0_\alpha \)-functions, the corresponding degree-structures are even isomorphic to the Wadge one.}
and that under the full \( \mathsf{AD} \) also the degree-structures resulting from (d) are semi-well-ordered (on the entire \( \pow (\mathcal{N}) \)): thus, we obtain a series of natural
classifications of subsets of the
Baire space which are weaker than the Wadge one.

Concerning Polish spaces different from the Baire space, using the
methods developed in~\cite{wad84} it is immediate to check that
the structure of Wadge degrees on any zero-dimensional Polish
space remains semi-well-ordered (this follows also from
Proposition~\ref{propfr2}). On the other hand, P.~Hertling showed
in~\cite{he96}  that the Wadge hierarchy on the real line \( \RR
\) is much more complicated than the structure of Wadge degrees on
the Baire space. In particular, there are infinite antichains and
infinite descending chains in the structure of Wadge degrees of
$\mathbf{\Delta}^0_2$ sets. Recently, this result has been
considerably strengthened in~\cite{sc10}. Moreover, P.~Schlicht
also showed  in~\cite{sc11} that the structure of Wadge degrees on
\emph{any} non zero-dimensional metric space must contain infinite
antichains, and V.~Selivanov  showed in~\cite{s05} that the Wadge
hierarchy is  more complicated also when considering other natural
topological spaces (e.g.\ the so-called \( \omega \)-algebraic
domains).

As already noted, if one passes from continuous reductions between sets to continuous reductions between functions, the situation becomes much more intricate. Even when considering the simplest
possible generalization, namely continuous reductions between partitions of the Baire space into
\( 3 \leq k \in \omega \) subsets, the degree-structure obtained is rather complicated, e.g.\  there are antichains of arbitrarily large finite size. On the other hand, it is still a well-quasi-order (briefly, a wqo), i.e.\ it has neither infinite
descending chains nor infinite antichains: hence it can still serve
as a scale to measure the topological complexity of $k$-partitions
of the Baire space --- see the end of Subsection~\ref{subsectionreducibilities} and the references contained therein.

In the fourth direction (more complicated reductions), the so called Weihrauch reducibility became recently
rather popular: it turns out to be very useful in characterizing the
topological complexity of some important computational problems, and also in understanding
the computational content of some important classical mathematical
theorems ---
see e.g.\ \cite{he96,bg11,bg11a,ksz10}.

In this paper we aim to make the first three kinds of generalizations interact with each other,
namely we will consider some weaker versions of the Wadge reducibility (including the ones mentioned above), and study the degree structures induced by them on arbitrary \emph{quasi-Polish} spaces, a collection of  spaces recently identified in~\cite{br} by M.~de
Brecht  as a natural class of spaces for DST and
CA. Each of these degree-structures should be intended as a tool for measuring the complexity of subsets (or partitions) of the space under consideration: a structure like the Wadge one is nearly optimal for this goal, but, as already noticed, we get a reasonable notion of complexity also if the structure is just a wqo. If instead the degree structure contains infinite antichains but is
well-founded, then we can at least assign a rank to the degrees (even if this rank could be not completely meaningful), while if it is also ill-founded it becomes completely useless as a notion of classification.
These considerations justify the following terminology: a degree-structure obtained by considering a notion of reducibility (between sets or partitions) on a topological space will be called

\begin{enumerate}[-]
\item
\textbf{\emph{very good}} if it is semi-well-ordered;
\item
\textbf{\emph{good}} if it is a wqo;
\item
\textbf{\emph{bad}} if it contains infinite antichains;
\item
\textbf{\emph{very bad}} if it contains both infinite descending chains and infinite antichains.
\end{enumerate}

By the results mentioned above, the Wadge hierarchy on any non zero-dimensional Polish space is always
bad, but we will show that for many other natural reducibilities, the corresponding hierarchy is very good on a great number of spaces. This is obtained by computing the minimal complexity of an isomorphism between such spaces and the Baire space. In particular, after recalling
some preliminaries in Section~\ref{sectionprel} and introducing various reducibility notions in Section~\ref{sectionreduc}, we will show in Section~\ref{sectionisom} that all uncountable quasi-Polish spaces are pairwise \( \mathsf{D}_\omega \)-isomorphic, and that any quasi-Polish space of topological dimension \( \neq \infty \) is even \( \mathsf{D}_3 \)-isomorphic to \( \mathcal{N} \) (and that, in general, the indices \( \omega \) and \( 3 \)  cannot be lowered). This fact, together with the results from~\cite{ros09,motbaire}, implies that the degree-structures induced by the classes of functions \( \mathsf{D}_\alpha \) and \( \mathcal{B}_\gamma \) (where \( \gamma < \omega_1 \) is additively closed) on \emph{any} uncountable quasi-Polish space are very good (when restricted to the degrees of Borel sets, or even, under corresponding determinacy assumptions, to the degrees of sets in any larger pointclass \( \boldsymbol{\Gamma} \subseteq \pow(\mathcal{N}) \)) whenever \( \alpha \geq \omega \), and that the same is true also for \( \alpha \geq 3 \) when considering quasi-Polish spaces of dimension \( \neq \infty \). In Section~\ref{sectionhierarchies} we will show that these results are nearly optimal by showing that the degree-structure induced by the class of functions \( \mathsf{D}_2 \) is (very) bad on many natural Polish spaces (like the real line \( \RR \) and its powers), and that the Wadge hierarchy can fail to be very good also on extremely simple countable quasi-Polish spaces.

\section{Notation and preliminaries}\label{sectionprel}

In this section  we introduce a great deal of notation that will be
used throughout the paper. The notation for pointclasses and for
isomorphisms between topological spaces will be introduced at the
beginning of Sections~\ref{sectionreduc} and~\ref{sectionisom},
respectively.

\subsection{Notation}

Unless otherwise specified, we will always work in
\( \ZF + \mathsf{DC} \), i.e.\ in the usual Zermelo-Fr\ae nkel set
theory together with the Axiom of Dependent Choice.

We freely  use the standard set-theoretic notation like \( |X| \)
for the cardinality of $X$, $X\times Y$ for the Cartesian product of
\( X \) and \( Y \), \( X \sqcup Y \) for the disjoint union of \( X
\) and \( Y \), $Y^X$ for the set of all functions $f \colon X\to
Y$, and $\mathscr{P}(X)$ for the set of all subsets of $X$. Given a
product \( X \times Y \) we denote by \( \pi_0 \) (respectively, \(
\pi_1 \)) the projection on the first (respectively, the second)
coordinate. For $A\subseteq X$, $\overline{A}$ denotes the
complement $X\setminus A$ of $A$ in $X$. For $\mathcal{A}\subseteq
\mathscr{P}(X)$, $BC(\mathcal{A})$ denotes the Boolean closure of
$\mathcal{A}$, i.e. the set of finite Boolean combinations of sets
in $\mathcal{A}$. We identify each nonzero natural number \( n \)
with the set of its predecessors \( \{ 0, \dotsc, n-1 \} \), and the
set of natural numbers, which will be denoted by \( \omega \), with
its order type under \( < \). The first uncountable ordinal is
denoted by \( \omega_1 \), while the class of all ordinal numbers is
denoted by \( \mathsf{On} \). Given a set \( X \) and a natural
number \( i \in \omega \), we let \( [X]^i  = \{ Y \subseteq X \mid
|Y|  = i \} \). Given an arbitrary partially ordered set \( (X, \leq
) \) (briefly, a \emph{poset}), we denote by \( < \) its strict
part, i.e.\ the relation on \( X \) defined by \( x < y \iff {x \leq
y \wedge x \neq y }\).

\subsection{Spaces and pointclasses}\label{subsectionspaces}

We assume the reader be familiar with the basic notions of
topology. The collection of all open subsets of a space $X$ (i.e.\
the topology of \( X \)) is denoted by $\tau_X$ or, when the space
is clear from the context, simply by \( \tau \). We abbreviate
``topological space'' to ``space'' and denote by \( \mathscr{X} \)
the collection of all (topological) spaces. Let \( \mathscr{Y}
\subseteq \mathscr{X} \): we say that \( X \in \mathscr{X} \) is
\emph{universal for \( \mathscr{Y} \)} if every \( Y \in \mathscr{Y}
\) can be \emph{topologically embedded} into \( X \), i.e.\ there is a
subspace \( X' \subseteq X \) such that \(Y\) is homeomorphic to \(
X'\) (where \( X' \) is endowed with the relative topology inherited from \( X \)).
A space $X$ is \emph{connected} if
there are no nonempty clopen proper subsets of $X$, and \emph{totally disconnected} if every connected subset contains at most one point.
A space $X$ is
called \emph{locally connected} if every element has arbitrarily
small connected open neighborhoods.
A space \( X \) is called \emph{\(\sigma\)-compact} if it can be written as a countable union of compact sets.
For  any space $X$, define the transfinite descending sequence
$\langle X^{(\alpha)} \mid \alpha \in \mathsf{On} \rangle$ of closed subsets of $X$ as follows:
$X^{(0)}=X$, $X^{(\alpha+1)} = $ the set of non-isolated points of
$X^{(\alpha)}$ (where $x$ is an isolated point of a space \( X \) if $\{x\}$ is
open in \( X \)), and $X^{(\alpha)}=\bigcap\{X^{(\beta)}\mid \beta<\alpha\}$ if $\alpha$ is a limit ordinal.  The space
$X$ is called {\em scattered} if and only if $\bigcap_{\alpha \in \mathsf{On}}
X^{(\alpha)}=\emptyset$.

Let ${\mathcal N}=\omega^\omega$ be the set of all infinite
sequences of natural numbers (i.e.\ of all functions
$\xi \colon \omega \to \omega$). Let $\omega^*$ be the set of
finite sequences of elements of $\omega$, including the empty
sequence. For $\sigma\in\omega^*$ and $\xi\in{\mathcal N}$, we
write $\sigma\sqsubseteq \xi$ to denote that $\sigma$ is an
initial segment of $\xi$. We denote the concatenation of $\sigma$
and $\xi$ by
$\sigma\xi=\sigma\cdot\xi$ , and  the set of all
extensions of $\sigma$ in $\mathcal{N}$ by $\sigma\cdot\mathcal{N}$. For $\xi \in \mathcal{N}$, we
can write $\xi = \xi(0) \xi(1)\cdots$ where $\xi(i)\in\omega$ for each
$i<\omega$.
Notations in the style of regular expressions like
$0^\omega$, $0^m1^n$ or $0^\ast 1$ have the obvious standard
meaning: for example, \( 0^\omega \) is the \( \omega \)-sequence constantly equal to \( 0 \), \( 0^m 1^n \) is the sequence formed by \( m \)-many \( 0 \)'s followed by \( n \)-many \( 1 \)'s, \( 0^\ast 1 = \{ 0^m 1 \mid m \in \omega \} \) is the set of all sequences constituted by a finite (possibly empty) sequence of \( 0 \)'s followed by (a unique) \( 1 \),  and so on.

When we endow \( \mathcal{N} \) with the product  of the discrete
topologies on \(\omega\) we obtain the so-called \emph{Baire
space}. This topology coincides with the topology generated by
 the collection of sets of the form
$\sigma\cdot\mathcal{N}$ for $\sigma\in\omega^*$. The Baire space
is of primary interest for DST and CA: its importance stems from
the fact that many countable objects are coded straightforwardly
by elements of $\mathcal{N}$, and it has very specific topological
properties. In particular, it is a perfect zero-dimensional space
and  the spaces ${\mathcal N}^2$, ${\mathcal N}^\omega$,
$\omega \times {\mathcal N}={\mathcal N}\sqcup{\mathcal
N} \sqcup \dotsc$ (endowed with the product topology) are all
homeomorphic to ${\mathcal N}$.

The subspace ${\mathcal C}=2^\omega$ of ${\mathcal N}$  formed by
the infinite binary strings (endowed with the relative topology
inherited from \( \mathcal{N} \)) is known as the {\em Cantor
space}. In this paper, we will also consider the space $\omega$
(with the discrete topology), the space $\mathbb{R}$ of reals
(with the standard topology), and the space of irrationals number (with the relative topology inherited from
$\mathbb{R}$), which is homeomorphic to \( \mathcal{N} \).

A {\em pointclass} on the space \( X \) is a collection \(
\mathbf{\Gamma}(X) \)   of subsets of \( X \). A {\em family of
pointclasses} is a family \( \mathbf{\Gamma}=\{\mathbf{\Gamma}(X)
\mid X  \in \mathscr{X} \} \) indexed by arbitrary topological
spaces such that each \( \mathbf{\Gamma}(X) \) is a pointclass on
\( X \) and \( \mathbf{\Gamma} \) is closed under continuous
preimages, i.e.\ \( f^{-1}(A)\in\mathbf{\Gamma}(X) \) for every \(
A\in\mathbf{\Gamma}(Y) \) and every continuous function \( f
\colon X\to Y \) (families of pointclasses are sometimes called  \emph{boldface pointclasses} by other authors). In particular, any pointclass
$\mathbf{\Gamma}(X)$ in such a family is downward closed under the
Wadge reducibility on \( X \).

Trivial examples of families of pointclasses are $\mathcal{E},\mathcal{F}$,
where $\mathcal{E}(X)=\{\emptyset\}$ and $\mathcal{F}(X)=\{X\}$ for any space $X \in \mathscr{X}$.
Another basic example  is given by the collection $\{\tau_X \mid X  \in \mathscr{X} \}$ of the topologies of all the spaces.

Finally, we define some operations on families of pointclasses
which are relevant to hierarchy theory. The usual set-theoretic
operations will be applied to the families of pointclasses
pointwise: for example, the union $\bigcup_i \mathbf{\Gamma}_i$ of
the families of pointclasses
$\mathbf{\Gamma}_0,\mathbf{\Gamma}_1,\ldots$ is defined by
$(\bigcup_i\mathbf{\Gamma}_i)(X)=\bigcup_i\mathbf{\Gamma}_i(X)$. A
large class of such operations is induced by the set-theoretic
operations of L.~V.~Kantorovich and E.~M.~Livenson which are now
better known under the name ``$\omega$-Boolean operations'' (see~\cite{s11a} for the general definition). Among them are the
operation $\mathbf{\Gamma}\mapsto\mathbf{\Gamma}_\sigma$, where
$\mathbf{\Gamma}(X)_\sigma$ is the set of all countable unions of
sets in $\mathbf{\Gamma}(X)$, the operation
$\mathbf{\Gamma}\mapsto\mathbf{\Gamma}_c$, where
$\mathbf{\Gamma}(X)_c$ is the set of all complements of sets in
$\mathbf{\Gamma}(X)$, and the operation
$\mathbf{\Gamma}\mapsto\mathbf{\Gamma}_d$, where
$\mathbf{\Gamma}(X)_d$ is the set of all differences of sets in
$\mathbf{\Gamma}(X)$.

\subsection{Classical hierarchies in arbitrary spaces}\label{subsectionhier}

First we   recall from~\cite{s04a} the
definition of the Borel hierarchy  in arbitrary spaces.

\begin{definition}\label{defbh}
For \( \alpha<\omega_1 \), define the family of poinclasses  \(
\mathbf{\Sigma}^0_\alpha =  \{ {\bf\Sigma}^0_\alpha(X) \mid X \in
\mathscr{X}\} \) by induction on $\alpha$ as follows:
${\bf\Sigma}^0_0(X)=\{\emptyset\}$, ${\bf\Sigma}^0_1(X) = \tau_X$,
and ${\bf\Sigma}^0_2(X) = ((\mathbf{\Sigma}^0_1(X))_d)_\sigma$  is
the collection of all countable unions of differences
of open sets. For \( \alpha > 2 \), ${\bf\Sigma}^0_\alpha(X) =
(\bigcup_{\beta<\alpha}({\bf\Sigma}^0_\beta(X))_c)_\sigma$ is the
class of countable unions  of sets in \(
\bigcup_{\beta<\alpha}({\bf\Sigma}^0_\beta(X))_c) \).

We also let \( {\bf\Pi}^0_\beta(X)= (\mathbf{\Sigma}^0_\beta(X))_c
\)  and \( \mathbf{\Delta}^0_\alpha = \mathbf{\Sigma}^0_\alpha
\cap \mathbf{\Pi}^0_\alpha \). We call a set \emph{proper  \(
{\bf\Sigma}^0_\beta(X) \)} if it is in \( {\bf\Sigma}^0_\beta(X)
\setminus {\bf\Pi}^0_\beta(X) \).
\end{definition}

Notice that by definition \( \mathbf{\Sigma}^0_0  = \mathcal{E} \) and \(
\mathbf{\Pi}^0_0 = \mathcal{F} \). When \( \lambda < \omega_1 \)
is a limit ordinal, we let \( \mathbf{\Sigma}^0_{<\lambda} =
\bigcup_{\alpha < \lambda } \mathbf{\Sigma}^0_\alpha \), and
similarly \( \mathbf{\Pi}^0_{<\lambda} = \bigcup_{\alpha < \lambda
} \mathbf{\Pi}^0_\alpha \) and \( \mathbf{\Delta}^0_{<\lambda} =
\bigcup_{\alpha < \lambda } \mathbf{\Delta}^0_\alpha \). Notice
that all of \( \mathbf{\Sigma}^0_{< \lambda} \), \(
\mathbf{\Pi}^0_{< \lambda} \) and \( \mathbf{\Delta}^0_{< \lambda
} \) are families of pointclasses as well.

The sequence $\langle \boldsymbol{\Sigma}^0_\alpha(X) , \boldsymbol{\Pi}^0_\alpha(X), \boldsymbol{\Delta}^0_\alpha(X)\mid \alpha<\omega_1
\rangle$ is called the {\em Borel hierarchy} of $X$. The pointclasses
${\bf\Sigma}^0_\alpha(X)$, ${\bf\Pi}^0_\alpha(X)$ are the {\em
non-selfdual levels} of the hierarchy (i.e.\ they are the levels which are not closed under complementation), while the pointclasses
${\bf\Delta}^0_\alpha(X)={\bf\Sigma}^0_\alpha(X)\cap{\bf\Pi}^0_\alpha(X)$
are the {\em self-dual levels}  (as is usual in
DST, we will apply this terminology also to the levels of the other
hierarchies considered below).  The pointclass ${\mathbf B}(X)$ of
{\em Borel sets} of $X$ is the union of all levels of the Borel
hierarchy, and \( \mathbf{B} = \{ \mathbf{B}(X) \mid X \in \mathscr{X} \} \) is the family of pointclasses of Borel sets.   It is straightforward to check by induction on \(
\alpha,\beta < \omega_1 \) that using Definition~\ref{defbh} one has the
following result.

\begin{proposition}\label{propbh}
For every  $X \in \mathscr{X}$ and for all $\alpha <\beta < \omega_1$,
${\bf\Sigma}^0_\alpha(X), \mathbf{\Pi}^0_\alpha(X) \subseteq {\bf\Delta}^0_\beta(X) \subseteq \boldsymbol{\Sigma}^0_\beta(X),\boldsymbol{\Pi}^0_\beta(X)$.
\end{proposition}

Thus if \( \lambda < \omega_1 \)  is a limit ordinal we have \(
\mathbf{\Sigma}^0_{< \lambda} = \mathbf{\Pi}^0_{< \lambda} =
\mathbf{\Delta}^0_{< \lambda} \).

\begin{remark}
Definition~\ref{defbh} applies to all the spaces \( X \in \mathscr{X}
\), and Proposition~\ref{propbh} holds true in the full
generality. Note that Definition~\ref{defbh} differs from the
classical definition for Polish spaces (see e.g.\ \cite[Section
11.B]{ke94}) only for the level 2, and that for the case of Polish
spaces our definition of Borel hierarchy is equivalent to the
classical one. The classical definition cannot be
applied in general  to non metrizable spaces \( X \) (like e.g.\ the non
discrete \(\omega\)-algebraic domains) precisely because with that definition the
inclusion \( \mathbf{\Sigma}^0_1(X)
 \subseteq \mathbf{\Sigma}^0_2(X) \) may fail.
\end{remark}

The Borel hierarchy is refined by the difference  hierarchies
(over the family of pointclasses \( \mathbf{\Sigma}^0_\alpha \),
\( \alpha < \omega_1 \)) introduced by Hausdorff and Kuratowski. Recall that an
ordinal $\alpha$ is called {\em even}  (respectively, {\em odd})
if $\alpha=\lambda+n$ where $\lambda$ is either zero or a limit
ordinal, $n<\omega$, and the number $n$ is even (respectively,
odd). For an ordinal $\alpha$, let $r(\alpha)=0$ if $\alpha$ is
even and $r(\alpha)=1$, otherwise. For any ordinal $1 \leq \alpha < \omega_1$,
consider the operation $D_\alpha$ sending any  sequence $\langle
A_\beta \mid \beta<\alpha \rangle$ of subsets of a space \( X \)
 to the subset of \( X \)
\[
D_\alpha(\langle A_\beta \mid \beta<\alpha \rangle)=
\bigcup\Big\{A_\beta\setminus\bigcup\nolimits_{\gamma<\beta}A_\gamma\mid
\beta<\alpha,\,r(\beta) \neq r(\alpha)\Big\}.
\]

\begin{definition}
 For any ordinal $1 \leq \alpha<\omega_1$ and  any family of pointclasses $\mathbf{\Gamma}$,
let $D_\alpha(\mathbf{\Gamma})(X)$ be the class of all sets of the form
$D_\alpha(\langle A_\beta \mid \beta<\alpha \rangle)$, where the $A_\beta$'s form an increasing (with respect to inclusion) sequence of sets in $\mathbf{\Gamma}(X)$, and then set \( D_\alpha(\boldsymbol{\Gamma}) = \{ D_\alpha(\boldsymbol{\Gamma})(X) \mid X \in \mathscr{X} \} \).

To simplify the notation, when \( \mathbf{\Gamma} = \mathbf{\Sigma}^0_1 \) we set
${\bf\Sigma}^{-1}_\alpha(X)=D_\alpha({\bf\Sigma}^0_1)(X)$, \( {\bf\Pi}^{-1}_\alpha(X) = \{ X \setminus A  \mid A \in \mathbf{\Sigma}^{-1}_\alpha(X) \} \), and \( \mathbf{\Delta}^{-1}_\alpha(X) = \mathbf{\Sigma}^{-1}_\alpha(X) \cap \mathbf{\Pi}^{-1}_\alpha(X) \) for every \( 1 \leq \alpha < \omega_1 \). Finally, we further set \( \mathbf{\Sigma}^{-1}_0(X) = \{ \emptyset \} \) and \( \mathbf{\Pi}^{-1}_0(X) = \{ X \} \).
\end{definition}

\noindent
For example, ${\bf\Sigma}^{-1}_4(X)$ and ${\bf\Sigma}^{-1}_\omega(X)$ consist of the the sets of the form, respectively,
$(A_1\setminus A_0)\cup (A_3\setminus A_2)$  and
$\bigcup_{i<\omega}(A_{2i+1}\setminus A_{2i})$, where the $A_i$'s form an increasing sequence of
open subsets of $X$. Notice that the requirement that the sequence of the \( A_\beta \)'s  be increasing can be dropped (yielding to an equivalent definition of the pointclass \( D_\alpha(\boldsymbol{\Gamma})(X) \)) when \( \boldsymbol{\Gamma}(X) \) is closed under countable unions. This in particular applies to the case when \( \boldsymbol{\Gamma} \) is the pointclass of open sets or one of the pointclasses \( \boldsymbol{\Sigma}^0_\alpha \).
Moreover, it is easy to see that any level of the difference hierarchy over \( \mathbf{\Sigma}^0_\alpha \), \( \alpha < \omega_1 \),
is again a family of pointclasses.

\subsection{$\omega$-continuous domains}\label{omega-cont}

In this section we will briefly review the notation and some (basic) facts concerning \(\omega\)-continuous domains which will be used in the following sections. For all undefined notions and for a more detailed presentation of this topic (as well as for all omitted proofs) we refer the reader to the standard monograph~\cite{bookdomains}.

Let \( (X , \leq) \) be  an arbitrary poset.
The {\em Alexandrov topology} on
\( (X, \leq) \) is formed by taking the upward closed subsets of \( X \) as
the open sets.   The continuous functions between two spaces endowed with the
Alexandrov topology coincide with the monotone (with respect to their partial orders) functions.

Let  \( (X,\leq) \) be a poset. A set $D\subseteq X$ is {\em directed} if any two
elements of $D$ have an upper bound in $D$. The poset \( (X,\leq) \) is called a {\em directed-complete partial
order}  (briefly, \emph{dcpo}) if any non-empty directed subset of $X$ has a
supremum in $X$.  The {\em Scott topology}
on a dcpo $(X,\leq)$ is formed by taking as open sets all the
upward closed sets $U\subseteq X$ such that $D\cap
U\not=\emptyset$ whenever $D$ is a non-empty directed subset of
$X$ whose supremum is in $U$. As it is well-known, every dcpo endowed with the Scott topology is automatically a \( T_0 \) space (it is enough to observe that if \( x \nleq y \) then \( x \in U \) and \( y \notin U \) for \( U = \{ z \in X \mid z \nleq y \} \), which is clearly Scott open). Note that the order $\leq$ may be
recovered from the Scott topology because it coincides
with its specialization order: \( x \leq y \) if and only if \( x \) belongs to the closure of \( \{ y \} \) with respect to the Scott topology.  An element $c \in X$ is {\em compact} if the set $\uparrow \! \!
c=\{x \mid c \leq x\}$ is open, and the set of all compact elements of \( X \) is denoted by \( X_0 \).
Note that for every \( c \in X_0 \), $\uparrow \! \! c$ is the
smallest open neighborhood of $c$, and that if \( (X, \leq)
\) has a top element, then the closure of every non-empty open set
is the entire space.

A dcpo $(X, \leq)$  is an {\em algebraic domain} if $\{\uparrow \! \!
c\mid c\in X_0\}$ is a basis for the Scott topology of $X$. An {\em $\omega$-algebraic
domain} is an algebraic domain $X$ such that $X_0$ is
countable.
 An important example of an
$\omega$-algebraic domain is the space  $P\omega$ of subsets of
$\omega$ with the Scott topology on the directed-complete lattice
$(\mathscr{P}(\omega), \subseteq)$ (\( P \omega \) is sometimes
called the \emph{Scott domain}):  in this space, the compact
elements are precisely the  finite subsets of $\omega$. Another
natural example which will be frequently considered in this paper is
\( (\omega^{\leq \omega}, \sqsubseteq ) \), where \( \omega^{\leq
\omega} = \omega^* \cup \omega^\omega \): in this case, the compact
elements are exactly the sequences in \( \omega^* \).

For any dcpo  $(X,\leq)$ and $x,y\in X$, let $x\ll y$ mean that
$y\in \Int(\uparrow \! \! x)$ where $\Int$ is the interior operator.
This relation is transitive and $x\ll y$ implies $x\leq y$. A dcpo
$(X, \leq)$ is a {\em continuous domain} if for any Scott-open set
$U$ and any $x\in U$ there is $b\in U$ with $b\ll x$. $(X,\leq)$ is
an {\em $\omega$-continuous domain} if there is a countable set
$B\subseteq X$ such that for any Scott-open set $U$ and any $x\in U$
there is $b\in U\cap B$ with $b\ll x$. Note that any
$\omega$-algebraic domain is an $\omega$-continuous domain because
we can take $B=X_0$.

In the next proposition we characterize scattered dcpo's  with the Scott topology.

\begin{proposition}\label{propco2}
Let $(X,\leq)$ be a dcpo with the  Scott topology. Then $X$ is
scattered if and only if there is no infinite \( \leq \)-ascending chain
$x_0 < x_1< \dotsc$ in $X$.
 \end{proposition}

\begin{proof}  Let $X$ have no infinite ascending chain, i.e.\ $(X,\geq)$ is well-founded. Consider the (unique) rank function
\( \rk \colon X \to \mathsf{On} \) defined by $\rk(x)=\sup\{\rk(y)+1\mid x<y\}$
for each $x\in X$. By induction,  $X^{(\alpha)}=\{x \in
X \mid \rk(x) \geq \alpha\}$ for every ordinal $\alpha \in \mathsf{On}$, hence
$\bigcap_{\alpha \in \mathsf{On} } X^{(\alpha)}=\emptyset$ and $X$ is scattered.

It remains to show that if $X$ has an infinite ascending chain
$x_0<x_1<\dotsc$ then $X$ is not scattered. It suffices to check
that the supremum $x$ of this chain is in $\bigcap_{ \alpha \in \mathsf{On}}
X^{(\alpha)}$. First notice that for every \( X' \subseteq X \) containing all the \( x_n \)'s, each \( x_n \) is not isolated in \( X' \) because \( x_m \in {\uparrow \! \! x_n \cap X'} \) (the smallest open set of \( X' \) containing \( x_n \)) for every \( m \geq n \). In particular, by induction on \( \alpha \in \mathsf{On} \) one can easily  show that \( x_n \in X^{(\alpha)} \) for every \( n \in \omega \).  We now check by induction that $x \in
X^{(\alpha)}$ for each $\alpha \in \mathsf{On}$. This is obvious when $\alpha=0$ or
$\alpha$ is a limit ordinal, so assume that $\alpha=\beta+1$.
Suppose, for a contradiction, that $x \notin X^{(\alpha)}$. Then by the inductive hypothesis
$x \in X^{(\beta)}\setminus X^{(\beta+1)}$, so $\{ x \}$ is Scott-open in $X^{(\beta)}$. Since $x_n\in X^{(\beta)}$ for all
$n<\omega$, then $x_n \in \{x\}$ for some $n<\omega$, which is a
contradiction because necessarily \( x_n \neq x \) for every \( n \in \omega \).
 \end{proof}

\begin{corollary}\label{cor21}
In any scattered dcpo, the Scott topology coincides with the
Alexandrov topology. The continuous functions between scattered
dcpo's coincide with the monotone functions.
 \end{corollary}

For future reference, we recall  a characterization of the levels of
the difference hierarchy over open sets in $\omega$-algebraic
domains obtained in~\cite{s04a} (in~\cite{s08} this was extended
to the context of $\omega$-continuous domains). Let $(X,\leq)$ be an
$\omega$-algebraic domain. A set $A\subseteq X$ is called {\em
approximable} if for any $x\in A$ there is a compact element
$c\leq x$ with $[c,x]\subseteq A$, where $[c,x]=\{y\in X\mid c\leq
y\leq x\}$.

Let \( (X,\leq ) \) be a dcpo endowed with the Scott topology. Given \( A \subseteq X \) and \( n \in \omega \),
a nondecreasing sequence \( a_0 \leq \dotsc \leq a_n \) of compact elements of \( X \) is said to be \emph{alternating for \( A \)} if \( a_i \in A \iff a_{i+1} \notin A \) for every \( i < n \). Notice that in this case we necessarily have \( a_0 < \dotsc < a_n \). For this reason, a sequence as above will be also called \emph{alternating chain for \( A \)}.
 An {\em alternating
tree for $A\subseteq X$} is a monotone function
$f \colon (T, \sqsubseteq) \to (X_0,\leq)$ such that:
\begin{enumerate}[(1)]
\item
$T\subseteq\omega^*$ is a well-founded tree (i.e.\ the partial
order $(T,\sqsupseteq)$ is  well-founded), and
\item
$f(\sigma)\in
A \iff f(\sigma n)\not\in A$, for each $\sigma n\in T$ (i.e.\ the image under \( f \) of any branch of \( T \) is an alternating chain for \( A \)).
\end{enumerate}
 The {\em
rank} of $f$ is the rank of  $(T,\sqsupseteq)$. An alternating tree $f$ is called 1-alternating
(respectively, 0-alternating) if $f(\emptyset)\in A$ (respectively,
$f(\emptyset)\not\in A$).

\begin{theorem} (\cite[Theorem 2.9]{s04a} and \cite[Proposition 4.13]{s05}) \label{t-dh}
Let
$X$ be an $\omega$-algebraic domain, $\alpha<\omega_1$, and
$A\subseteq X$. Then $A\in{\bf\Sigma}^{-1}_\alpha$ if and only if
$A$ and $X\setminus A$ are approximable and there is no
1-alternating tree of rank $\alpha$ for $A$. Moreover, if
$\alpha<\omega$  then any set in
$\mathbf{\Sigma}^{-1}_\alpha \setminus \mathbf{\Pi}^{-1}_\alpha$
(resp. in
$\mathbf{\Delta}^{-1}_{\alpha+1} \setminus (\mathbf{\Sigma}^{-1}_\alpha \cup \mathbf{\Pi}^{-1}_\alpha)$)
is Wadge complete in $\mathbf{\Sigma}^{-1}_\alpha$ (resp. in
$\mathbf{\Delta}^{-1}_{\alpha+1}$).
\end{theorem}

\subsection{Polish and quasi-Polish spaces}\label{subsectionpolish}

Recall that a space \( X \) is {\em Polish} if it is countably
based and admits a metric \( d \) compatible with its topology such that \( (X,d) \)
is a complete metric space. Examples of Polish spaces are the
Baire  space, the Cantor space, the space of reals \( \RR \) and
its Cartesian powers \( \RR^n \) (\( n \in \omega \)),  the closed
unit interval \( [0,1] \),
 the Hilbert cube \( [0,1]^\omega \), and the space \( \RR^\omega \).
It is well-known that both the Hilbert cube and \( \RR^\omega \) are universal for Polish spaces (see e.g.\
\cite[Theorem 4.14]{ke94}).

A natural variant of Polish spaces has recently emerged, the so-called \emph{quasi-Polish} spaces. This class includes all Polish spaces and all \(\omega\)-continuous domains (the main objects under consideration in DST and domain theory, respectively), and provides a unitary approach to their topological analysis. Moreover, it has shown to be a relevant class of spaces for CA. In the rest of this section we will provide the definition of these spaces and recall some of their properties that will be used later.

Given a set \( X \), call a function \( d \) from \( X\times X \) to the nonnegative reals \emph{quasi-metric} whenever \( x=y \) if and only if  \( d(x,y)=d(y,x)=0 \), and \( d(x,y)\leq d(x,z)+d(z,y) \) (but we don't require \( d \) to be symmetric). In particular, every metric is a quasi-metric. Every quasi-metric on \( X \) canonically induce the topology \( \tau_d \) on \( X \), where \( \tau_d \) is the topology generated by the open balls \( B_d(x, \varepsilon) = \{ y \in X \mid d(x,y) < \varepsilon \} \) for \( x \in X \) and \( 0 \neq \varepsilon \in \RR^+ \). A (topological) space \( X \) is called \emph{quasi-metrizable} if there is a quasi-metric on \( X \) which generates its topology.
If \( d \) is a quasi-metric on \( X \), let \( \hat{d} \) be the metric on \( X \) defined by \( \hat{d}(x,y)= \max\{d(x,y),d(y,x)\} \). A sequence \( \langle x_n \mid n \in \omega \rangle \) is called \emph{\( d \)-Cauchy sequence} if for every \( 0 \neq \varepsilon \in \RR^+ \) there is \( N \in \omega \) such that \( d(x_n,x_m) < \varepsilon \) for all \( N \leq n \leq m \). We say that the quasi-metric \( d \) on \( X \) is \emph{complete} if every \( d \)-Cauchy sequence converges with respect to \( \hat{d} \) (notice that this definition is coherent with the notion of completeness for a metric \( d \), as in this case \( \hat{d} = d \)).

\begin{definition}
A \( T_0 \) space \( X \) is called \emph{quasi-Polish} if it is countably based and there is a complete quasi-metric which generates its topology. When we fix a particular compatible complete quasi-metric \( d \) on \( X \), we say that \( (X,d) \) is a \emph{quasi-Polish metric space}.
\end{definition}

Notice that every Polish space is automatically quasi-Polish, but, as recalled above, also every \(\omega\)-continuous domain is quasi-Polish by \cite[Corollary 45]{br}. For example, a complete quasi-metric which is compatible with the topology of the Scott domain \( P \omega \) is given by \( d(x,y) = 0 \) if \( x \subseteq y \) and \( d(x,y) = 2^{-(n+1)} \) if \( n \) is the smallest element in \( x \setminus y \) (for every, \( x,y \subseteq \omega \)).  De Brecht's paper~\cite{br} shows that there is  a reasonable descriptive
set theory for the class of quasi-Polish spaces which extends the classical theory for Polish spaces in many directions, for example:

\begin{proposition}\cite[Corollary 23]{br} \label{propsubspace}
A subspace of a quasi-Polish space \( X \) is quasi-Polish if and only if it is \( \mathbf{\Pi}^0_2(X) \).
\end{proposition}

 It is not difficult to see that if \( (X, d) \) is a quasi-Polish metric space then \( (X, \hat{d}) \) is a Polish metric space, and that the following holds.

\begin{proposition}\cite[Corollary 14]{br}\label{propqp}
For every quasi-Polish metric space \( (X, d) \), $ \tau_d \subseteq \tau_{\hat{d}}\subseteq{\mathbf
\Sigma}^0_2(X,\tau_d)$. Hence, in particular,  $\mathbf{\Sigma}^0_{<\omega}(X,\tau_d)=\mathbf{\Sigma}^0_{<\omega}(X,\tau_{\hat{d}})$, and the identity function $\id_X$ is continuous from  $(X,\tau_{\hat{d}})$ to $(X,\tau_d)$ and is  ${\mathbf
\Sigma}^0_2$-measurable from $(X,\tau_d)$ to $(X,\tau_{\hat{d}})$.
\end{proposition}

This implies that each quasi-Polish space is in fact a standard Borel space, and hence that the Souslin's separation theorem holds in the context of quasi-Polish spaces \cite[Theorem 58]{br}: if \( X \) is quasi-Polish then \( \mathbf{B}(X) = \mathbf{\Delta}^1_1(X) \). Moreover, the Borel hierarchy on uncountable quasi-Polish spaces does not collapse.

\begin{proposition}\cite[Theorem  18]{br}
If  $X$ is a quasi-Polish space, then for all $\alpha <\beta < \omega_1$,
${\bf\Sigma}^0_\alpha(X), \mathbf{\Pi}^0_\alpha(X) \subsetneq {\bf\Delta}^0_\beta(X) \subsetneq \boldsymbol{\Sigma}^0_\beta(X), \boldsymbol{\Pi}^0_\beta(X)$.
\end{proposition}

As for universal quasi-Polish spaces, we have the following result:

\begin{proposition}\cite[Corollary 24]{br}\label{proppi2}
A space is quasi-Polish if and only if it is homeomorphic to a ${\mathbf
\Pi}^0_2$-subset of $P\omega$ (with the relative topology inherited from \( P \omega \)). In particular, $P\omega$ is a universal quasi-Polish space.
\end{proposition}

A quasi-Polish space need not to be \( T_1 \). However one can still prove that the complexity of the singletons is not too high. Recall from e.g.\ \cite{br} that a space \( X \) \emph{satisfies the \( T_D \)-axiom} if  \( \{ x \} \) is the intersection of an open and a closed set for every \( x \in X \).

\begin{proposition}\label{propsingl}
\begin{enumerate}[(1)]
\item
\cite[Proposition 8]{br} If $X$ is a countably based $T_0$-space then $\{ x \}\in{\mathbf
\Pi}^0_2(X$) for any $x\in X$.
\item
\cite[Theorem 65]{br} A countably based space is scattered if and only if it is a countable quasi-Polish space satisfying the \( T_D \)-axiom.
\item
(Folklore) If \( (X,\leq) \) is a dcpo endowed with the Scott topology, then \( \{ c \} \) is the intersection of an open set and a closed set for every compact element \( c \in X_0 \).
\end{enumerate}
\end{proposition}

Of course, a quasi-Polish space need not to have any other special topological property: for example, all nondiscrete \(\omega\)-continuous domain are not Hausdorff and not regular. As for metrizability, we have the following result:

\begin{proposition}\cite[Corollary 42]{br} \label{propmetrizable}
A metrizable space is quasi-Polish if and only if it is Polish.
\end{proposition}

Finally, among the various characterizations of the class of quasi-Polish spaces presented in~\cite{br}, the following one will be of interest for the results of this paper.

\begin{proposition} \cite[Theorem 53]{br} \label{propquasiPolishalgebraic}
A topological space \( X \) is a quasi-Polish space if and only if it is homeomorphic to the set of non-compact elements of some \(\omega\)-algebraic domain.
\end{proposition}

\subsection{Reducibilities}\label{subsectionreducibilities}

In this subsection we introduce and briefly discuss some notions of reducibility
which serve as tools for measuring the topological
complexity of problems (e.g.\ sets, partitions, and so on) in DST and CA.

\begin{definition}
Let \( X  \in \mathscr{X} \) be a topological space. A collection of functions \( \mathcal{F} \) from \( X \) to itself is called \emph{reducibility (on \( X \))} if it contains the identity function \( \id_X \) and is closed under composition.
\end{definition}

Given a reducibility \( \F \) on \( X \), one can consider the preorder \( \leq^X_\F \) on \( \mathscr{P}(X) \) associated to \( \F \) obtained by setting for \(A,B \subseteq X \)
\[
A \leq^X_\F B \iff A = f^{-1}(B) \text{ for some } f \in \F.
 \]
The preorder \( \leq^X_ \F \) canonically induces the equivalence relation
\[
A \equiv^X_\F B \iff {{A \leq^X_\F B} \wedge {B \leq^X_\F A}}.
 \]
Given \(A \subseteq X \), the set \( [A]^X_\F = \{ B \subseteq X \mid A \equiv^X_\F B \} \) is called the \emph{\( \F \)-degree} of \( A \) (in \( X \)).
We denote
the set of \( \F \)-degrees by \( \mathcal{D}^X_\F \).  Notice that
\( \leq^X_\F \) canonically induces on \( \mathcal{D}^X_\F \) the partial order \( [A]^X_\F \leq [B]^X_\F \iff A \leq^X_\F B \). The structures \( (\mathscr{P}(X), \leq^X_\F) \)   or its \( \equiv^X_\F \)-quotient \( (\mathcal{D}^X_\F, \leq ) \) are both called \emph{\( \F \)-hierarchy (on \( X \))} or \emph{hierarchy of the \( \F \)-degrees}. When the space \( X \) is clear from the context, we drop the superscript referring to \( X \) in all the notation above. Notice that if \( \F \) is the collection of all continuous functions, then \( \leq_\F \) coincides with the Wadge reducibility \( \leq_\mathsf{W} \). We will also sometimes consider the restriction of the \( \F \)-hierarchy to some suitable pointclass \( \mathbf{\Gamma}(X) \): in this case, the structure \( (\mathbf{\Gamma}(X), \leq_\F ) \) and its \( \equiv_\F \)-quotient (whenever it is well-defined) will be called \emph{\( (\mathbf{\Gamma}, \F) \)-hierarchy}.

An interesting variant of the reducibility between subsets of \( X \) considered above is obtained by considering \( X \)-namings instead of subsets of \( X \).

\begin{definition}
Let \( X \) be a topological space. An \emph{\( X \)-naming} is a function $\nu$ with domain  $X$.
\end{definition}

There are several
natural reducibility notions for namings, the most basic of which
is the following generalization of \( \leq_\mathsf{W} \).

\begin{definition} \cite{s05, s11a}
An \( X \)-naming $\mu$ is {\em Wadge reducible} to an \( X \)-naming
$\nu$ (in symbols $\mu \leq_{\mathsf{W}} \nu$) if $\mu = \nu\circ f$ for some
continuous function $f \colon X \to X$.

An \( X \)-naming $\mu$ is {\em
Wadge equivalent} to $\nu$ (in symbols $\mu \equiv_\mathsf{W} \nu$), if $\mu \leq_\mathsf{W} \nu$
and $\nu \leq_\mathsf{W} \mu$.
\end{definition}

For any set $S$, one can then consider the preorder $(S^X,\leq_\mathsf{W})$ (or its \( \equiv_\mathsf{W} \)-quotient structure). This gives a  generalization of the preorder formed by the classical Wadge
reducibility on subsets of $X$, because if
$S=\{0,1\}$ then the structures $(\mathscr{P}(X), \leq_\mathsf{W})$ and
$(S^X,\leq_\mathsf{W})$ are isomorphic: $A\leq_\mathsf{W} B$ if and only if $c_A\leq_\mathsf{W}
c_B$, where $c_A \colon X \to 2$ is the characteristic function
of the set $A \subseteq X$. Passing to an arbitrary set \( S \), the Wadge reducibility between \( X \)-namings on \( S \) corresponds to the continuous reducibility between partitions of \( X \) in (at most) \( |S| \)-many pieces. For this reason, when \( S  = \kappa\) is a cardinal number the elements of \( \kappa^X \) are also called \emph{\( \kappa \)-partitions of \( X \)}. Moreover, as for the Wadge reducibility between subsets of \( X \),  we can consider the restriction of \( (S^X,\leq_\mathsf{W}) \) to a pointclass \( \mathbf{\Gamma}(X) \). In particular, when \( S = k \) (for some \( k \in \omega \)) we denote by \( (\mathbf{\Gamma}(X))_k \) the set of \( k \)-partitions \( \nu \in k^X
 \) such that \( \nu^{-1}(i) \in \mathbf{\Gamma}(X) \) for every \( i < k \).

Already for $3\leq k \in \omega$, the structure $(k^\mathcal{N},\leq_\mathsf{W})$ becomes much more complicated than the structure of Wadge degrees on \( \mathcal{N} \), but when restricted to suitable pointclasses it is quite well understood.
In~\cite{ems87} it is shown that the structure
$((\mathbf{\Delta}^1_1(\mathcal{N}))_k,\leq_\mathsf{W})$ is a wqo,
i.e.\ it has neither infinite descending chain nor infinite
antichain. In~\cite{he93,s07a} the \( \equiv_\mathsf{W} \)-quotient structures of, respectively,
$((BC(\mathbf{\Sigma}^0_1)(\mathcal{N}))_k,\leq_\mathsf{W})$ and
$((\mathbf{\Delta}^0_2(\mathcal{N}))_k,\leq_\mathsf{W})$  were
characterized in terms of the relation of homomorphism between
finite and, respectively, countable well-founded $k$-labeled forests.
The mentioned characterizations considerably clarify the corresponding
structures $((BC(\mathbf{\Sigma}^0_1)(\mathcal{N}))_k,\leq_\mathsf{W})$ and
$((\mathbf{\Delta}^0_2(\mathcal{N}))_k,\leq_\mathsf{W})$, and led to deep definability theories for them developed in~\cite{ks07,ks09,ksz09}. In particular, both structures have undecidable first-order theories, and their automorphism groups are isomorphic to the symmetic group on $k$. Similar results are also known for $k$-partitions of $\omega^{\leq\omega}$, see~\cite{s10}.

Of course one can consider other variants on the notion of continuous reducibility between \( X \)-namings. For example,  given a reducibility \( \F \) on \( X \) and a set \( S \), one can consider the \( \F \)-reducibility \( \leq^{S,X}_\F \) between \( X \)-namings defined in the obvious way (as usual, when \( S \) and/or \( X \) are clear from the context we will drop any reference to them in the notation): in this paper we will also provide some results related to this more general notion of reducibility.

\section{Some examples of reducibilities}\label{sectionreduc}

There are various notions of reducibility \( \mathcal{F} \) that have been considered in the literature (see e.g.\ \cite{an06, andmar, ros09, motbaire, motsuperamenable}). In this section we will provide several examples which are relevant for this paper.

Let \( \mathbf{\Gamma},\mathbf{\Delta} \) be two families of pointclasses and \( X,Y \) be arbitrary topological spaces. We denote by \( \mathbf{\Gamma}\mathbf{\Delta}(X,Y) \) (respectively,  \( \mathbf{\Gamma}\mathbf{\Delta}[X,Y]\)) the collection of functions \( f \colon X\to Y \) such that \( f^{-1}(A) \in \mathbf{\Gamma}(X) \) for all \( A\in\mathbf{\Delta}(Y) \) (respectively,  \( f(A)\in\mathbf{\Delta}(Y) \) for all \( A\in\mathbf{\Gamma}(X) \)). Notice that if \( f \colon X \to Y \) is an injection then \( f \in \mathbf{\Gamma} \mathbf{\Delta}(X,Y) \iff f^{-1} \in \mathbf{\Delta} \mathbf{\Gamma}[f(X),X] \). If moreover \( f \colon X \to Y \) is such that \( f(X) \in \mathbf{\Delta} \) and \( \mathbf{\Delta} \) is closed under finite intersections, then \( f \in \mathbf{\Gamma} \mathbf{\Delta}[X,Y] \iff f^{-1} \in \mathbf{\Delta} \mathbf{\Gamma}(f(X),X) \). We abbreviate \( \mathbf{\Gamma}\mathbf{\Gamma}(X,Y) \) with  \( \mathbf{\Gamma}(X,Y) \) and \( \mathbf{\Gamma}\mathbf{\Gamma}[X,Y] \) with \( \mathbf{\Gamma}[X,Y]\). When writing \( \mathbf{\Gamma}\subseteq\mathbf{\Gamma}' \) we mean that \( \mathbf{\Gamma}(X)\subseteq\mathbf{\Gamma}'(X) \) for all topological spaces \( X \).

\begin{remark}\label{remfun}
Let \( \mathbf{\Gamma},\mathbf{\Gamma}',\mathbf{\Delta},\mathbf{\Delta}', \mathbf{\Lambda} \) be families of pointclasses and  \( X,Y,Z \) be arbitrary topological spaces.
\begin{enumerate}[(1)]
 \item
If \( \mathbf{\Gamma} \subseteq \mathbf{\Gamma}' \) and \( \mathbf{\Delta} \subseteq \mathbf{\Delta}' \) then \( \mathbf{\Gamma}\mathbf{\Delta}'(X,Y) \subseteq \mathbf{\Gamma}'\mathbf{\Delta}(X,Y) \)  and \( \mathbf{\Gamma}' \mathbf{\Delta}[X,Y] \subseteq \mathbf{\Gamma}\mathbf{\Delta}'[X,Y]\).

 \item
If \(  f \in \mathbf{\Gamma}\mathbf{\Lambda}(X,Y) \) and \( g \in \mathbf{\Lambda}\mathbf{\Delta}(Y,Z) \) then \( g \circ f \in \mathbf{\Gamma}\mathbf{\Delta}(X,Z) \).

 \item
 If \( f \in \mathbf{\Gamma} \mathbf{\Lambda}[X,Y] \) and \( g \in \mathbf{\Lambda}\mathbf{\Delta}[Y,Z] \) then \( g \circ f \in \mathbf{\Gamma}\mathbf{\Delta}[X,Z] \).
  \item
\( \mathbf{\Gamma}(X,X) \) is closed under composition and contains the identity function (hence is a reducibility on \( X \)).
 \end{enumerate}
 \end{remark}

In this paper, we will often consider the sets of functions given by using the levels of the Borel hierarchy in the above definitions. For ease of notation, when
\( \mathbf{\Gamma} = \mathbf{\Sigma}^0_\alpha \) and
\( \mathbf{\Delta} = \mathbf{\Sigma}^0_\beta \) (for
\( \alpha, \beta < \omega_1 \)) we will write
\( \mathbf{\Sigma}^0_{\alpha,\beta} (X,Y) \) and
\( \mathbf{\Sigma}^0_{\alpha,\beta}[X,Y] \) instead of
\( \mathbf{\Sigma}^0_\alpha \mathbf{\Sigma}^0_\beta(X,Y) \) and
\( \mathbf{\Sigma}^0_\alpha \mathbf{\Sigma}^0_\beta[X,Y] \), respectively. Similarly, we
will write \( \mathbf{\Sigma}^0_{< \alpha,\beta}(X,Y) \) instead of
\( \mathbf{\Sigma}^0_{<\alpha} \mathbf{\Sigma}^0_\beta(X,Y) \).
Moreover, we will often denote the class of continuous functions
\( \mathbf{\Sigma}^0_{1}(X,Y) \) by \( \mathsf{W}(X,Y) \), and write \( \mathsf{W}(X) \) (or even
just \( \mathsf{W} \), if \( X \) is clear from the context) instead of
\( \mathsf{W}(X,X) \). (The
symbol \( \mathsf{W} \) stands for W.\ Wadge, who was the first to initiate in~\cite{wad84} a systematic study of the quasi-order \( \leq_{\mathsf{W}} \)
on \( \mathcal{N} \).)

Some of these classes of functions are well-known in DST. For example,
\( \mathbf{\Sigma}^0_\alpha(X,Y) \) coincides (for \( \alpha \geq 1 \) and
\( X = Y = \mathcal{N} \)) with the class of functions \( \mathsf{D}_\alpha \)
considered in~\cite{ros09}: for this reason, the class
\( \mathbf{\Sigma}^0_\alpha(X,Y) \) will be often denoted by
\( \mathsf{D}_\alpha(X,Y) \). When \( X=Y \) we will simplify a little bit the
notation by setting \( \mathsf{D}_\alpha(X) = \mathsf{D}_\alpha(X,X) \), and
even drop the reference to \(X \) when such space is clear from the context. Notice that the classes \( \mathsf{D}_\alpha(X) \)
are always reducibilities by Remark~\ref{remfun}(4). It follows immediately from the above definitions that $\mathsf{D}_\alpha(X,Y)=\mathbf{\Pi}^0_\alpha(X,Y)$  for $1 \leq \alpha<\omega_1$, and that if \( \alpha \geq 2 \) then \( \mathsf{D}_\alpha(X,Y) =\mathbf{\Delta}^0_\alpha(X,Y) \); when $Y$ is zero-dimensional, then we also have that $\mathsf{D}_1(X,Y)=\mathbf{\mathbf{\Delta}}^0_1(X,Y)$.

Another well-known class is \( \mathbf{\Sigma}^0_{\alpha,1}(X,Y) \),
the collection of all $\mathbf{\Sigma}^0_\alpha$-measurable
functions from $X$ to $Y$. Recall that by \cite[Theorem 24.3]{ke94}
if \( \alpha = \beta +1 \) and \( X,Y \) are metrizable with \( Y \)
separable, the class \( \mathbf{\Sigma}^0_{\alpha,1} \) coincides
with the class of Baire class \(\beta\) functions (as defined e.g.\
in \cite[Definition 24.1]{ke94}). The classes \(
\mathbf{\Sigma}^0_{\alpha,1}(X,X) \) are not closed under
composition if \( \alpha > 1 \): as computed in \cite[Theorem
6.4]{motbaire}, the closure under composition of \(
\mathbf{\Sigma}^0_{\alpha,1}(X,X) \) is given by \(
\mathcal{B}_{\gamma}(X) = \bigcup_{\beta < \gamma}
\mathbf{\Sigma}^0_{\beta,1}(X,X) \), where \( \gamma = \alpha \cdot
\omega \) is the first additively closed ordinal above \( \alpha \)
(as usual, we will drop the reference to \( X \) whenever such space
will be clear from the context). Hence, when \( \gamma \) is
additively closed the set \( \mathcal{B}_{\gamma}(X) \) is a
reducibility on \( X \). The reducibilities \( \mathcal{B}_\gamma(X)
\) and their induced degree-structures have been studied in~\cite{motbaire}. Notice also that in general \(
\bigcup_{\alpha<\gamma} \mathsf{D}_\alpha(X) \subsetneq
\mathcal{B}_\gamma(X) \subsetneq \mathbf{\Sigma}^0_{< \gamma,1}(X,X)
\)

We now state some properties of these classes of functions.

\begin{proposition}\label{propsigmafun}
\begin{enumerate}[(1)]
 \item If $1\leq\alpha<\beta<\omega_1$ then $\mathsf{D}_\alpha(X,Y)\subseteq \mathsf{D}_\beta(X,Y)$, and if \(\beta\) is limit \( \mathsf{D}_\alpha(X,Y) \subseteq \mathbf{\Sigma}^0_{< \beta}(X,Y) \subseteq \mathsf{D}_\beta(X,Y) \).

  \item
Let \( 1 \leq \alpha,\beta < \omega_1 \) and \( \delta = \max \{ \alpha,\beta \} \cdot \omega \) (i.e.\ \(\delta\) is the first additively closed ordinal strictly above \(\alpha\) and \(\beta\)). Then $\mathbf{\Sigma}^0_{\alpha,\beta}(X,Y) \subseteq \mathbf{\Sigma}^0_{< \delta}(X,Y)$. In particular, \( \bigcup_{\alpha< \gamma} \mathbf{\Sigma}^0_{\alpha,1} (X,Y) \subseteq \mathbf{\Sigma}^0_{< \gamma}(X,Y) \subseteq \mathsf{D}_\gamma(X,Y) \) for every additively closed \( \gamma < \omega_1 \).
 \end{enumerate}
 \end{proposition}

\begin{proof}
The proof of  (1) is straightforward, so we just consider (2). If
$\alpha\leq \beta$ then
$\mathbf{\Sigma}^0_{\alpha,\beta}(X,Y)\subseteq\mathbf{\Sigma}^0_\beta(X,Y)\subseteq
\mathbf{\Sigma}^0_{<\delta}(X,Y)$ by (1), hence we can assume \(
\beta<\alpha \). Arguing by induction on \( \gamma < \omega_1 \),
one easily obtains  \( \mathbf{\Sigma}^0_{\alpha,\beta}(X,Y)
\subseteq \mathbf{\Sigma}^0_{\alpha+\gamma,\beta+\gamma}(X,Y) \).
Let \( \alpha' \leq \alpha \) be such that \( \beta + \alpha' =
\alpha \),  and let \( \gamma = \alpha + (\alpha' \cdot \omega) =
\beta + (\alpha' \cdot \omega)  \leq \alpha \cdot \omega =
\delta\). We claim that \( \mathbf{\Sigma}^0_{\alpha,\beta}
\subseteq \mathbf{\Sigma}^0_{< \gamma}(X,Y) \), which obviously
implies the desired result. Let \( f \in
\mathbf{\Sigma}^0_{\alpha,\beta}(X,Y) \) and \( A \in
\mathbf{\Sigma}^0_{<\gamma}(Y) \). Then for some \( k<\omega \) we
have \( A \in \mathbf{\Sigma}^0_{\beta+(\alpha' \cdot k)}(Y) \),
hence \( f^{-1}(A) \in \mathbf{\Sigma}^0_{\alpha+(\alpha' \cdot
k)}(X) \subseteq \mathbf{\Sigma}^0_{<\gamma}(X) \), as required.
 \end{proof}

\noindent
In particular, by Proposition~\ref{propsigmafun}(2) all Baire class \( 1 \) functions (i.e.\ the functions in \( \mathbf{\Sigma}^0_{2,1}(X,Y)  \)) are in \( \mathbf{\Sigma}^0_{< \omega}(X,Y) \), and hence also in \( \mathsf{D}_\omega(X,Y) \).

Other kind of reducibilities which will be considered in this paper are given by classes of piecewise defined functions.
Given a family of pointclasses \( \mathbf{\Gamma} \), a \emph{\( \mathbf{\Gamma} \)-partition} of a space \( X \) is a sequence
 \( \langle D_n \mid n \in \omega \rangle \) of pairwise disjoint sets
 from \( \mathbf{\Gamma}(X) \) such that \( X = \bigcup_{n \in \omega} D_n \). Notice that, in particular, every \( \mathbf{\Sigma}^0_\alpha \)-partition
 of \( X \) is automatically a \( \mathbf{\Delta}^0_\alpha \)-partition of \( X \) (for every  \( \alpha < \omega_1 \)).

\begin{definition} \cite{motgames}
Given two spaces \( X, Y  \in \mathscr{X}\), a collection of functions  \( \F \) from (subsets of) \( X \) to \( Y \), and an ordinal \( \alpha < \omega_1 \), we will denote by
\( \mathsf{D}_\alpha^{\F}(X,Y) \) \footnote{Notice that for \( X = Y = \mathcal{N} \), this class of functions was denoted by \( \tilde{\mathsf{D}}^\F_\alpha \) in~\cite{motgames}. However, here we will not use the other class of piecewise defined functions considered in that paper, so we can safely simplify the notation dropping the decoration on the symbol \( \mathsf{D} \).} the collection
of those \( f
\colon X \to Y\) for which there is a \( \mathbf{\Sigma}^0_\alpha \)-partition
(equivalently,
a \( \mathbf{\Delta}^0_\alpha \)-partition)
\( \langle D_n \mid n \in \omega \rangle \) of \( X \) and
a sequence \( \langle f_n \colon D_n \to Y \mid n \in \omega \rangle \) of functions from \( \F \) such
that  \( f = \bigcup_{n \in \omega} f_n \).
\end{definition}

In particular, we will be interested in the classes
\( \mathsf{D}^\mathsf{W}_\alpha(X,Y) =
\mathsf{D}^{\mathsf{W}(\subseteq X,Y)}_\alpha(X,Y) \), where \( \mathsf{W}( \subseteq X, Y) = \bigcup_{X' \subseteq X} \mathsf{W}(X',Y) \),  for various
\( \alpha < \omega_1 \).  Notice also that
\( \mathbf{\Sigma}^0_\alpha(X,Y) = \mathsf{D}^{\mathbf{\Sigma}^0_\alpha(\subseteq X,Y)}_\alpha(X,Y) \) with \( \mathbf{\Sigma}^0_\alpha(\subseteq X,Y) = \bigcup_{X' \subseteq X} \boldsymbol{\Sigma}^0_\alpha(X',Y) \),  and that
\( \mathsf{D}^\mathsf{W}_\alpha(X,Y)  \subseteq
\mathsf{D}_\alpha(X,Y) \). As for the other classes of functions, we will write
\( \mathsf{D}^\mathsf{W}_\alpha(X) \) instead of
\( \mathsf{D}^\mathsf{W}_\alpha(X,X) \), and even drop the reference to
\( X \) when there is no danger of confusion. It is not hard to check that each of the
\( \mathsf{D}^\mathsf{W}_\alpha(X) \) is a reducibility on \( X \).

It is a remarkable theorem of Jayne and Rogers \cite[Theorem 5]{JR82} (but
see also~\cite{motsem,kacmotsem} for a shorter and simpler proof) that if
\( X,Y \) are Polish spaces\footnote{In fact the Jayne-Rogers result is even
more general, in that its conclusion holds also when \( X, Y \) are arbitrary metric spaces
with \( X \) an absolute Souslin-\( \mathscr{F} \) set.} then \( \mathsf{D}_2 (X,Y) =
\mathsf{D}^\mathsf{W}_2(X,Y) \). This result has been recently extended to the
level \( 3 \) (for the special case \( X  = Y = \mathcal{N} \)) by B. Semmes.

\begin{theorem} \cite{sem}  \label{theorsem}
\( \mathsf{D}_3(\mathcal{N}) = \mathsf{D}^\mathsf{W}_3 ( \mathcal{N}) \).
\end{theorem}

Whether this result can be extended to all \( 3 < n < \omega \) is a major open problem, but notice however that, as observed e.g.\ in~\cite{andslo,ros09}, it is not possible to generalize the Jayne-Rogers theorem to levels \( \alpha \geq \omega \). To see this we need to recall the definition of containment between functions introduced in~\cite{sol}, and the definition of a very special function, called the Pawlikowski function. We call embedding any function between two topological spaces which is an homeomorphism on its range.

\begin{definition} \cite{sol} \label{defcontainment}
Let \( X_0,X_1,Y_0,Y_1 \in \mathscr{X} \) be topological spaces and consider \( f \colon X_0 \to Y_0 \) and \( g \colon X_1 \to Y_1 \). We say that \( f \) is \emph{contained} in \( g \), and we write \( f \sqsubseteq g \), just in case there are two embeddings \( \varphi \colon
X_0 \to X_1\) and \( \psi \colon f(X_0) \to Y_1\)  such that \( \psi
\circ f = g \circ \varphi \).
\end{definition}

It is not hard to check that if \( f \) and \( g \) are as in Definition~\ref{defcontainment} and \( f \sqsubseteq g \), then \( g \in \mathbf{\Sigma}^0_{\alpha,\beta}(X_1,Y_1) \) implies \( f \in \mathbf{\Sigma}^0_{\alpha,\beta}(X_0,Y_0) \) for every \( 1 \leq \alpha, \beta < \omega_1 \).

Let us endow the space \( (\omega + 1)^\omega \) with the product of the order topology on \(  \omega + 1 \). The \emph{Pawlikowski function} is the function \( P \colon (\omega+1)^\omega \to \mathcal{N} \) defined by
\[
 P(x)(n) =
\begin{cases}
x(n)+1 &  \text{if }x(n) \in \omega \\
0 & \text{if }x(n) = \omega.
\end{cases}
\]
Since \( (\omega+1)^\omega \) is a perfect nonempty compact metrizable zero-dimensional space, it is homeomorphic to the Cantor space \( \mathcal{C} \) by Brouwer's theorem \cite[Theorem 7.4]{ke94}: therefore, using the fact that \( \mathcal{N} \) is homeomorphic to a \( G_\delta \) subset of \( \mathcal{C} \),  \( P \) can actually be regarded as a function from \( \mathcal{C} \) to \( \mathcal{C} \) .

\begin{definition}
Let \( X,Y \) be arbitrary Polish spaces. We say that a function \( f \colon X \to Y \) \emph{can be decomposed} into countably many continuous functions (briefly, it is \emph{decomposable}) if there is some partition \( \langle D_n \mid n \in \omega \rangle \)  of \(X \) into countably many pieces (of arbitrary complexity) such that \( f \restriction D_n \) is continuous for every \( n < \omega \).
\end{definition}

\noindent
Notice in particular that all the functions in \( \bigcup_{\alpha<\omega_1} \mathsf{D}^\mathsf{W}_\alpha(X,Y) \) are decomposable by definition, and that if \( f \sqsubseteq g \) and \( g \) is decomposable, then \( f \) is decomposable as well.
The function \( P \) was introduced as an example of a Baire class \( 1 \) function (hence \( P \in \mathbf{\Sigma}^0_{2,1}(\mathcal{C}, \mathcal{C}) \)) which is not decomposable --- see e.g.\ \cite{sol}. Using this fact, we can now prove that the Jayne-Rogers theorem cannot be generalized to the infinite levels of the Borel hierarchy for several uncountable quasi-Polish spaces.

\begin{proposition}
Assume that \( X,Y \) are uncountable quasi-Polish spaces  and \( \omega \leq \alpha < \omega_1 \). If \( \mathcal{C} \) can be embedded into \( Y \), then \( \mathsf{D}_\alpha^\mathsf{W}(X,Y) \subsetneq \mathsf{D}_\alpha(X,Y) \). In particular, \( \mathsf{D}_\alpha^\mathsf{W}(X,Y) \subsetneq \mathsf{D}_\alpha(X,Y) \) for \( Y \) a Polish space, \( Y = \omega^{\leq \omega} \), or \( Y = P \omega \).
\end{proposition}

\begin{proof}
First observe that \( P \in \mathbf{\Sigma}^0_{< \omega}(\mathcal{C},\mathcal{C}) \subseteq \mathsf{D}_\alpha(\mathcal{C}) \) by Proposition~\ref{propsigmafun}, but \( P \notin \bigcup_{\alpha < \omega_1} \mathsf{D}^\mathsf{W}_\alpha(\mathcal{C}) \) because \( P \) is not even decomposable.
Suppose now that \( X,Y \) are arbitrary uncountable quasi-Polish spaces, and further assume that \( \mathcal{C} \) embeds into \( Y \).

\begin{claim}\label{claiminjection}
Let \( Z \) be an uncountable quasi-Polish space. Then there is a
continuous injection of \( \mathcal{C} \) onto a \(
\boldsymbol{\Pi}^0_2 \) subset of \( Z \).
\end{claim}

\begin{proof}[Proof of the Claim.]
Let \( d \) be  a complete quasi-metric on \( Z \) compatible with
its topology. By the Cantor-Bendixon theorem \cite[Theorem
6.4]{ke94}, there is an embedding \( \varphi \colon \mathcal{C} \to
(Z, \tau_{\hat{d}}) \). Notice that the range of \( \varphi \) is
automatically \( \tau_{\hat{d}} \)-closed. Therefore by Proposition~\ref{propqp} we have that \( \varphi \colon \mathcal{C} \to (Z,
\tau_d) \) is continuous as well, and that its range is a \(
\boldsymbol{\Pi}^0_2 \) set with respect to \( \tau_d \).
\end{proof}

Let \( \varphi \colon \mathcal{C} \to X \) be a continuous
injection with \( \varphi(\mathcal{C}) \in \boldsymbol{\Pi}^0_2(X)
\),  and let \( \psi \colon \mathcal{C} \to Y \) be an embedding.
 Pick an arbitrary \( y_0 \in Y \) and define \( P' \colon X \to Y \)
by setting \( P'(x) = \psi(P(\varphi^{-1}(x))) \) if \( x \in
\varphi(\mathcal{C}) \) and \( P'(x) = y_0 \) otherwise. It is
straightforward to check that \( P' \in
\mathbf{\Sigma}^0_{3,1}(X,Y) \subseteq   \mathbf{\Sigma}^0_{<
\omega}(X,Y) \subseteq \mathsf{D}_\alpha(X,Y) \) by Proposition~\ref{propsigmafun}. We will now show that \( P' \colon X \to Y \)
is not decomposable, which clearly implies that \( P' \notin
\mathsf{D}^\mathsf{W}_\alpha(X,Y) \), as required. Assume towards
a contradiction that \( P' \) is decomposable,  and let \( \langle
X_n \mid n \in \omega \rangle \) be a countable partition of \( X
\) such that \( P'_n = P' \restriction X_n \) is continuous for
every \( n \in \omega \). Then \( \langle \varphi^{-1}(X_n) \mid n
\in \omega  \rangle \) is a countable partition of \( \mathcal{C}
\): we will show that for every \( n \in \omega \) the function \(
P_n =  P \restriction \varphi^{-1}(X_n) \) is continuous,
contradicting the fact that \( P \) is not decomposable. Let \( U
\subseteq \mathcal{C} \) be open. Since \( \psi \) is an
embedding, there is an open \( V \subseteq Y \) such that \(
\psi(U) = V \cap \psi(\mathcal{C}) \). By definition of \( P' \),
\( P_n^{-1}(U) = \varphi^{-1}((P'_n)^{-1}(V)) \), which is open in
\( \varphi^{-1}(X_n) \) because \( \varphi \) and \( P'_n \) are
both continuous: therefore \( P_n \) is continuous.
\end{proof}

\begin{corollary}
Assume that \( X,Y \) are uncountable quasi-Polish spaces  and \( \omega \leq \alpha < \omega_1 \). If \( Y \) is Hausdorff, then \( \mathsf{D}_\alpha^\mathsf{W}(X,Y) \subsetneq \mathsf{D}_\alpha(X,Y) \).
\end{corollary}

\begin{proof}
It is enough to show that if  \( Y \) is an uncountable Hausdorff quasi-Polish space then there is an embedding of \( \mathcal{C} \) into \( Y \). Let \( \varphi \colon \mathcal{C} \to Y \) be a continuous injection (which exists by Claim~\ref{claiminjection}). We want to show that for every open \( U \subseteq \mathcal{C} \), \( \varphi(U) \) is open in \( \varphi(\mathcal{C}) \). Since \( \mathcal{C} \) is compact, \( \mathcal{C} \setminus U \) is compact as well. Since \( \varphi \) is continuous, then \( D = \varphi(\mathcal{C} \setminus U) \) is compact as well, and hence also closed in \( Y \). But then \( \varphi(U) = (Y \setminus D) \cap \varphi(\mathcal{C}) \) is an open set with respect to the relative topology of \( \varphi(\mathcal{C}) \), as required.
\end{proof}

The Pawlikowski function \( P \) can in fact be used to characterize decomposable  functions within certain Borel classes. In~\cite{sol}, Solecki  proved that if \( f \in \mathbf{\Sigma}^0_{2,1}(X,Y) \) with \( X,Y \) Polish spaces\footnote{Solecki's theorem applies to a slightly wider context, i.e.\ to the case when \( X \) is an analytic  space and \( Y \) is separable metric.}, then \( f \) is decomposable if and only if \( P \not\sqsubseteq f \). Using
the technique of changes of topologies and arguing by induction on
\( n < \omega \), this characterization can easily be
extended\footnote{In~\cite[Theorem 1.1]{pawsab},
Solecki's characterization
of decomposable functions is further extended (using different and more
involved methods) to the even wider context of all Borel functions from an
analytic space \( X \) to a separable metrizable space \( Y \), but here we will
not need the above characterization in such generality.} to the wider context of
functions in \( \bigcup_{1 \leq n < \omega} \boldsymbol{\Sigma}^0_{n,1}(X,Y) \).

\begin{theorem} \cite[Lemma 5.7]{motdec} \label{theornondecomposable}
Let \( X,Y \) be Polish spaces and let \( f \) be in
\(\boldsymbol{\Sigma}^0_{n,1}(X,Y) \) for some \( 1 \leq n <
\omega \). Then \( f \) is decomposable if and only if \( P
\not\sqsubseteq f \).
\end{theorem}

Finally, we recall from \cite[Proposition 6.6]{motbaire} the following result on the topological complexity of \( P \).

\begin{proposition} \label{propPisnotinD_n}
For every \( n \in \omega \),  \( P \notin \mathsf{D}_n(\mathcal{C}) \).
\end{proposition}

Combining all above results together, we have the following
proposition (see also Lemma 5.8 in~\cite{motdec}).

\begin{proposition}\label{propDndec}
Let \( X,Y \) be Polish spaces. For every \( n < \omega \) and \( f \in \mathsf{D}_n(X,Y) \), \( f \) is decomposable.
\end{proposition}

\begin{proof}
By Proposition~\ref{propPisnotinD_n} and the observation following Definition~\ref{defcontainment}, we have that \( P \not\sqsubseteq f \), hence \( f \) is decomposable by Theorem~\ref{theornondecomposable}.
\end{proof}

\section{Isomorphisms of minimal complexity between quasi-Polish spaces}\label{sectionisom}

The following definition extends in various directions the topological notion of homeomorphism.

\begin{definition}\label{defhomeovariants}
\begin{enumerate}[(1)]
\item
Let \( \F \) be a collection of functions between topological spaces, and \( X,Y \in \mathscr{X} \). We say that \( X \) and \( Y \) are \emph{\( \F \)-isomorphic} (\( X \simeq_\F Y \) in symbols) if there is a bijection \( f \colon X \to Y \) such that both \( f \) and \( f^{-1} \) belong to \( \F \).

\item
If \( \mathbf{\Gamma} \) is a family of pointclasses, we say that two topological spaces \( X , Y \) are \emph{\( \mathbf{\Gamma} \)-isomorphic} (\( X \simeq_{\mathbf{\Gamma}} Y  \) in symbols) if \( X \simeq_\F Y \) where \( \F = \bigcup \{ \mathbf{\Gamma}(X,Y) \mid X,Y \in \mathscr{X} \} \). This is obviously equivalent to requiring that \( X \simeq_\G Y \) where \( \G = \bigcup \{ \mathbf{\Gamma}[X,Y] \mid X,Y \in \mathscr{X} \} \).

\item
We say that \( X,Y \in \mathscr{X} \) are \emph{piecewise homeomorphic} (\( X \simeq_{\mathsf{pw}} Y \) in symbols) if there are countable partitions \( \langle X_n \mid n \in \omega \rangle , \langle Y_n \mid n \in \omega \rangle \) of, respectively, \( X \) and \( Y \) such that \( X_n \) and \( Y_n \) are homeomorphic for every \( n \in \omega \).

\item
Given a family of pointclasses \( \mathbf{\Gamma} \), we say that \( X \) and \( Y \) are \emph{\( \mathbf{\Gamma} \)-piecewise homeomorphic} (\( X \simeq_{\mathsf{pw}(\mathbf{\Gamma})} Y \) in symbols) if and only if there are partitions \( \langle X_n \mid n \in \omega \rangle , \langle Y_n \mid n \in \omega \rangle \) of, respectively, \( X \) and \( Y \) consisting of sets in \( \mathbf{\Gamma} \) such that \( X_n \) and \( Y_n \) are homeomorphic for every \( n \in \omega \).
\end{enumerate}
\end{definition}

It is obvious that if \( \F \subseteq \G \) are two sets of functions between topological spaces and \( X,Y \in \mathscr{X} \) are such that \( X \simeq_\F Y \) then \( X \simeq_\G Y \). Since the notion of \( \mathbf{\Sigma}^0_\alpha \)-isomorphism, \( \mathbf{\Pi}^0_\alpha \)-isomorphism, and \( \mathbf{\Delta}^0_\alpha \)-isomorphism all coincide for \( 2 \leq \alpha < \omega_1 \), for simplicity of notation we will write \( X \simeq_\alpha Y \) instead of \( X \simeq_{\mathbf{\Sigma}^0_\alpha} Y \). Similarly, when \( \F = \bigcup   \{ \mathsf{D}^\mathsf{W}_\alpha(X,Y) \mid X,Y \in \mathscr{X} \} \) we will simply write \( X \simeq^\mathsf{W}_\alpha Y \) instead of \( X \simeq_\F Y \).

\begin{lemma}\label{lemmapwhomeo}
Let \( 1 \leq \alpha < \omega_1 \) and \( X,Y \in \mathscr{X} \). Then \( X \simeq^\mathsf{W}_\alpha Y \) if and only if \( X \simeq_{\mathsf{pw}(\mathbf{\Sigma}^0_\alpha)} Y \) (equivalently, if and only if  \( X \simeq_{\mathsf{pw}(\mathbf{\Delta}^0_\alpha)} Y \)).

Similarly, let \( \F \) be the class of all decomposable functions. Then \( X \simeq_\F Y \) if and only if \( X \simeq_{\mathsf{pw}} Y \).
\end{lemma}

\begin{proof}
We just consider the first part of the lemma,
as the second one can be proved in a similar way.
The direction from right to left directly follows from the definition of
\( \mathsf{D}^\mathsf{W}_\alpha(X,Y) \), so let us assume that
\( f \colon X \to Y \) is a bijection such that
\( f \in \mathsf{D}^\mathsf{W}_\alpha(X,Y) \) and
\( f^{-1} \in \mathsf{D}^\mathsf{W}_\alpha(Y,X) \). By definition, there are
partitions \( \langle X'_n  \mid n \in \omega \rangle \) and
\( \langle Y'_m \mid m \in \omega \rangle \) of, respectively, \( X \) and \( Y \)
in \( \mathbf{\Sigma}^0_\alpha \) pieces such that \( f \restriction X'_n \) and
\( f^{-1} \restriction Y'_m \) are continuous for every \( n,m \leq \omega \).
This implies that for every \( n,m < \omega \) the sets
\( X_{\langle n ,m \rangle} = X'_n \cap f^{-1}(Y'_m) \) and
\( Y_{\langle m,n \rangle} = Y'_m \cap f(X'_n) \) (where
\( \langle \cdot, \cdot \rangle \) denotes a bijection between
\( \omega \times \omega \) and \( \omega \)) are in
\( \mathbf{\Sigma}^0_\alpha \) as well and form two countable partitions of,
respectively \( X \) and \( Y \). Then it is easy to see that
\( f \restriction X_{\langle n,m \rangle} \colon X_{\langle n,m \rangle} \to Y_{\langle m,n \rangle} \) is a bijection witnessing that
\( X_{\langle n,m \rangle} \) and \( Y_{\langle m,n \rangle} \) are
homeomorphic, hence we are done.
\end{proof}

It is a classical result of DST that  every two uncountable Polish
spaces \( X , Y \) are \( \mathbf{B} \)-isomorphic (see e.g.\
\cite[Theorem 15.6]{ke94}). The next proposition extends this
result to the context of uncountable quasi-Polish spaces and
computes an upper bound for the complexity of the
Borel-isomorphism according to Definition~\ref{defhomeovariants}.

\begin{proposition}\label{propgeneralhomeo}
Let \( \F = \bigcup \{ \mathbf{\Sigma}^0_{3,1}(X,Y) \mid X,Y \in \mathscr{X} \} \) and let \( X,Y \) be two uncountable quasi-Polish spaces. Then \( X \cong_\F Y \). In particular, \( X \cong_{\mathbf{\Delta}^0_{< \omega}} Y \) and hence also \( X \cong_\omega Y \).
\end{proposition}

\begin{proof}
 Let \( d_X \) and \( d_Y \) be complete quasi-metrics compatible with the topologies of, respectively, \( X \) and \( Y \), and let \( \hat{d}_X, \hat{d}_Y \) be the metrics induced by \( d_X \) and \( d_Y \). Then by Proposition~\ref{propqp} \( \id_X \colon (X \tau_{\hat{d}_X}) \to (X, \tau_{d_X}) \) is a continuous function with \( \mathbf{\Sigma}^0_2 \)-measurable inverse, and similarly for \( \id_Y \colon (Y, \tau_{\hat{d}_Y}) \to (Y, \tau_{d_Y}) \). Since \( (X , \tau_{\hat{d}_X} ) \) and \( (Y , \tau_{\hat{d}_Y} ) \) are uncountable Polish spaces, by e.g.\ \cite[p.\ 212]{kur} there is a bijection \( g \colon (X,\tau_{\hat{d}_X}) \to (Y , \tau_{\hat{d}_Y} ) \) such that both \( g \) and \( g^{-1} \) are \( \mathbf{\Sigma}^0_2 \)-measurable. Hence \( f = \id_Y \circ g \circ \id^{-1}_X \) is a bijection between \( (X, \tau_{d_X}) \) and \( (Y , \tau_{d_Y} ) \) such that both \( f \) and \( f^{-1} \) are \( \mathbf{\Sigma}^0_3 \)-measurable.

The second part of the Proposition follows from Proposition~\ref{propsigmafun}(2).
\end{proof}

\begin{proposition} \label{propisoquasiPolishalgebraic}
Every quasi-Polish space is \( \mathsf{D}^\mathsf{W}_4 \)-isomorphic to an \(\omega\)-algebraic domain.
\end{proposition}

\begin{proof}
Let \( Y \) be a quasi-Polish space. We can clearly assume that \( Y
\) is infinite (otherwise \( Y \) itself is an \(\omega\)-algebraic
domain). By Proposition~\ref{propquasiPolishalgebraic}, there is an
\(\omega\)-algebraic domain $X$ and a function \( f \colon Y \to X
\) such that \( f \) is an homeomorphism between \( Y \) and \( X
\setminus X_0 \), where \( X_0 \) is the (countable) set of compact
elements of \( X \) (see Subsection~\ref{omega-cont}). By Lemma~\ref{lemmapwhomeo}, it is enough to show that \( Y
\cong_{\mathsf{pw}(\mathbf{\Delta}^0_4)} X \).
 Let \( \langle x_n \mid n \in \omega \rangle \) be an enumeration without
repetitions of \( X_0 \) and \( \langle y_n \mid n \in \omega \rangle \) be an
enumeration without repetitions of an infinite countable subset \( Y_0 \) of
\( Y \) such that $Y\setminus Y_0$ is non-empty. Then \( \langle Y \setminus Y_0, \{ y_n \} \mid n \in \omega \rangle \)
and \( \langle X \setminus (X_0 \cup f(Y_0)), \{ f(y_n) \},
\{ x_n \} \mid n \in \omega \rangle \) are countable partitions of, respectively,
\( Y \) and \( X \) into \( \mathbf{\Pi}^0_3 \) pieces by Proposition~\ref{propsingl}(1). The function
\( f \restriction (Y \setminus Y_0) \) is an homeomorphism between
\( Y \setminus Y_0 \) and \( X \setminus (X_0 \cup f(Y_0)) \). For every \( n \in \omega \), the function sending \( y_{2n} \) to \( f(y_n) \) is an homeomorphism between
\( \{ y_{2n} \} \) and \( \{ f(y_n) \} \), while the function sending
\( y_{2n+1} \) to \( x_n \) is an homeomorphism between
\( \{ y_{2n+1} \} \) and \( \{ x_n \} \). Hence
\( Y \cong_{\mathsf{pw}(\mathbf{\Delta}^0_4)} X \), as required.
\end{proof}

\begin{proposition}\label{propgeneralhomeoctbl}
\begin{enumerate}[(1)]
 \item
Let \( X,Y \) be countable countably based  \( T_0 \)-spaces. Then \( X \simeq^\mathsf{W}_3 Y \) if and only if \( |X|=|Y| \).
 \item
Let \( X,Y \) be countable \( T_1 \) spaces. Then  \( X \simeq^\mathsf{W}_2 Y \) if and only if \( |X|=|Y| \).
\item
Let \( X,Y \) be scattered countably based spaces. Then  \( X \simeq^\mathsf{W}_2 Y \) if and only if \( |X|=|Y| \).
\end{enumerate}

In particular, \( X \simeq^\mathsf{W}_3 Y \) (respectively, \( X \simeq^\mathsf{W}_2 Y \)) for \(X,Y\) countable quasi-Polish (respectively, Polish) spaces of the same cardinality.
\end{proposition}

\begin{proof*} (1) For the nontrivial direction, notice that by Proposition~\ref{propsingl}(1) any bijection \( f \colon X \to Y \) is a witness of \( X \simeq^\mathsf{W}_3 Y \).

(2) It is a classical fact that  a space is \( T_1 \) if and only if  its singletons are closed: hence any bijection between \( X \) and \( Y \) witnesses \( X \simeq^\mathsf{W}_2 Y \).

(3) By Proposition~\ref{propsingl}(2), \( \{ x \} \) is the intersection of an open set and a closed set for every \( x \in X \), and similarly for every \( y \in Y \): hence any bijection between \( X \) and \( Y \) witnesses \( X \simeq^\mathsf{W}_2 Y \). \hfill \usebox{\proofbox}
\end{proof*}

Of course the general results above (Propositions~\ref{propgeneralhomeo},\ref{propisoquasiPolishalgebraic} and~\ref{propgeneralhomeoctbl}) do not give in general an optimal bound (in the sense of Definition~\ref{defhomeovariants}) on the minimal complexity of an isomorphism between two specific quasi-Polish spaces \( X \) and \( Y \). In the next proposition we collect some easy observations concerning the possible complexity of isomorphism between concrete examples of quasi-Polish spaces, including the following:
\begin{enumerate}[(1)]
\item
\( \omega \) endowed with the discrete topology;
\item
the space \( \RR^n \) (\( n \in \omega \)) endowed with the product of the order topology on \( \RR \);
\item
the \(\omega\)-algebraic domain \( (\omega^{\leq \omega}, \sqsubseteq ) \) endowed with the Scott topology.
\end{enumerate}

\begin{remark}\label{remDW2}
It is straightforward to check that if \( f \in \mathsf{D}^\mathsf{W}_2(X,Y) \) with \( X \) \(\sigma\)-compact and \( Y \) an Hausdorff space, then the range of \( f \) is \(\sigma\)-compact as well. In particular, if \( X,Y \) are Polish spaces with \( X \) \(\sigma\)-compact and \( Y \) non \(\sigma\)-compact, then there is no onto \( f \in \mathsf{D}^\mathsf{W}_2(X,Y) \).
\end{remark}

\begin{proposition}\label{propexamples}
\begin{enumerate}[(1)]
 \item \( \mathcal{N}  \simeq^\mathsf{W}_2  \omega \sqcup
\mathcal{N}  \);
 \item if \( X \) is a \(\sigma\)-compact
quasi-Polish space then \( \mathcal{N} \not\simeq^\mathsf{W}_2 X
\). In particular, \( \mathcal{N} \not\simeq^\mathsf{W}_2
\mathcal{C} \), \( \mathcal{N} \not\simeq^\mathsf{W}_2 \RR^n \)
for every \( n < \omega \), and \( \mathcal{N}
\not\simeq^\mathsf{W}_2 \omega^{\leq \omega} \);
 \item \(
\mathcal{N} \simeq^{\mathsf{W}}_3 \mathcal{C} \). More precisely,
there is a bijection  \( f \colon \mathcal{N} \to \mathcal{C} \)
such that \( f \in \mathsf{D}^\mathsf{W}_2(\mathcal{N},
\mathcal{C}) \) and \( f^{-1} \in
\mathsf{D}^\mathsf{W}_3(\mathcal{C}, \mathcal{N}) \);
 \item \(
\mathcal{N} \simeq^\mathsf{W}_3 \RR^n \) for every \(1 \leq  n <
\omega \). More precisely, there is a bijection  \( f \colon
\mathcal{N} \to \RR^n \) such that \( f \in
\mathsf{D}^\mathsf{W}_2(\mathcal{N}, \RR^n) \) and \( f^{-1} \in
\mathsf{D}^\mathsf{W}_3(\RR^n , \mathcal{N}) \);
 \item \(
\mathcal{N} \simeq^\mathsf{W}_3 \omega^{\leq \omega} \). More
precisely, there is a bijection  \( f \colon \mathcal{N} \to
\omega^{\leq \omega} \) such that \( f \in
\mathsf{D}^\mathsf{W}_2(\mathcal{N}, \omega^{\leq \omega}) \) and
\( f^{-1} \in \mathsf{D}^\mathsf{W}_3(\omega^{\leq \omega} ,
\mathcal{N}) \).
\end{enumerate}
\end{proposition}

\begin{proof*}
(1)
The space \( X = \mathcal{N} \setminus \{ n 0^\omega \mid n \in \omega \} \) is a nonempty perfect zero-dimensional Polish space whose compact subsets all have empty interior, hence it is homeomorphic to \( \mathcal{N} \) by the Alexandrov-Urysohn theorem \cite[Theorem 7.7]{ke94}. Let \( \hat{f} \colon X \to \mathcal{N} \) be a witness of this fact, and extend \( \hat{f} \) to a bijection \( f \colon \mathcal{N} \to \omega \sqcup \mathcal{N} \) by setting \( f(n 0^\omega) = n \) for every \( n \in \omega \). Since \( X \) is open in \( \mathcal{N} \) and each \( \{ n 0^\omega \} \) is closed in \( \mathcal{N}  \), the partition \( \langle X, \{ n0^\omega \} \mid n \in \omega \rangle \) of \( \mathcal{N} \) witnesses that \( f \in \mathsf{D}^\mathsf{W}_2(\mathcal{N}, \omega \sqcup \mathcal{N}) \). Conversely, the partition \( \langle \mathcal{N}, \{ n \} \mid n \in \omega \rangle \) is a clopen partition of \( \omega \sqcup \mathcal{N} \) witnessing \( f^{-1} \in \mathsf{D}^\mathsf{W}_2(\omega \sqcup \mathcal{N}, \mathcal{N}) \).

(2)
Since \( \mathcal{N} \) is not \(\sigma\)-compact, the claim follows from Remark~\ref{remDW2}.

(3)
By part (1), it is enough to prove the claim with \( \mathcal{N} \) replaced by \( \omega \sqcup \mathcal{N} \).
Let \( \hat{f} \colon \mathcal{N} \to \mathcal{C} \) be the well-known homeomorphism between \( \mathcal{N} \) and \( Y
= \{ y \in \mathcal{C} \mid  \forall n \exists m \geq n \, (y(m) = 1) \} \) given by \( \hat{f} (x) =
0^{x(0)}10^{x(1)}10^{x(2)}10^{x(3)} \dotsc \) Since \( \mathcal{C} \setminus Y \) is countable, we can fix an
enumeration \( \langle y_n  \mid n \in \omega \rangle \) without repetitions of such a set. Extend \( \hat{f} \) to a bijection \( f
\colon \omega \sqcup \mathcal{N} \to \mathcal{C} \) by setting \( f(n) = y_n \) for every \( n \in \omega \). Since each point
of the spaces \( \omega \sqcup \mathcal{N} \) and \( \mathcal{C} \) is closed, \( \mathcal{N} \) is (cl)open in \( \omega
\sqcup \mathcal{N} \), and \( Y \) is a (proper) \( \mathbf{\Pi}^0_2(\mathcal{C}) \) set, we have that \( \langle
\mathcal{N}, \{ n \} \mid n \in \omega \rangle \) is a clopen partition of \( \omega \sqcup \mathcal{N} \) witnessing \( f \in
\mathsf{D}^\mathsf{W}_2(\omega \sqcup \mathcal{N}, \mathcal{C}) \), and \( \langle Y, \{ y_n \} \mid n \in \omega
\rangle \) is a \( \mathbf{\Pi}^0_2 \)-partition of \( \mathcal{C} \) witnessing \( f^{-1} \in
\mathsf{D}^\mathsf{W}_3(\mathcal{C}, \omega \sqcup \mathcal{N}) \).

(4)
Let first \( n = 1\). By part (1), it is enough to prove the claim with \( \mathcal{N} \) replaced by \( \omega \sqcup \mathcal{N} \). Let \( \langle q_k \mid k \in \omega \rangle \) be an enumeration without repetition of the set of rational numbers \( \QQ \). It is  well-known that \( \mathcal{N} \) and \( \RR \setminus \QQ \) are homeomorphic, so let \( \hat{f} \) be a witness of this fact. Extend \( \hat{f} \) to a bijection \( f \colon \omega \sqcup \mathcal{N} \to \RR \) by setting \( f(k) = q_k \) for every \( k \in \omega \). Since  \( \RR \setminus \QQ \) is a (proper) \( \mathbf{\Pi}^0_2 (\RR) \) set and each singleton of \( \RR \) (and hence of \( \QQ \)) is closed, we have that \( \langle \mathcal{N}, \{ n \} \mid n \in \omega \rangle \) is a clopen partition of \( \omega \sqcup \mathcal{N} \) witnessing \( f \in \mathsf{D}^\mathsf{W}_2(\omega \sqcup \mathcal{N}, \RR) \), and \( \langle \RR \setminus \QQ, \{ q_n \} \mid n \in \omega \rangle \) is a \( \mathbf{\Pi}^0_2 \)-partition of \( \mathcal{C} \) witnessing \( f^{-1} \in \mathsf{D}^\mathsf{W}_3(\RR, \omega \sqcup \mathcal{N}) \).

Now assume \( n > 1\).
First observe that \( Z =  \bigcup_{0 < i < n } ([n]^i \times \QQ^i \times \mathcal{N}) \) (where each \( [n]^i \times \QQ^i \) is endowed with the discrete topology) is homeomorphic to \( \mathcal{N} \), hence by part (1) and the fact that \( \mathcal{N} \sqcup \mathcal{N} \) is homeomorphic to \( \mathcal{N} \) it is enough to prove the claim with \( \mathcal{N} \) replaced by \( \omega \sqcup \mathcal{N} \sqcup Z \). For each \( 0 < i < n \), \( a = \{a_0, \dotsc, a_{i-1} \} \in [n]^{i} \) and \( s \in \QQ^i \), let \( f_{a,s} \) be an homeomorphism between \( \mathcal{N} \) and
\[
\RR_{a,s} = \{ x \in \RR^n \mid \forall j < i \, (x(a_j) = s(j)) \wedge \forall k \notin a \, (x(k) \in \RR \setminus \QQ)  \}
\]
(such an homeomorphism exists because \( \RR_{a,s} \) is homeomorphic to \( ( \RR \setminus \QQ)^{n-i} \)). Let also \( \hat{f} \) be an homeomorphism between  \( \mathcal{N} \) and \( ( \RR \setminus \QQ)^n \), and \( \langle t_k \mid k \in \omega \rangle \) be an enumeration without repetitions of \( \QQ^n \). Then define
\[
f \colon  \omega \sqcup \mathcal{N} \sqcup Z  \to \RR^n
\]
by setting \( f(k) = t_k \), \( f(x) = \hat{f}(x) \), and \( f(a,s,x) = f_{a,s}(x) \) for every \( k \in \omega \), \( x \in \mathcal{N} \) and \( (a,s) \in \bigcup_{0 < i < n } ([n]^i \times \QQ^i) \). It is easy to check that \( f \) is in fact a bijection. Moreover, since all the \( \RR_{a,s} \) and \( (\RR \setminus \QQ)^n \) are \( \mathbf{\Pi}^0_2(\RR^n) \) sets and all points are closed in \( \RR^n \), we have that
 \[ \langle \mathcal{N}, \{ k \}, \{ (a,s,x) \mid x \in \mathcal{N} \} \mid k \in \omega, (a,s) \in \bigcup_{0<i<n} ([n]^i \times \QQ^i)  \rangle \]
 is a clopen partition of \( \omega \sqcup \mathcal{N} \sqcup Z \) witnessing \( f \in \mathsf{D}^\mathsf{W}_2(\omega \sqcup \mathcal{N} \sqcup Z, \RR^n) \), and   \[ \langle \QQ^n, (\RR \setminus \QQ)^n, \RR_{a,s} \mid (a,s) \in \bigcup_{0<i<n} ([n]^i \times \QQ^i) \rangle \]
 is a \( \mathbf{\Pi}^0_2 \)-partition of \( \RR^n \) witnessing \( f^{-1} \in \mathsf{D}^\mathsf{W}_3(\RR^n, \omega \sqcup \mathcal{N} \sqcup Z) \).

(5)
By part (1), it is again enough to prove the claim with \( \mathcal{N}\) replaced by \( \omega \sqcup \mathcal{N} \).  Let \( \langle \sigma_n \mid n \in \omega \rangle \) be an enumeration without repetition of \( \omega^* \). Define \( f \colon \omega \sqcup \mathcal{N} \to \omega^{\leq \omega} \) by setting \( f(n) = \sigma_n \) and \( f(x) = x \) for all \( n \in \omega \) and \( x \in \mathcal{N} \). Since all the \( \sigma \in \omega^* \) are compact elements of \( \omega^{\leq \omega} \), their singletons are \( \mathbf{\Delta}^0_2 (\omega^{\leq \omega})\) subsets by Proposition~\ref{propsingl}(3), hence \( \mathcal{N} \) is a \( \mathbf{\Pi}^0_2(\omega^{\leq \omega}) \) set. Therefore, \( \langle \mathcal{N}, \{ n \} \mid n \in \omega \rangle \) is a clopen partition of \( \omega \sqcup \mathcal{N} \) witnessing \( f \in \mathsf{D}^\mathsf{W}_2(\omega \sqcup \mathcal{N}, \omega^{\leq \omega}) \), while \( \langle \mathcal{N}, \{ \sigma \} \mid \sigma \in \omega^*  \rangle \) is a \( \mathbf{\Pi}^0_2\)-partition of \( \omega^{\leq \omega} \) witnessing \( f^{-1} \in \mathsf{D}^\mathsf{W}_3(\omega^{\leq \omega}, \omega \sqcup \mathcal{N}) \).  \hfill \usebox{\proofbox}
\end{proof*}

A natural way to compute the complexity of an isomorphism between two topological spaces is given by the following variant of the usual Schr\"oder-Bernstein argument (see also~\cite{JR79b}).

\begin{lemma}\label{lemmaSB}
Let \( 1 \leq \alpha < \omega_1 \), \( X,Y \in \mathscr{X} \), and \( \F  \) be a collection of functions between topological spaces closed under restrictions (i.e.\ \( f \restriction X' \in \F \) for every \( f \colon X \to Y \in \F \) and \( X' \subseteq X\)). If \( X \) is \( \F \)-isomorphic to a  subset of \( Y \) via some \( f \in \mathbf{\Pi}^0_\alpha[X,Y] \) and \( Y \) is \( \F \)-isomorphic to a  subset of \( X \) via some \( g \in \mathbf{\Pi}^0_\alpha[Y,X] \), then there are \( X' \in \mathbf{\Delta}^0_{\alpha + 1}(X) \) and \( Y' \in \mathbf{\Delta}^0_{\alpha+1}(Y) \) such that \( X' \simeq_\F Y' \) and \( {X \setminus X'} \simeq_\F {Y \setminus Y'} \).

In particular, if \( X \) (respectively, \( Y \)) is homeomorphic to a \( \mathbf{\Pi}^0_\alpha \) subset of \( Y \) (respectively, \( X \)), then \( X \simeq^\mathsf{W}_{\alpha+1} Y \).
\end{lemma}

\begin{proof}
Inductively define \( X_n \subseteq X \) and \( Y_n \subseteq Y \), \( n \in \omega \), by setting \( X_0 = X \), \( Y_0 = Y \), \( X_{n+1} = g(Y_n) \), \(Y_{n+1} = f(X_n) \). Let also \( X_ \infty = \bigcap_{n \in \omega } X_n \) and \( Y_\infty = \bigcap_{n \in \omega} Y_n \). By our assumption on \( f \) and \( g \), all of \( X_n, Y_n, X_{\infty}, Y_\infty \) are in \( \mathbf{\Pi}^0_\alpha \). Let \( X'  = X_\infty \cup \bigcup_{n \in \omega} (X_{2n} \setminus X_{2n+1}) \) and \( Y' = Y_\infty \cup \bigcup_{n \in \omega} (Y_{2n+1} \setminus Y_{2n+2}) \). By their definition, \( X' \) and \( Y' \) are both in \( \mathbf{\Sigma}^0_{\alpha+1} \). Since \( X \setminus X' = \bigcup_{n \in \omega } (X_{2n+1} \setminus X_{2n})\) and \( Y \setminus Y' = \bigcup_{n \in \omega} (Y_{2n} \setminus Y_{2n+1}) \) are both in \( \mathbf{\Sigma}^0_{\alpha+1} \) as well, we have that \( X', Y' \in \mathbf{\Delta}^0_{\alpha+1} \). Finally, \( f \restriction X' \) and \( g^{-1} \restriction (X \setminus X') \) witness that \( X' \simeq_\F Y' \) and \( {X \setminus X'} \simeq_\F {Y \setminus Y'} \) because \( \F \) is closed under restrictions and \(  (f \restriction X')^{-1} = f^{-1} \restriction Y' \) and \( (g^{-1} \restriction (X \setminus X'))^{-1} = g \restriction (Y \setminus Y') \).
\end{proof}

We can immediately derive some corollaries from Lemma~\ref{lemmaSB}.
We need to recall the following definition from general topology: two
spaces are of the same Fr\'echet dimension type if each one is homeomorphic
 to a subset of the other.

\begin{corollary}\label{corfrechet}
If \( X,Y\) are two quasi-Polish  spaces which are of the same
Fr\'echet dimension type then \( X \simeq^\mathsf{W}_3 Y \).
\end{corollary}

\begin{proof}
Apply the second part of Lemma~\ref{lemmaSB} with \( \alpha = 2 \),
using the fact that the class of quasi-Polish spaces is closed under
homeomorphism and Proposition~\ref{propsubspace}.
\end{proof}

The second part of the next corollary has been essentially already
noticed in \cite[Theorem 6.5]{JR79b}.

\begin{corollary}\label{corcompact}
If \( X,Y \) are  two quasi-Polish spaces such that \( X \) is
homeomorphic to a closed subset of \( Y \) and \( Y \) is
homeomorphic to a closed subset of \( X \) then \( X
\simeq^\mathsf{W}_2 Y \).

In particular, if  \( X,Y \) are compact Hausdorff quasi-Polish
spaces of the same Fr\'echet dimension type then \( X
\simeq^\mathsf{W}_2 Y \).
\end{corollary}

\begin{proof}
The first part follows from Lemma~\ref{lemmaSB} with \( \alpha = 1 \). The second part follows from the first one and the classical facts that the class of compact spaces is closed under continuous images, and that a compact subset of an Hausdorff space is closed.
\end{proof}

Our next goal  is to extend Proposition~\ref{propexamples} (3)--(5)
to a wider class of quasi-Polish spaces (see Theorem~\ref{theordim}). Such generalization will involve the definition of
the (inductive) topological dimension of a space \( X \), denoted in
this paper by \( \dim(X) \) --- see e.g.\ \cite[p. 24]{hurwal}.

\begin{definition}
The empty set  \( \emptyset \) is the only space in \( \mathscr{X}
\) with \emph{dimension \( - 1 \)}, in symbols \( \dim(\emptyset) =
-1 \).

Let \( \alpha \)  be an ordinal and \( \emptyset \neq X \in
\mathscr{X} \). We say that \( X \) has \emph{dimension \( \leq
\alpha \)}, \( \dim(X) \leq \alpha \) in symbols, if every \( x \in
X \) has arbitrarily small neighborhoods whose boundaries have
dimension \( < \alpha \), i.e.\ for every \( x \in X \) and every
open set \( U \) containing \( x \) there is an open \( x \in V
\subseteq U \) such that \( \dim(\partial V) \leq \beta \) (where \(
\partial V = \mathrm{cl}(V) \setminus V \) and $\mathrm{cl}(V)$ is the closure of $V$ in \( X \)) for some \( \beta < \alpha
\).

We say that a space \( X \) has \emph{dimension \( \alpha \)}, \( \dim(X) = \alpha \) in symbols, if \( \dim(X) \leq \alpha \) and \( \dim(X) \nleq \beta \) for all \( \beta < \alpha \).

Finally, we say that a space \( X \) has \emph{dimension \( \infty \)}, \( \dim(X) = \infty \) in symbols, if \( \dim(X) \nleq \alpha \) for every  \( \alpha \in \mathsf{On} \).
\end{definition}

It is obvious that the dimension of a space is a topological invariant (i.e.\ \( \dim(X)  = \dim(Y) \) whenever \( X \) and \( Y \) are homeomorphic). Moreover, one can easily check that \( \dim(X) \leq \alpha \) (for \(\alpha\) an ordinal) if and only if there is a base of the topology of \( X \) consisting of open sets whose boundaries have dimension \( < \alpha \). Therefore, if \( X \) is countably based and \( \dim(X) \neq \infty \), then \( \dim(X) = \alpha \) for some \emph{countable} ordinal \(\alpha\).

The following lemma shows that the notion of dimension is monotone.

\begin{lemma} \cite[Theorem III 1]{hurwal}\label{lemmamonotone}
Let \( X \in \mathscr{X} \)  and \(\alpha\) be an ordinal such
that \( \dim(X) \leq \alpha \). Then for every \( Y \subseteq X
\), \( \dim(Y) \leq \alpha \) (where \( Y \) is endowed with the
relative topology inherited from \( X \)).
\end{lemma}

\begin{proof}
This is proved by induction on \(\alpha\), using the fact  that if
\( Y \subseteq X \) and \( U \) is open in \( X \) then the
boundary in \( Y \) of \( U \cap Y \) is contained in the boundary
of \( U \) in \( X \).
\end{proof}

It is a classical fact that for every \( \alpha < \omega_1 \)
there is a compact Polish space of dimension \( \alpha \), and
that the Hilbert cube \( [0,1]^\omega \) is a compact Polish space
of dimension \( \infty \). Here we provide various examples of
computations of the dimension of some concrete quasi-Polish spaces
which are relevant for the results of this paper.

\begin{example}\label{exdim1}
\emph{Finite dimension}.
\begin{enumerate}[(1)]
\item
\( \dim(\mathcal{N}) = \dim(\mathcal{C}) = 0 \);
\item
\( \dim(\RR^n) = n \) for every \(0 \neq  n \leq \omega \);
\item
for \( n < \omega\), let \( L_n \) be the (finite) quasi-Polish space obtained by endowing the dcpo \( (n, \leq ) \) with the Scott (equivalently, the Alexandrov)  topology: then \( \dim(L_n) = n-1 \).
\end{enumerate}
\end{example}

\begin{proof*}
(1)
The canonical basis for \( \mathcal{N} \) and \( \mathcal{C} \) (namely, the collection of all sets of the form
\( \sigma \cdot  \mathcal{N} \)
for \( \sigma \in \omega^* \) and, respectively, \( \sigma \cdot \mathcal{C} \) for \( \sigma \in 2^* \)) consist
of clopen sets, hence their elements have empty boundary.

(2)
This is a classical (nontrivial) fact, see e.g.\ \cite[Theorem IV 1]{hurwal}.

(3)
This is proved by induction on \( n \geq 0 \). If \( n = 0 \), then \( L_n = \emptyset \) and hence \( \dim(L_0) = -1 \) by definition. Now assume \( \dim(L_i) = i \) for every \( i \leq n \) and consider the space \( L_{n+1} \).  Every open set of \( L_n \) is of the form \( U_i = \{ j \in L_{n+1} \mid j \geq i \} \) for some \( i \leq n \), and \( \partial U_i = L_i \): hence by the inductive hypothesis \( \dim(\partial U_i) < n+1 \) for every \( i \leq n \), which implies \( \dim(L_{n+1}) \leq n+1 \). Moreover, the set \( \{ n \} \) is open in \( L_{n+1} \), and is obviously the minimal open set containing \( n \). Since \( \partial \{ n \} = L_n \), \( \dim(L_{n+1}) > \dim(L_n) = n \). Therefore \( \dim(L_{n+1}) = n+1 \), as desired.  \hfill \usebox{\proofbox}
\end{proof*}

\begin{example}\label{exdim2}
\emph{Transfinite dimension}.
\begin{enumerate}[(1)]
\item
the disjoint union \( X =  \bigsqcup_{0 \neq n \in \omega} [0,1]^n \) of the \( n \)-dimensional cubes \( [0,1]^n \) is a Polish space of dimension \( \omega \);
\item
let \( \omega^{\leq \omega} \) be the \(\omega\)-algebraic domain  \( (\omega^{\leq \omega}, \sqsubseteq) \) endowed with the Scott topology: then \( \dim(\omega^{ \leq \omega})  = \omega \);
\item
for \( \alpha < \omega_1 \), let \( L_{\alpha+1} \) be the quasi-Polish space obtained by endowing the dcpo\footnote{Here we cannot consider the limit case, as if \( \alpha \) is limit then the poset \( (\alpha, \leq) \) is not directed-complete, and hence falls out of the scope of the spaces considered in this paper.} \( (\alpha+1, \leq ) \) with the Scott  topology. Then  \( \dim(L_{\alpha+1}) = \alpha \).
\end{enumerate}
\end{example}

\begin{proof*} (1)
By part (2), each \( [0,1]^n \) has dimension \( n \). Since \( [0,1]^n \) is topologically embedded in \( X \), by Lemma~\ref{lemmamonotone} we have \( \dim(X) \geq n \) for every \( n \in \omega \), and hence \( \dim(X) \geq \omega \). Let \( \mathcal{B}_n = \{ B_{n,m} \mid m \in \omega \} \) be a countable basis of \( [0,1]^n \) such that \( \dim(\partial B_{n,m}) < n \) for every \( m \in \omega \). Then \(  \mathcal{B} = \bigcup_{n \in \omega } \mathcal{B}_n \) is a basis for \( X \) with the property that for every \( U \in \mathcal{B} \), \( \dim(\partial U) < \omega \): hence \( \dim(X) \leq \omega \), and therefore \( \dim(X) = \omega \).

(2)
Since every \( L_n \) can be topologically embedded in \( \omega^{ \leq \omega } \), \( \dim(\omega^{\leq \omega}) \geq \omega \) by Lemma~\ref{lemmamonotone} and Example~\ref{exdim1}(2). Consider the basis \( \mathcal{B} \) of \( \omega^{ \leq \omega} \) consisting of the open sets generated by its compact elements, i.e.\ of the sets \(  \sigma \cdot \omega^{\leq \omega} \) for \( \sigma \in \omega^* \). Then \( \partial \sigma \cdot \omega^{\leq \omega} = \{ \tau \sqsubseteq \sigma \mid \tau \neq \sigma \} \). Therefore \( \partial \sigma \cdot \omega^{\leq \omega} \) is homeomorphic to \( L_n \), where \( n \) is the length of \( \sigma \): this means that, by Example~\ref{exdim1}(2) again, \( \dim(\partial U) < \omega \) for every \( U \in \mathcal{B} \), and hence \( \dim(\omega^{\leq \omega}) \leq \omega \). Therefore \( \dim(\omega^{\leq \omega}) = \omega \).

(3)
By an inductive argument similar to the proof of Example~\ref{exdim1}(3). \hfill \usebox{\proofbox}
\end{proof*}

\begin{example}\label{exdim3}
\emph{Dimension \( \infty \)}.
\begin{enumerate}[(1)]
\item the Hilbert cube \( [0,1]^\omega \),  the space \(
\RR^\omega \) (both endowed with the product topology), and the
Scott domain \( P \omega \) have all dimension \( \infty \); \item
Let \( \mathsf{C}_\infty \) be the quasi-Polish space obtained by
endowing the poset \( (\omega, \geq ) \) with the Scott
(equivalently, the Alexandrov) topology. Then \( \mathsf{C}_\infty
\) is a (scattered) countable space with \(
\dim(\mathsf{C}_\infty) = \infty \). Hence the space \(
\mathsf{UC}_\infty = \mathsf{C}_\infty \times \mathcal{N} \),
endowed with the product topology, is an (uncountable)
quasi-Polish space of dimension \( \infty \).
\end{enumerate}
\end{example}

\begin{proof*}
(1) It is a classical fact that \( \dim([0,1]^\omega) =
\dim(\RR^n) = \infty \) --- see e.g.\ Corollary on p.~51 of~\cite{hurwal}. Since the Hilbert cube can be topologically
embedded into \( P \omega \) by Proposition~\ref{proppi2}, it
follows from Lemma~\ref{lemmamonotone} that also the Scott domain
has dimension \( \infty \).

(2) To show that a topological space \( X \) has  dimension \(
\infty \) it is enough to find a point \( x \in X\) and an open
neighborhood \( U \) of \( x \) such that \( X \) can be
topologically embedded into \( \partial V \) for every open \( x
\in V \subseteq U \). Consider the point \( 0 \in
\mathsf{C}_\infty \). Since \( 0 \) is a compact element,  the
basic open set \( U =  \uparrow \! \! 0=  \{ 0 \} \) generated by
\( 0 \) is a minimal (with respect to inclusion) open neighborhood for this point, hence it is
enough to show that \( \mathsf{C}_\infty \) can be topologically
embedded into \( \partial U \). Since \( \mathsf{C}_\infty \) has
a topmost element (i.e.\ \( 0 \) itself), \( \partial U =
\mathsf{C}_\infty \setminus U \); but then the map sending \( n \)
into \( n+1 \) (for every \( n \in \omega \)) is clearly an
homeomorphism between \( \mathsf{C}_\infty \) and \(
\mathsf{C}_\infty \setminus U \), and hence \( \dim(\mathsf{C}_\infty)
= \infty \), as required. The second part of the claim follows
from Lemma~\ref{lemmamonotone} and the fact that \(
\mathsf{C}_\infty \) can be topologically embedded into \(
\mathsf{UC}_\infty \) in the obvious way. \hfill
\usebox{\proofbox}
\end{proof*}

\begin{remark}
The definition of dimension is usually formulated
for separable metric spaces~\cite{hurwal} or for regular topological spaces. This is
because the received opinion is that outside this scope this notion becomes
somewhat pathological.
Examples~\ref{exdim1}, \ref{exdim2} and~\ref{exdim3} show e.g.\ that there
are finite (quasi-Polish) spaces with nonzero
dimension,\footnote{Notice that there are also examples
of \emph{Hausdorff} countable spaces with nonzero dimension.} and countable (quasi-Polish) spaces with arbitrarily high ordinal
dimension, or even of dimension \( \infty \): this seems to contradict our
intuition of ``geometric dimension''.  Nevertheless,
Lemma~\ref{lemmamonotone} shows that some natural properties of the
dimension function \( \dim(\cdot) \) are preserved when considering arbitrary
spaces, and Theorem~\ref{theordim} will show that it remains a quite useful
notion also in this broader context.
\end{remark}

We now recall some classical results that will be used later.

\begin{lemma} (see e.g.\ \cite[pp. 50-51]{hurwal}) \label{lemmainfinitedimension}
Let \( X \) be a Polish space. Then the following are equivalent:
\begin{enumerate}[(1)]
\item
\( \dim(X) \neq \infty \);
\item
\( X = \bigcup_{n < \omega} X_n \) with all the \( X_n \) of finite dimension (i.e.\ \( \dim(X_n) < \omega \) for every \( n < \omega \));
\item
\( X = \bigcup_{n < \omega } X_n \) with all the \( X_n \) of dimension \( 0 \).
\end{enumerate}
\end{lemma}

Notice that by Example~\ref{exdim3}(2), Lemma~\ref{lemmainfinitedimension} cannot be extended to the context of arbitrary quasi-Polish spaces: the space \( \mathsf{C}_\infty \) has dimension \( \infty \), but can be decomposed into countably many zero-dimensional spaces (namely, its singletons). Similarly, \( \mathsf{UC}_\infty \) can be decomposed into countably many copies of \( \mathcal{N} \). On the other hand, we have the following corollary.

\begin{corollary}\label{corScottdomain}
The Scott domain \( P \omega \) cannot be written as \( \bigcup_{n< \omega} X_n \) with all the \( X_n \) of finite dimension. The same is true if \( P \omega \) is replaced with any quasi-Polish space which is universal for (compact) Polish spaces.
\end{corollary}

\begin{proof}
Since \( P \omega \) is universal for the class of all (quasi-)Polish spaces by Proposition~\ref{proppi2}, any decomposition into countably many finite dimensional spaces of \( P \omega \) would induce a similar decomposition of e.g.\ \( [0,1]^\omega \), contradicting Lemma~\ref{lemmainfinitedimension}.
\end{proof}

\begin{lemma} \cite[Theorem 6.1]{JR79b} \label{lemma0dim}
Let \( X \) be an uncountable zero-dimensional Polish space:
\begin{enumerate}[(1)]
\item
if \( X \) is \(\sigma\)-compact then \( X \simeq^\mathsf{W}_2 \mathcal{C} \);
\item
if \( X \) is not \(\sigma\)-compact then \( X \simeq^\mathsf{W}_2 \mathcal{N} \).
\end{enumerate}
\end{lemma}

\begin{proof}
First assume that \( X \) is \(\sigma\)-compact. Since \( \mathcal{C} \) is (\(\sigma\)-)compact as well, it is enough to show that \( X \simeq^\mathsf{W}_2 {(\omega \times \mathcal{C}) \cup (\omega \times \{ 3^\omega\} )} \). Let \( K_n \) be compact sets such that \( X = \bigcup_{n \in \omega}K_n \). Notice that since \( X \) is zero-dimensional we can assume that the \( K_n \)'s are pairwise disjoint. (If not, replace the \( K_n \)'s with any closed refinement of the partition of \( X \) given by the sets \( D_n = K_n \setminus \bigcup_{i < n} K_i \in  \mathbf{\Delta}^0_2 \).) By the Cantor-Bendixson theorem \cite[Theorem 6.4]{ke94} and the obvious fact that every uncountable zero-dimensional Polish space can be partitioned into countably many closed sets such that infinitely many of them are uncountable and infinitely many of them are singletons, we can further assume that all the \( K_{2n} \) are nonempty and perfect and all the \( K_{2n+1} \) are singletons. By Brouwer's theorem \cite[Theorem 7.4]{ke94}, for every \( n \in \omega \) there is an homeomorphism \( f_{2n} \) between \( K_{2n} \) and \( \{ n \} \times \mathcal{C} \). Let \( f_{2n+1} \) be the constant function sending the unique point in \( K_{2n+1} \) to \( (n,3^\omega) \). Then the closed partitions \( \langle K_n \mid n \in \omega \rangle \) and \( \langle \{n\} \times \mathcal{C}, \{ (n,3^\omega) \} \mid n \in \omega \rangle \) of, respectively, \( X \) and \( (\omega \times \mathcal{C}) \cup (\omega \times \{ 3^\omega\} ) \), together with the homeomorphisms \( \langle f_n \mid n \in \omega \rangle \), witness that \( X \simeq_{\mathsf{pw}(\mathbf{\Delta}^0_2)} {(\omega \times \mathcal{C}) \cup (\omega \times \{ 3^\omega\} )} \), and hence \( X \simeq^\mathsf{W}_2 {(\omega \times \mathcal{C}) \cup (\omega \times \{ 3^\omega\} )} \) by Lemma~\ref{lemmapwhomeo}.

Now assume that \( X \) is not \(\sigma\)-compact. By Hurewicz' theorem (see e.g.\ \cite[Theorem 7.10]{ke94}, \( X \) contains a closed set homeomorphic to \( \mathcal{N} \). Conversely, by \cite[Theorem 7.8]{ke94} we have that \( X \), being zero-dimensional,  is homeomorphic to a closed subset of \( \mathcal{N} \). Hence \( X \simeq^\mathsf{W}_2 \mathcal{N} \) by Lemma~\ref{lemmaSB}.
\end{proof}

\begin{remark}
Formally, \cite[Theorem 6.1]{JR79b} is stated using \( \mathsf{D}_2 \)-isomorphisms instead of \( \mathsf{D}^\mathsf{W}_2 \)-isomorphisms. However, our formulation can be recovered from that result \emph{a posteriori} by using \cite[Theorem 5]{JR82}.
\end{remark}

We are now ready to prove the main theorem of this section.

\begin{theorem}\label{theordim}
Let \( X \) be an uncountable quasi-Polish space.
\begin{enumerate}[(1)]
\item
if \( \dim(X) \neq \infty \) then there is a bijection \( f \colon \mathcal{N} \to X \) such that \( f \in \mathsf{D}^\mathsf{W}_2(\mathcal{N},X) \) and \( f^{-1} \in \mathsf{D}^\mathsf{W}_3(X, \mathcal{N}) \). In particular, \( \mathcal{N} \simeq^\mathsf{W}_3 X \);
\item
if \( \dim(X) = \infty \) and \( X \) is Polish then \( \mathcal{N} \not\simeq^\mathsf{W}_\alpha X \) for every \( \alpha < \omega_1 \) and \( \mathcal{N} \not\simeq_n X \) for every \( n < \omega \);
\item
 \( P \omega \not\simeq^\mathsf{W}_\alpha \mathcal{N} \) for every \( \alpha < \omega_1 \) and \( P \omega \not\simeq_n \mathcal{N} \) for every \( n < \omega \).  The same result holds when replacing \( P\omega \) with any other quasi-Polish space which is universal for (compact) Polish spaces;
\item
\( \mathsf{UC}_\infty \simeq^\mathsf{W}_2 \mathcal{N} \). Therefore \( \mathsf{UC}_\infty \not\simeq^\mathsf{W}_\alpha X \)
 (\( \alpha < \omega_1 \)) and \( \mathsf{UC}_\infty \not\simeq_n X \) (\( n \in \omega \)) for \( X \) a Polish space
 of dimension \( \infty \) (e.g.\ \( X = [0,1]^\omega \) or \( X = \RR^\omega \)) or \( X = P \omega \).
\end{enumerate}
\end{theorem}

\begin{proof*} (1)
Since \( X \) is uncountable (hence nonempty) and countably based, we can assume \( 0 \leq \dim(X) < \omega_1 \). We
argue by induction on \( \dim(X) = \alpha\). If \( \alpha = 0 \), then \( X \) is Hausdorff and regular. By the Urysohn's
metrization theorem \cite[Theorem 1.1]{ke94}, we have that \( X \) is metrizable and hence Polish by Proposition~\ref{propmetrizable}. Therefore the claim for \( X \)  follows from Lemma~\ref{lemma0dim} and Proposition~\ref{propexamples}(3).
Now assume that \( \alpha > 0 \) and that the claim is true for every quasi-Polish space of dimension \( < \alpha \). Let \(
\mathcal{B} = \{ B_n \mid n \in \omega \} \) be a countable base for the topology of \( X \) such that \( \dim (\partial B_n) <
\alpha \) for every \( n \in \omega \). Let \( X' = X \setminus \bigcup_{n \in \omega} \partial B_n \) and inductively define \(
B'_n = \partial B_n \setminus \bigcup_{i < n} \partial B_i \). All of \( X', B'_n \) are \( \mathbf{\Pi}^0_2(X) \) sets, so they are
quasi-Polish by Proposition~\ref{propsubspace}. Moreover, they clearly form a (countable) partition of \( X \). The space \(
X' \) is zero-dimensional, so by Lemma~\ref{lemma0dim} and Proposition~\ref{propexamples}(3) there is a bijection \( f_0 \colon \{ 0 \} \times \mathcal{N} \to X' \) such that \( f_0 \in \mathsf{D}^\mathsf{W}_2(\{ 0 \} \times \mathcal{N},X') \) and \( f_0^{-1} \in \mathsf{D}^\mathsf{W}_3(X', \{ 0 \} \times \mathcal{N}) \).
 By our hypothesis on \( \mathcal{B} \) and Lemma~\ref{lemmamonotone}, \( \dim(B'_n) < \alpha \) for every \( n \in \omega \), hence by inductive hypothesis for each \( n \in
\omega \) there is a bijection \( f_{n+1} \colon \{ n+1 \} \times \mathcal{N} \to B'_n \) such that \( f_{n+1} \in
\mathsf{D}^\mathsf{W}_2(\{ n+1 \} \times \mathcal{N}, B'_n) \) and \( f_{n+1}^{-1} \in \mathsf{D}^\mathsf{W}_3(B'_n
, \{ n+1 \} \times \mathcal{N}) \). Let \( h \) be an homeomorphism between \( \mathcal{N} \) and \( \omega \times
\mathcal{N} \): then, using the fact that each \( \{ n \} \times \mathcal{N} \) is clopen in \( \omega \times \mathcal{N} \)
and all of \( X', B'_n \) are in \( \mathbf{\Pi}^0_2(X) \subseteq \mathbf{\Delta}^0_3(X) \), it is straightforward to check that
\( f = \Big ( \bigcup_{n \in \omega} f_n \Big ) \circ h  \colon \mathcal{N} \to X \) is a bijection such that \( f \in \mathsf{D}^\mathsf{W}_2(\mathcal{N},X) \) and \( f^{-1} \in \mathsf{D}^\mathsf{W}_3(X, \mathcal{N}) \).

(2)
For the first part, assume toward a contradiction that \( \mathcal{N} \simeq^\mathsf{W}_\alpha X \) for some \( \alpha < \omega_1 \). Then
\( \mathcal{N} \simeq_{\mathsf{pw}(\mathbf{\Delta}^0_\alpha)} X \) by Lemma~\ref{lemmapwhomeo}, i.e.\ there would be
countable partitions \( \langle Y_n \mid n \in \omega \rangle , \langle X_n \mid n \in \omega \rangle \) in \(
\mathbf{\Delta}^0_\alpha \) sets of, respectively, \( \mathcal{N} \) and \( X \) such that \( Y_n \) is homeomorphic to \( X_n
\) for every \( n \in \omega \). In particular, by Lemma~\ref{lemmamonotone} each of the \( Y_n \) would be zero-dimensional, and
since homeomorphisms preserve dimension, we would also have \( \dim(X_n) = 0 \) for every \( n \in \omega\), contradicting
Lemma~\ref{lemmainfinitedimension}.

For the second part, we argue again by contradiction. Let \( f \) be a witness of \( \mathcal{N} \simeq_n X \) (for some \( n \in \omega \)). By
Proposition~\ref{propDndec}, both \( f \) and \( f^{-1} \) are decomposable, and hence \( \mathcal{N}
\simeq_{\mathsf{pw}} X \) by Lemma~\ref{lemmapwhomeo}. From this fact, arguing as in the first part, we again reach a contradiction with Lemma~\ref{lemmainfinitedimension}.

(3)
For the first part, argue as in part (2), using Corollary~\ref{corScottdomain} instead of Lemma~\ref{lemmainfinitedimension}.

For the second part, assume towards a contradiction that \( P
\omega \simeq_n \mathcal{N} \) (for some \( n < \omega \)). Since
\( P \omega \) is universal for quasi-Polish spaces by Proposition~\ref{proppi2}, there is  a \( \boldsymbol{\Pi}^0_2(P \omega) \)
set $X$ which is homeomorphic to \( [0,1]^\omega \), and hence the
image \( Y \subseteq \mathcal{N} \) of \( X \) under any witness of \( P \omega \simeq_n
\mathcal{N} \) would be a Borel set 
such that \(  Y \simeq_n [0,1]^\omega \). Arguing as in part (2),
we reach a contradiction with Lemma~\ref{lemmainfinitedimension}.

(4)
Since \( \mathcal{N} \) is homeomorphic to \( \omega \times \mathcal{N} \), it is enough to prove that \(
\mathsf{UC}_\infty \simeq^\mathsf{W}_2 \omega \times \mathcal{N} \). Since each element  of \( \mathsf{C}_\infty
\) is compact, \( \{ n \} \in \mathbf{\Delta}^0_2(\mathsf{C}_\infty) \) for every \( n \in \omega \). Hence the sets \( X_n = \{
n \} \times \mathcal{N} \), which are all homeomorphic to \( \mathcal{N} \),  form a countable partition of \( \mathsf{UC}_\infty \) in \( \mathbf{\Delta}^0_2 \) pieces. It
follows that \( \mathsf{UC}_\infty \simeq_{\mathsf{pw}(\mathbf{\Delta}^0_2)} \omega \times \mathcal{N} \), whence \(
\mathsf{UC}_\infty \simeq^\mathsf{W}_2 \omega \times \mathcal{N} \) by Lemma~\ref{lemmapwhomeo}.
The second part of the claim follows from the first one and parts (2) and (3). \hfill \usebox{\proofbox}
\end{proof*}

\begin{remark} \label{remJR}
\begin{enumerate}[(1)]
\item
The special case of Theorem~\ref{theordim}(1) in which \( X \) is assumed to
be Polish essentially appeared in  \cite[Theorem 8.1]{JR79b}. The unique
differences are that in their statement Jayne and Rogers used the assumption
that  \( X \)  is a countable union of spaces of finite dimension (which for
\( X \) Polish is equivalent to \( \dim(X) \neq \infty \) by Lemma~\ref{lemmainfinitedimension}), and that the conclusion is weakened in
\cite[Theorem 8.1]{JR79b} to \( \mathcal{N} \simeq_3 X \) (but their proof
already gives the more precise statement considered here). Concerning the
proofs, at a first glance the Jayne-Rogers original argument could seem
different from the one used here, as it does not involve any induction on the
dimension. However, their argument heavily relies on
\cite[Theorem III 3]{hurwal}, whose proof already implicitly shows the result
under consideration and is essentially the same as the one we used above.
Therefore the instance of Theorem~\ref{theordim}(1) concerning Polish spaces
can be dated back at least to the work of Hurewicz and Wallman of 1948.
\item
Theorem~\ref{theordim}(1) is optimal: in fact by (the obvious generalization to transfinite dimension of) \cite[Theorem 13]{JR82}, if \( X \simeq_2 Y \) then \( \dim(X) = \dim(Y) \). The converse is not true, as if \( X \) is a compact Polish space of dimension \( \alpha < \omega_1 \), then \( \dim(X \sqcup \mathcal{N}) = \dim(X)  = \alpha\) (by the generalization  to transfinite dimensions of the Sum Theorem \cite[Theorem III 2]{hurwal}) but \( X \sqcup \mathcal{N} \not\simeq_2 X \) by \cite[Theorem 5]{JR82} and Remark~\ref{remDW2}. Nevertheless, in some specific cases one can get a better bound, e.g.:
\begin{enumerate}[(a)]
\item Let \( X \) be an uncountable Polish space embedded in \(
\RR^n \). Then \( X \) is \( \mathsf{D}_2  \)-isomorphic
(equivalently, \( \mathsf{D}^\mathsf{W}_2 \)-isomorphic) to \(
\RR^n \) if and only if it is \(\sigma\)-compact and of dimension
\( n \) \cite[Theorem 6.3]{JR79b}; \item A Polish space that is
locally Euclidean of dimension \( n \) is    \( \mathsf{D}_2
\)-isomorphic  (equivalently, \( \mathsf{D}^\mathsf{W}_2
\)-isomorphic) to \( \RR^n \) \cite[Theorem 6.4]{JR79b};
 \item Let
\( X , Y \) be two \(\sigma\)-compact metric  spaces of dimension
\( \alpha < \omega_1 \) that are universal for the compact metric
spaces of dimension \(\alpha\). Then \( X \simeq_2 Y \) by Theorem
7.1 in~\cite{JR79b}.
\end{enumerate}
\end{enumerate}
\end{remark}

Theorem~\ref{theordim} has also several corollaries which generalize many results of various nature. The first one is related to Lemma~\ref{lemmainfinitedimension}. Recall that a zero-dimensional space \( X \) is called \( h \)-homogeneous if every clopen \( U \subseteq X \) is homeomorphic to the entire space \( X \). Examples of (uncountable) \( h \)-homogeneous spaces are \( \mathcal{ N} \) and \( \mathcal{C} \).

\begin{corollary} \label{corequiv}
Let \( X \) be an uncountable Polish space. Then the following are equivalent:
\begin{enumerate}[(1)]
\item
\( \dim(X) \neq \infty \);
\item
 \( X = \bigcup_{n < \omega } X_n \) with all the \( X_n \) of finite dimension (equivalently, of dimension \( \neq \infty \), or of dimension \( 0 \));
\item
\( X = \bigcup_{n < \omega} X_n \) with each \( X_n \) a zero-dimensional \( h \)-homogeneous Polish space;
\item
\( X \simeq_n \mathcal{N} \) for some \( n < \omega \);
\item
\( X \simeq^\mathsf{W}_3 \mathcal{N} \).
\end{enumerate}
\end{corollary}

\begin{proof}
(1) \( \Rightarrow \) (5) by Theorem~\ref{theordim}(1), (5) \( \Rightarrow \) (4) is obvious, and (4) \( \Rightarrow \) (1) by Theorem~\ref{theordim}(2). Moreover (3) \( \Rightarrow \) (2) is obvious, and (2) \( \Rightarrow \) (1) by Lemma~\ref{lemmainfinitedimension}, so it is enough to show (5) \( \Rightarrow \) (3). By Lemma~\ref{lemmapwhomeo}, \( \mathcal{N} \simeq_{\mathsf{pw}(\mathsf{\Delta}^0_3)} X \). It is a classical fact that every countable partition of \( \mathcal{N} \) into \( \mathbf{\Delta}^0_3 \) pieces can be refined to a countable partition in \( \mathbf{\Pi}^0_2 \) pieces, hence there is a countable partition \( \langle X'_n \mid  n \in \omega \rangle \) of \( X \) such that each \( X_n \) is homeomorphic to a \( \mathbf{\Pi}^0_2 \)-subset of \( \mathcal{N} \). This means that each \( X'_n \) is a Polish zero-dimensional space.  By e.g.\ \cite[Theorem 1]{ost}, all of these \( X'_n \) can be written as countable unions \( \bigcup_{m \in \omega } X_{n,m} \) of \( h \)-homogeneous spaces such that \( X_{n,m} \in \mathbf{\Pi}^0_2(X_n) \) for all \( m \in \omega \). Thus each \( X_{n,m} \) is Polish by \cite[Theorem 3.11]{ke94}, and hence any enumeration without repetitions \( \langle X_n \mid n \in \omega \rangle \) of \( \{ X_{n,m} \mid n,m , \omega \} \) satisfies all the requirements of (3).
\end{proof}

\noindent
Obviously, conditions (1)--(4) are true also when \( X \) is a countable Polish space and \( \mathcal{N} \) is replaced by \(\omega\): however, the counterexamples \( \mathsf{C}_\infty \) and \( \mathsf{UC}_\infty \) of Example~\ref{exdim3}(2) show that this is no more true for countable \emph{quasi-Polish} spaces, and that Corollary~\ref{corequiv} cannot be extended in general to arbitrary quasi-Polish spaces. Nevertheless we have the following:

\begin{corollary}
Let \( X \) be an arbitrary quasi-Polish space of dimension \( \neq \infty \). Then \( X \) can be written as a countable union of zero-dimensional \( h \)-homogeneous Polish spaces.
\end{corollary}

\begin{proof}
If \( X \) is countable the result is trivial (simply take the singletons of the elements of \( X \) as countable partition). If \( X \) is uncountable, it is enough to observe that the proofs of (1) \( \Rightarrow \) (5) and (5) \( \Rightarrow \) (3) of Corollary~\ref{corequiv} are in fact valid for arbitrary uncountable quasi-Polish spaces.
\end{proof}

The next  corollary extends Semmes' generalization (Theorem~\ref{theorsem} in this paper) of Jayne-Rogers \cite[Theorem
5]{JR82}.

\begin{corollary} \label{corsem}
Let \( X \)  be any quasi-Polish space which is either countable
or of dimension \( \neq \infty \). Then \( \mathsf{D}_3(X) =
\mathsf{D}^\mathsf{W}_3(X) \).
\end{corollary}

\begin{proof}
If \( X \) is countable the result follows from Proposition~\ref{propsingl} because then every function \( f \colon X \to X \)
is in \( \mathsf{D}^\mathsf{W}_3(X) \). If \( X \) is uncountable,
the result follows from Theorem~\ref{theordim}(1) and the fact
that both the classes \( \bigcup_{X,Y \in \mathscr{X}}
\mathsf{D}_3(X,Y) \) and \( \bigcup_{X,Y \in \mathscr{X}}
\mathsf{D}^\mathsf{W}_3(X,Y) \) are closed under composition.
\end{proof}

This corollary shows an interesting phenomenon:  Jayne-Rogers'
original result stating that \( \mathsf{D}_2(X)
= \mathsf{D}^\mathsf{W}_2(X) \)~\cite[Theorem 5]{JR82} can arguably be considered to be simpler,  even in the case \( X =
\mathcal{N} \), than
Theorem~\ref{theorsem}. However, contrarily to the case of
Corollary~\ref{corsem}, it cannot be generalized to \( \omega
\)-algebraic domains, as shown by the following counterexample
communicated to the authors by M.~de Brecht. Let \( X = (\omega+1,
\leq) \) be endowed with the Scott topology (i.e.\ \( X =
L_{\omega+1} \) from Example~\ref{exdim2}(3)), and consider the
function \( f \colon X \to X \) defined by \( f(\omega) = \omega
\), \( f(2i) = 2i+1 \) and \( f(2i+1) = 2i \) (for every \( i \in
\omega \)): then \( f \in \mathsf{D}_2(X) \setminus
\mathsf{D}^\mathsf{W}_2(X) \). (A similar counterexample can of
course be given also for uncountable quasi-Polish spaces --- it is
enough to consider the space \( X \times \mathcal{N} \) and the
map \( (x,y) \mapsto (f(x),y) \), where \( f \) is the function defined above.)

The assumption that \( X \) be of dimension \( \neq \infty \) in the uncountable case of Corollary~\ref{corsem} is not a true limitation: for example, by using Theorem~\ref{theordim}(4), one gets that the corollary remains true also for \( X = \mathsf{UC}_\infty \). However, the general case remains unclear.

\begin{question}
Is it possible to further generalize Corollary~\ref{corsem} to all uncountable (quasi-)Polish spaces of dimension \( \infty \)?
\end{question}

\noindent Notice that if Corollary~\ref{corsem} holds true for a
space which is universal for all (quasi-)Polish spaces, then the
answer to the previous question is automatically positive.

Finally, the next corollary generalizes a recent result of Ostrovsky \cite[p.~663]{ost} concerning the possibility of representing Borel sets as countable unions of \( h \)-homogeneous \( \boldsymbol{\Pi}^0_2 \) sets.

\begin{corollary}
Let \( X \) be a quasi-Polish space of dimension \( \neq \infty \). Then every Borel set \( B \subseteq X \) can be partitioned into countably many \( h \)-homogeneous \( \mathbf{\Pi}^0_2 \) subspaces of \( B \).
\end{corollary}

\begin{proof}
By Theorem~\ref{theordim}(1), \( B \simeq^\mathsf{W}_{3} B' \) for some Borel \( B' \subseteq \mathcal{N} \), hence by Lemma~\ref{lemmapwhomeo} there are two countable partitions \( \langle B_n \mid n \in \omega \rangle \) and \( \langle B'_n \mid n \in \omega \rangle \) of, respectively, \( B \) and \( B' \) into \( \mathbf{\Delta}^0_3 \) pieces such that \( B_n \) is homeomorphic to \( B'_n \) for every \( n \in \omega \). Inspecting the proofs of the mentioned results, it is not hard to see that all of the \( B_n, B'_n \) can in fact be assumed to be in \( \mathbf{\Pi}^0_2 \). By \cite[Theorem 1]{ost}, each \( B'_n  \) can be partitioned into a countable union of \( h \)-homogeneous \( \mathbf{\Pi}^0_2(B'_n) \) (hence also \( \mathbf{\Pi}^0_2(B') \)) sets. Therefore, a similar decomposition can be obtained also for the \( B_n \)'s, and the union of these partitions gives the desired partition of \( B \).
\end{proof}

Theorem~\ref{theordim}(2)-(4) shows that the bound for the complexity of an isomorphism between \emph{arbitrary} (quasi-)Polish spaces obtained in Proposition~\ref{propgeneralhomeo} cannot be improved, but leaves open the problem of computing the minimal complexity of an isomorphism between two \emph{Polish} spaces with dimension \( \infty \). In this direction, we can make some basic observation.

\begin{proposition}\label{propdiminfinity}
Let \( X,Y \) be (quasi-)Polish spaces of dimension \( \infty \).
\begin{enumerate}[(1)]
\item
If \( X,Y \) are both universal for the class of all (quasi-)Polish spaces  then \( X \simeq^\mathsf{W}_3 Y \).
\item
The Hilbert cube \( [0,1]^\omega \) and the space \( \RR^\omega \) are both universal for Polish spaces, but \( [0,1]^\omega \not\simeq^\mathsf{W}_2 \RR^\omega \) (i.e.\ \( [0,1]^\omega \not\simeq_2 \RR^\omega \)):  hence the bound of part (1) cannot be improved.
\item
If \(X \) and \(Y \) are compact Hausdorff spaces and they are both universal for compact Polish spaces then \( X \simeq^\mathsf{W}_2 Y \).
\end{enumerate}
\end{proposition}

\begin{proof}
Parts (1) and (3) follow from Corollaries~\ref{corfrechet} and~\ref{corcompact}, respectively. Part (2) follows from Remark~\ref{remDW2}, the fact that the Hilbert cube \( [0,1]^\omega \) is
compact (by Tychonoff's theorem \cite[Proposition 4.1 (vi)]{ke94}),
and the fact that the space \( \RR^\omega \) is not
\(\sigma\)-compact. To see this, consider \( \bigcup_{n \in \omega}
K_n \subseteq \RR^\omega \) with all \( K_n \) compact: we will show
that \( \RR^\omega \setminus \bigcup_{n \in \omega} K_n \neq
\emptyset \). For every \( n \in \omega \), let \( \pi_n \colon
\RR^\omega \to \RR \) be the projection on the \( n \)-th
coordinate. Since \( \pi_n \) is continuous, \( \pi_n(K_n )
\subseteq \RR \) is compact, hence there is \( x_n \in \RR \setminus
\pi_n(K_n) \). Notice that every \( y = \langle y_k \mid k \in
\omega \rangle \in \RR^\omega \) with \( y_n = x_n \) does not
belong to \( K_n \): hence \( x = \langle x_n \mid n \in \omega
\rangle \in \RR^\omega \setminus \bigcup_{n \in \omega} K_n \), as
required.
\end{proof}

\begin{question}
Are there Polish spaces \( X,Y \) of dimension \( \infty \) such that \( X \not\simeq^\mathsf{W}_3 Y \)? Are there Polish spaces \( X,Y \) of dimension \( \infty \) such that \( X \not\simeq_n Y \) for every \( n \in \omega \)?
\end{question}

\section{Degree structures in quasi-Polish spaces} \label{sectionhierarchies}

In this section we will analyze the \( \F \)-hierarchies on \( X \) for various reducibilities \( \F \) and various quasi-Polish spaces \( X \). We start with some general ways to transfer results about the \( \F \)-hierarchy from one quasi-Polish space to another (for a fixed collection \( \F \)).

\begin{definition}
Let \( \F \) be a collection of functions between topological spaces. For \( X,Y \in \mathscr{X} \), denote by \( \F(X,Y) \) the collection of functions from \( \F \) with domain  \( X \) and range included in \( Y \) (obviously, \( Y \subseteq Y' \) implies \( \F(X,Y) \subseteq \F(X,Y') \)). The collection \( \F \) is said \emph{family of reducibilities} if the following conditions hold:
\begin{enumerate}[(1)]
\item
\( \F \) contains all the identity functions, i.e.\ \( \id_X \in \F(X,X) \) for every \(X \in \mathscr{X} \) (and hence \( \id_X \in \F(X,Y) \) for every \( X \subseteq Y \in \mathscr{X} \));
\item
\( \F \) is closed under composition, i.e.\ if \( X,Y,Z \in \mathscr{X} \), \( f \in \F(X,Y) \) and \( g \in \F(Y,Z) \) then \( g \circ f \in \F(X,Z) \);
\item
\( F \) is closed under restrictions, i.e.\ for every \( X,Y \in \mathscr{X} \), \( f \in \F(X,Y) \) and \( X' \subseteq X \), we have \( f \restriction X' \in \F(X',Y) \).
\end{enumerate}
\end{definition}

\noindent
For \( X \in \mathscr{X} \), we will abbreviate \( \F(X,X) \) with \( \F(X) \).
Notice that  if \( \F \) is a family of reducibilities then \( \F(X) \) is a reducibility on \( X \) for every \( X \in \mathscr{X} \) by conditions (1) and (2).
Notice that all the classes of functions considered in Section~\ref{sectionreduc}, namely
\begin{enumerate}[-]
\item
\( \bigcup_{X,Y \in \mathscr{X}} \mathsf{D}_\alpha(X,Y) \) for \( \alpha < \omega_1 \),
\item
\( \bigcup_{X,Y \in \mathscr{X}} \mathsf{D}^\mathsf{W}_\alpha(X,Y) \) for \( \alpha < \omega_1 \), and
\item
\( \bigcup_{X,Y \in \mathscr{X}} \mathcal{B}_\gamma(X,Y)  = \bigcup_{X,Y \in \mathscr{X}} \bigcup_{\beta < \gamma} \mathbf{\Sigma}^0_{\beta,1}(X,Y) \) for \( \gamma < \omega_1 \) additively closed,
\end{enumerate}
are in fact families of reducibilities. With a little abuse of notation, in what follows we will denote the above families of reducibilities by, respectively, \( \mathsf{D}_\alpha \), \( \mathsf{D}^\mathsf{W}_\alpha \) and \( \mathcal{B}_\gamma \) whenever this is not a source of confusion.

\begin{proposition}\label{propiso}
Let \( \F \) be a family of reducibilities and \( X,Y \in \mathscr{X} \).
If \( X\simeq_\F Y \) then \( (\mathscr{P}(X), \leq^X_\F) \) is isomorphic to \( (\mathscr{P}(Y), \leq^Y_\F) \) and, moreover, \( (S^X,\leq^{S,X}_\F) \) is isomorphic to \( (S^Y,\leq^{S,Y}_\F) \) for every set \( S \).
 \end{proposition}

\begin{proof}
It is enough to consider the more general case of \( X \)- and \(Y \)-namings, so let us fix a set \( S \). Let \( f \colon X\to Y \) be a bijection such that both \( f \in \F(X,Y) \) and \( f^{-1} \in \F(Y,X) \). Then \( \nu \mapsto \nu  \circ f^{-1} \) is an isomorphism between \( (S^X, \leq^{S,X}_\F) \) and \( (S^X,\leq^{S,Y}_\F) \). In fact, if \( \mu \leq^{S,X}_\F \nu \) via \( g \in \F(X) \) then \( \mu \circ  f^{-1} \leq^{S,Y}_\F \nu \circ f^{-1} \) via \( f \circ g \circ f^{-1} \), which is in \( \F(Y) \) since \( \F \) is closed under composition. Similarly, if \( \mu \circ f^{-1} \leq^{S,Y}_\F \nu \circ f^{-1} \) via some \( h \in \F(Y) \), then \( \mu \leq^{S,X}_\F \nu \) via \( f^{-1} \circ h \circ f \in \F(X) \).
\end{proof}

Notice that the converse to Proposition~\ref{propiso} does not hold in general. For example, it is not hard to check that (under \( \mathsf{AD} \)) the \( \mathsf{D}_2 \)-hierarchies on, respectively, \( \mathcal{N} \) and \( \mathcal{C} \) are isomorphic (in fact, they are isomorphic to the Wadge hierarchy on \( \mathcal{N} \)), while \( \mathcal{N} \not\simeq_{\mathsf{D}_2} \mathcal{C} \) by Proposition~\ref{propexamples}(2).

\begin{definition}
Let  \( \F \) be a family of reducibilities and \( X \in
\mathscr{X} \). We say that \( Y \subseteq X \) is an \emph{\( \F
\)-retract} of \( X \) if there is \( r \in \F(X,Y) \) such that \( r \restriction Y = \id_Y \) (such an
\( r \) will be called \emph{\( \F \)-retraction} of \( X \) onto
\( Y \)).
\end{definition}

The next fact extends \cite[Proposition 2.3]{s05}.

\begin{proposition}\label{propfr2}
Let \( \F \) be a family of reducibilities and \( X \in \mathscr{X} \). If \( Y \) is an \( \F \)-retract of \( X \), then there is an injection from \( (\mathscr{P}(Y), \leq_\F^Y) \) into \( (\mathscr{P}(X), \leq_\F^X) \). Similarly, for every set \( S \) there is an injection from \( (S^Y, \leq_\F^{S,Y}) \) into \( (S^X, \leq_\F^{S,X}) \).
 \end{proposition}

\begin{proof}
It is again enough to consider the more general case of \( X \)- and \( Y \)-namings. Let \( r \colon X \to Y \) be an  \( \F \)-retraction of \( X \) onto \( Y \). Then the map \( \nu \mapsto \nu \circ r \) is an injection of \( (S^Y, \leq_\F^{S,Y}) \) into \( (S^X, \leq_\F^{S,X}) \). In fact, if \( \mu \leq^{S,Y}_\F \nu \) via some \( g \in \F(Y) \), then it is easy to check that \( \mu \circ r \leq^{S,X}_\F \nu \) via \( g \circ r \), which is in \( \F(X) \) because \( \F \) is closed under composition. Conversely, if \( \mu \circ r \leq^{S,X}_\F \nu \circ r \) via some \( h \in \F(X) \) then \( \mu \leq^{S,Y}_\F \nu \) via \( r \circ (h \restriction Y) \), which is in \( \F(Y,Y) \) because \( \F \) is closed under composition and restrictions.
\end{proof}

\begin{corollary}\label{corfr2}
Let \( \emptyset \neq Y \subseteq X \) be quasi-Polish spaces, \( S \) be a set, and \( \F \) be a family of reducibilities. If \( \mathsf{D}^\mathsf{W}_3(X,Y) \subseteq \F \) then there is an injection from \( (\mathscr{P}(Y), \leq_\F^Y) \) (respectively, \( (S^Y, \leq_\F^{S,Y}) \)) into \( (\mathscr{P}(X), \leq_\F^X) \) (respectively, \( (S^Y, \leq_\F^{S,Y}) \)).

The same conclusion holds also if \( Y \in \mathbf{\Delta}^0_2(X) \) and \( \mathsf{D}^\mathsf{W}_2(X,Y) \subseteq \F \).
\end{corollary}

\begin{proof}
The map \( f \colon X \to Y \) such that \( f \restriction Y = \id_Y \) and \( f(x) = y_0 \) for \( x \in X \setminus Y \) (where \( y_0 \) is any fixed element of \( Y \)) is a \( \mathsf{D}^\mathsf{W}_3 \)-retraction of \( X \) onto \( Y \) because \( Y \in \mathbf{\Pi}^0_2(X) \) by Proposition~\ref{propsubspace}. If moreover \( Y \in \mathbf{\Delta}^0_2(X) \), then \( f \) is also a \( \mathsf{D}^\mathsf{W}_2(X,Y) \)-retraction of \( X \) onto \( Y \). Hence the result follows from Proposition~\ref{propfr2}.
\end{proof}

\begin{corollary}\label{coruniversalspaces}
Let \( \F \) be a family of reducibilities such that \( \bigcup_{X,Y \in \mathscr{X}} \mathsf{D}^\mathsf{W}_3(X,Y) \subseteq \F \). The \( \F \)-hierarchy on \( P \omega \) is (very) good if and only if for all quasi-Polish spaces \( X \), the \( \F \)-hierarchy on \( X \) is (very) good. Similarly, the \( \F \)-hierarchy on \( P \omega \) is (very) bad if and only if there exists a quasi-Polish space \( X \) such that the \( \F \)-hierarchy on \( X \) is (very) bad.

The same result holds when replacing \( P \omega \) with the Hilbert cube or with the space \( \RR^\omega \),  and letting \( X \) vary only on Polish spaces.
\end{corollary}

\begin{proof}
By Corollary~\ref{corfr2} and the fact that \( P \omega \) (respectively, \( [0,1]^\omega \) or \( \RR^\omega \)) is universal for the class of quasi-Polish (respectively, Polish) spaces.
\end{proof}

\begin{lemma} (Folklore)\label{lemmawqo}
Let \( \leq, \preceq \) be preorders on an arbitrary set \( A \) such that \( \preceq \) extends \( \leq \) (i.e.\ \( {\leq} \subseteq {\preceq} \)).
\begin{enumerate}[(1)]
\item
every antichain with respect to \( \preceq \) is an antichain with respect to \( \leq \);
\item
if \( \leq \) is a wqo then so is \( \preceq \).
\end{enumerate}
\end{lemma}

\begin{proof}
Part (1) is obvious, so let us consider just part (2). By (1), \( \preceq \) cannot contain infinite antichains. Assume towards a contradiction that \( a_0 \succeq a_1 \succeq \dotsc \) is an infinite (countable) \( \preceq \)-decreasing sequence of elements from \( A \). Clearly \( a_i \nleq a_j \) for every \( i < j \). Define a coloring \( c \colon [\omega]^2 \to \{ 0, 1 \} \) by setting (for \( i < j \in \omega \)) \( c(\{ i,j\}) = 0 \iff a_j \leq a_i \). By Ramsey's theorem, there is an infinite \( H \subseteq \omega \) such that \( c \restriction [H]^2 \) is constantly equal to either \( 0 \) or \( 1 \). In the first case the sequence \( \vec{a} =  \langle a_i  \mid i \in H  \rangle \) is an infinite \( \leq \)-descending chain, while in the second case \( \vec{a} \) is an infinite \( \leq \) antichain: therefore, in both cases we reach a contradiction with the fact that \( \leq \) is a wqo.
\end{proof}

\begin{proposition}  \label{propgreaterhierarchy}
Suppose \( \F \subseteq \G\)  are families of reducibilities and $X$ is quasi-Polish.
\begin{enumerate}[(1)]
\item
If \( (\mathscr{P}(X), \leq_\F^X) \) is very good, then so is  \( (\mathscr{P}(X), \leq_\G^X) \). Similarly, for every set  $S$, if \( (S^X, \leq_\F^X) \) is very good, then so is \( (S^X, \leq_\G^X) \). The same results hold when replacing very good with good.
\item
If the $\G$-hierarchy on $X$ is (very) bad, then the $\F$-hierarchy on $X$ is bad. Similarly, for every set $S\in\mathcal{P}(X)$, if \( (S^X, \leq_\G^X) \) is (very) bad, then \( (S^X, \leq_\F^X) \) is bad.
\end{enumerate}
\end{proposition}

\begin{proof}
Use Lemma~\ref{lemmawqo}.
\end{proof}

Notice that Proposition~\ref{propgreaterhierarchy}(2) cannot be strenghtened to the statement: ``If the $\G$-hierarchy on $X$ is very bad, then the $\F$-hierarchy on $X$ is very bad''. This is because it is possible  that every \( \leq_\G^X \)-descending chain of subsets of \( X \) in fact consists of \( \leq^X_\F \)-incomparable elements.

\begin{remark}
Notice that if the so-called Semi-linear Ordering Principle for \( \F \) on the space \( X \)
\begin{equation}
\tag{\( \mathsf{SLO}^{\F,X} \)}
\forall  A,B \subseteq X \, ( {A \leq_\F^X B }\vee{\overline{B} \leq_\F^X A})
 \end{equation}
holds (which is the case, under \( \mathsf{AD} \), for every \( \F \supseteq \mathsf{W} \) and \( X \) zero-dimensional Polish space), then the first part of Proposition~\ref{propgreaterhierarchy}(1) can be strengthened to the following:
\begin{enumerate}[-]
\item There is an injection from \( (\mathscr{P}(X), \leq_\G^X) \) into \( (\mathscr{P}(X), \leq_\F^X) \).
\end{enumerate}

This is because in this case \( (\mathcal{D}^X_\G,\leq) \) is a coarsification of   \( (\mathcal{D}^X_\F,\leq) \). To see this, it is enough to show that if \(A,B \subseteq X \) are such that \( A <_\G^X B \) then \( A <_\F^X B \). Clearly, \( B \nleq^X_\F A \), so it remains to show that \( A \leq^X_\F B \). Notice that \( \mathsf{SLO}^{\F,X} \Rightarrow \mathsf{SLO}^{\G,X} \), hence \( A <^X_\G \overline{B} \) as well. This implies that \( \overline{B} \nleq^X_\F A \), and hence \( A \leq^X_\F B \) by \( \mathsf{SLO}^{\F,X} \).
\end{remark}

\begin{corollary}\label{corpartitions}
Let \( \F \supseteq \mathsf{W} \) be a reducibility on \( \mathcal{N} \) and \( k \in \omega \). Then \( ((\boldsymbol{\Delta}^1_1(\mathcal{N}))_k , \leq_\F) \) is good (but not very good if \( k \geq 3 \)). Similarly, under \( \mathsf{AD} \) we have that   \( ((\mathscr{P}(\mathcal{N}))_k , \leq_\F) \) is good (but not very good if \( k \geq 3 \)).
\end{corollary}

\begin{proof}
The claim is well-known for \( \F = \mathsf{W} \) (see the discussion in Subsection~\ref{subsectionreducibilities}). For \( \mathsf{W} \subsetneq \F \), apply Proposition~\ref{propgreaterhierarchy}(1). To see that \( ((\boldsymbol{\Delta}^1_1(\mathcal{N}))_k , \leq_\F) \) is not very good when \( k \geq 3 \), notice that the \( k \)-partitions \( \nu_i \colon \mathcal{N} \to k \) (for \( i < k \)) defined by \( \nu_i(x) = i \) for every \( x \in \mathcal{N} \) are in fact clopen partitions and are pairwise \( \leq_\F \)-incomparable.
\end{proof}

\begin{remark}\label{remrestricted}
Notice that all the previous results hold ``locally'' i.e.\ when considering \( (\mathbf{\Gamma},\F ) \)-hierarchies  in place of \( \F \)-hierarchies, as long as \( \mathbf{\Gamma} \) is a family of pointclasses closed under \( \F \)-preimages, i.e.\ such that for every \( X,Y \in \mathscr{X} \) and \( f \in \F(X,Y) \), if \( A \in \mathbf{\Gamma}(Y) \) then \( f^{-1}(A) \in \mathbf{\Gamma}(X) \).
\end{remark}

\subsection{Degree structures in uncountable quasi-Polish spaces}\label{subsectionuncountable}

In this subsection we consider various reducibilities \( \F \) and study the \( \F \)-hierarchies on arbitrary uncountable quasi-Polish spaces \( X \). We begin with an important corollary to the results obtained in Section~\ref{sectionisom}.

\begin{theorem}  \label{theorcoruncountable}
Let \( X \) be an uncountable quasi-Polish space and \( \F \) be a family of reducibilities.
\begin{enumerate}[(1)]
\item
If $\dim(X)=0$, then the $\F$-hierarchy on $X$ can be embedded into the \( \F \)-hierarchy on \( \mathcal{N} \) for every \( \F \supseteq \mathsf{D}_1 = \mathsf{W} \). Hence the \( (\mathbf{B},\F) \)-hierarchy on \( X \) is very good, and assuming \( \mathsf{AD} \) the entire \( \F \)-hierarchy on \( X \) is very good.

\item
Assume $\dim(X)=0$ and that \( \F \supseteq \mathsf{D}^\mathsf{W}_2 = \mathsf{D}_2 \). If $X$ is $\sigma$-compact then  the \( \F \)-hierarchy on \( X \) is isomorphic to the \( \F \)-hierarchy on \( \mathcal{C} \), while if $X$ is not $\sigma$-compact then  the \( \F \)-hierarchy on \( X \) is isomorphic to the \( \F \)-hierarchy on \( \mathcal{N} \). Hence, if e.g.\ \( \F = \mathsf{D}_\alpha \) for some \( \alpha \geq 2 \), then the \( (\mathbf{B},\F) \)-hierarchy on \( X \) is isomorphic to the \( (\mathbf{B}, \mathsf{D}_1 ) \)-hierarchy on \( \mathcal{N} \), and assuming \( \mathsf{AD} \) the entire \( \F \)-hierarchy on \( X \) is isomorphic to the Wadge  hierarchy on \( \mathcal{N} \).

\item
If $\dim(X)\neq\infty$ then  the \( \F \)-hierarchy on \( X \) is isomorphic to the \( \F \)-hierarchy on \( \mathcal{N} \) whenever \( \F \supseteq \mathsf{D}^\mathsf{W}_3 \). Hence the \( (\mathbf{B}, \F ) \)-hierarchy on \( X \) is very good, and assuming \( \mathsf{AD} \) the \( \F \)-hierarchy on \( X \) is very good as well. Moreover, if e.g.\ \( \F = \mathsf{D}_\alpha \) for some \( \alpha \geq 3 \), then the \( \F \)-hierarchy on \( X \) is isomorphic to the Wadge  hierarchy on \( \mathcal{N} \).

\item
If \( X \) is universal for Polish (respectively, quasi-Polish) spaces and \( \F \supseteq \mathsf{D}^\mathsf{W}_3 \), then the \( \F \)-hierarchy on \( X \) is isomorphic to the \( \F \)-hierarchy on \( [0,1]^\omega \) (respectively, on \( P \omega \)). Moreover, the \( \F \)-hierarchy on \( X \) is (very) good if and only if the \( \F \)-hierarchy on every Polish (respectively, quasi-Polish) space  is (very) good, and it is (very) bad if and only if the  \( \F \)-hierarchy on some Polish (respectively, quasi-Polish) space  is (very) bad.

\item
If \( \F \supseteq \mathcal{B}_\omega \) (hence, in particular,  if \( \F = \mathsf{D}_\alpha \) for some \( \alpha \geq \omega \)) then the \( \F \)-hierarchy on \( X \) is isomorphic to the \( \F \)-hierarchy on \( \mathcal{N} \). Hence the \( (\mathbf{B}, \F ) \)-hierarchy on \( X \) is very good, and assuming \( \mathsf{AD} \) the \( \F \)-hierarchy on \( X \) is very good as well. In fact, if e.g.\ \( \F = \mathsf{D}_\alpha \) for some \( \alpha \geq \omega \), then the \( \F \)-hierarchy on \( X \) is isomorphic to the Wadge  hierarchy on \( \mathcal{N} \).
\end{enumerate}
Analogous results hold for \( k \)-partitions  of \( X \) (for every \( k \in \omega \)) when replacing ``very good'' by ``good'' in all the statements above.
\end{theorem}

\begin{proof}
Let us first consider the first item of the list. Without loss of generality we can assume
that \( X \) is a closed subset of \(\mathcal{N} \) by
\cite[Theorem 7.8]{ke94}, and hence that \( X \) is a (\(
\mathsf{W} \)-)retract of \(\mathcal{N} \) by \cite[Theorem
7.3]{ke94}. Hence the claim follows from Proposition~\ref{propfr2} and the results from~\cite{wad84}.

The other claims of the list follow from Proposition~\ref{propiso} and, respectively,  Lemma~\ref{lemma0dim} and the fact that every zero-dimensional quasi-Polish space is Polish (see the proof of Theorem~\ref{theordim}), Theorem~\ref{theordim}(1), Proposition~\ref{propdiminfinity} and Corollary~\ref{coruniversalspaces}, and Proposition~\ref{propgeneralhomeo}, together with the results from~\cite{ros09, motbaire}.

The results about \( k \)-partitions can be obtained in a similar way using Corollary~\ref{corpartitions}.
\end{proof}

Theorem~\ref{theorcoruncountable} leaves open the problem of determining the \( \F \)-hierarchy on \( X \) for many reducibilities \( \F \), most notably for \( \F = \mathsf{D_1} = \mathsf{W} \) and \( \F = \mathsf{D}_2 \) on quasi-Polish spaces \( X \) of dimension \( \neq \infty \), and for \( \F = \mathsf{D}_n \), \( 1 \leq n \in \omega \), for quasi-Polish spaces of dimension \( \infty \): in the rest of this subsection we will give some partial answers to this problem.

Let us first consider the $\mathsf{D}_1$- and $\mathsf{D}_2$-hierarchies on  uncountable Polish spaces. The case of the \( \mathsf{D}_1 \)-hierarchies is now quite well-understood. For example, we have the following results.

\begin{theorem} \cite{he96} The $\mathsf{D}_1$-hierarchy on $\mathbb{R}^n$ and $[0,1]^n$ (for $1\leq n\leq \omega$) is very bad.
\end{theorem}

\begin{theorem} \cite{ik10,sc10}
Let $X=\mathbb{R}^n$ for $1\leq n\leq \omega$.
\begin{enumerate}[(1)]
\item Any countable partial order can be embedded into the $\mathsf{D}_1$-hierarchy on $X$.
\item ($\mathsf{ZFC}$) Any partial order of size $\omega_1$ can be embedded into the $\mathsf{D}_1$-hierarchy on $X$.
\end{enumerate}
\end{theorem}

\noindent
Notice that by Proposition~\ref{propfr2} these results hold also for Polish spaces  admitting a continuous retraction to $\mathbb{R}^n$,  like e.g.\ \( [0,1]^n\).

\begin{theorem} \cite{sc11} \label{theorbad}
Suppose $X$ is a metric space with $\dim(X)\neq0$: then the $\mathsf{D}_1$-hierarchy on $X$ is bad. In fact it contains uncountable antichains.
\end{theorem}

Notice that Theorem~\ref{theorbad} cannot be improved by replacing
bad with very bad, as by \cite[Theorem 11]{cook}  there is an
uncountable connected compact Polish space \( X \) with the property
that  all continuous maps \( f \colon X \to X \) are either constant
or the identity map; this implies that all nonempty subsets $A \neq
X$ are Wadge incomparable, and hence that the Wadge hierarchy on \(
X \) is bad but not very bad. Moreover, we cannot require \( X \) to
be just quasi-Polish: for example, the Wadge hierarchy on the
perfect \(\omega\)-algebraic domain \( L_{\omega+1} \) from Example~\ref{exdim2}(3), which has dimension \(\omega\), is good (but not
very good), hence it does not contain infinite antichains.

Let us now consider the $\mathsf{D}_2$-hierarchy on locally
connected Polish spaces. Notice that for any $f\in\mathsf{D}_2(X,Y)$
between Polish spaces $X,Y$ there is a nonempty open set $U\subseteq
X$ such that  the restriction of $f$ to the closure \(
\mathrm{cl}(U) \) of $U$ is continuous. In fact, the Jayne-Rogers
theorem \( \mathsf{D}_2(X,Y) = \mathsf{D}^\mathsf{W}_2(X,Y) \)
implies that for such an \( f \) there is a closed covering $\langle
X_k \mid k \in \omega \rangle$  of $X$ such that $f \restriction
X_k$ is continuous for each $k \in \omega$. By Baire's category

theorem there is \( k \in \omega \) such that $X_k$ is not meager,
hence $U \subseteq X_k$ for some nonempty open set \( U \). Since \(
X_k \) is closed, it follows that \( \mathrm{cl}(U) \subseteq X_k
\), and hence \( f \restriction \mathrm{cl}(   U) = (f \restriction
X_k) \restriction \mathrm{cl}(U) \) is continuous.

\begin{proposition} \label{propnotverygood}
Suppose $X$ is an  uncountable locally connected Polish space. Then
the $(\mathbf{B}, \mathsf{D}_2)$-hierarchy on $X$ is not very good.
\end{proposition}

\begin{proof}
We will find Borel sets $A,B\subseteq X$ such that $\{A,
\overline{A}, B, \overline{B} \}$ is an antichain with respect to
$\mathsf{D}_2(X)$-reducibility. Since $X$ is uncountable, there is a
compact set $Y\subseteq X$ with $Y \simeq_1 \mathcal{C}$ by
\cite[Theorem 13.6]{ke94}. Let $A$ be a proper ${\bf\Sigma}^0_2(X)$
set such that both $A$ and $\overline{A}$ are dense in $X$, and let
$B \subseteq Y$ be a proper ${\bf\Sigma}^0_3(Y)$ (so that, in
particular, \( B \) is also a proper \( \boldsymbol{\Sigma}^0_3(X)
\) set). By their topological complexity (and the fact that all the
pointclasses \( \boldsymbol{\Sigma}^0_2(X), \boldsymbol{\Pi}^0_2(X),
\boldsymbol{\Sigma}^0_3(X), \boldsymbol{\Pi}^0_3(X) \) are closed
under \( \mathsf{D}_2 \)-preimages), we have that the pairs \( (A,
\overline{A} ) \) and \( ( B, \overline{B} ) \) are \( \mathsf{D}_2
\)-incomparable, and that \( B, \overline{B} \nleq_{\mathsf{D}_2} A,
\overline{A} \). Hence it remains to show that \( A, \overline{A}
\not\leq_{\mathsf{D}_2} B , \overline{B} \). In fact, since \( D_0
\leq_{\mathsf{D}_2} D_1 \iff \overline{D_0} \leq_{\mathsf{D}_2}
\overline{D_1} \) for every \( D_0, D_1 \subseteq X \), it suffices
to show that \( C \nleq_{\mathsf{D}_2} B \) for every \( C \in \{ A,
\overline{A} \} \). Suppose toward a contradiction that there is a
reduction $f\in\mathsf{D}_2(X)$ of $C$ to $B$.
Then by the observation preceding this proposition, there is a nonempty open connected set $U\subseteq X$ such that $f\restriction U$ is continuous. Since $C\cap U$ is dense in $U$, its image $f(C\cap U)$ is also dense in $f(U)$. This implies $f(U)\subseteq Y$, since $f(C\cap U) \subseteq B \subseteq Y$ and $Y$ is closed. Since $f(U)$ is connected and $Y$ is totally disconnected, $f\restriction U$ is constant. But this contradicts our assumption that $C=f^{-1}(B)$ because both \( C \) and \( \overline{C} \) are dense and hence have nonempty intersection with \( U \).
\end{proof}

Using Corollary~\ref{corfr2}, we have also the following result.

\begin{corollary}
Suppose that $X$ is quasi-Polish and that there is $Y\in{\bf\Delta}^0_2(X)$ which is an uncountable locally connected Polish space. Then the $(\mathsf{B},\mathsf{D}_2)$-hierarchy on $X$ is not very good.
\end{corollary}

We will now turn our attention to Euclidean spaces and show that their $(\mathbf{B},\mathsf{D}_2)$-hierarchy is in fact bad. In what follows, we will crucially use two simple properties of the real line \( \RR \), namely the fact that, since \( \RR \) is connected, every continuous function \( f \colon \RR \to \RR \) maps intervals to (possibly degenerate) intervals, and the fact that \( \RR \) is $\sigma$-compact. Recall that
in an arbitrary Polish space \( X \) the image of a closed subset of \( X \) under a continuous reduction can be analytic non-Borel: however, if \( X \) is $\sigma$-compact the situation becomes considerably simpler, as it is shown by the next lemma.

\begin{lemma} \label{lemsaintraymond}
Let $2 \leq \alpha < \omega_1$. Suppose $X$ and $Y$ are Polish spaces, $f \colon X\rightarrow Y$ is surjective, and $A\subseteq Y$.
\begin{enumerate}[(1)]
\item \cite[Theorem 5]{sr76} If $X$ is compact, $f$ is continuous, and $f^{-1}(A)\in {\bf\Sigma}^0_\alpha(X)$, then $A\in{\bf\Sigma}^0_\alpha(Y)$.
\item If $X$ is $\sigma$-compact, $f\in\mathsf{D}_2(X,Y)$, and $f^{-1}(A)\in {\bf\Sigma}^0_\alpha(X)$, then $A\in{\bf\Sigma}^0_\alpha(Y)$.
\end{enumerate}
\end{lemma}

\begin{proof} The first part is proved in~\cite{sr76}. For the second part, if $X$ is $\sigma$-compact then by the Jayne-Rogers theorem \( \mathsf{D}_2(X,Y) = \mathsf{D}^\mathsf{W}_2(X,Y) \) there is a countable covering $\langle X_k \mid k\in\omega \rangle$ of $X$ consisting of compact sets such that $f \restriction X_k$ is continuous for each $k\in\omega$.
Since $(f\restriction X_k)^{-1}(A\cap f(X_k))= f^{-1}(A)\cap X_k\in {\bf\Sigma}^0_\alpha(X)$, the first part applied to \( f \restriction X_k \colon X_k \to f(X_k) \) implies that $A\cap f(X_k)\in {\bf\Sigma}^0_\alpha(f(X_k)) \subseteq \boldsymbol{\Sigma}^0_\alpha(Y)$ (because \( f(X_k) \) is compact, and hence closed in \( Y \)). Hence $A = \bigcup_{k \in \omega} (A \cap f(X_k)) \in {\bf\Sigma}^0_\alpha(Y)$.
\end{proof}

To construct antichains in the $(\mathbf{B},\mathsf{D}_2)$-hierarchies of a Polish spaces \( X \), we will consider sets which are everywhere proper $ {\bf\Sigma}^0_\alpha(X)$ for some $1< \alpha < \omega_1$.

\begin{lemma} \label{lemmaeverywhere}
Suppose $X$ is a perfect Polish space and $1< \alpha < \omega_1$. Then there is $A\in {\bf\Sigma}^0_\alpha(X)$ such that for all nonempty open sets $U\subseteq X$, $A\cap U \in {\bf\Sigma}^0_\alpha(X) \setminus \mathbf{\Pi}^0_\alpha(X)$.
\end{lemma}

\begin{proof}
Let \( d \) be a compatible metric for \( X \).
We first construct a sequence $\langle C_n \mid n\in\omega \rangle$ of disjoint nowhere dense closed subsets of $X$
with $\lim_{n \to \infty} \diam(C_n)=0$ (where the operator \( \diam \) refers to the chosen metric \( d \)) and such that $\bigcup_{n\in\omega} C_n$ is dense. Let \( \langle B_n \mid n \in \omega \rangle \) be an enumeration of a basis for the topology of \( X \). Observe that for every nonempty open set \( U \subseteq X \) there is a homeomorphic copy \( C \subseteq U \) of \( \mathcal{C} \) which is nowhere dense (in \( X \)), and that \( C \) is necessarily closed since \( \mathcal{C} \) is compact and \( X \) is Hausdorff. By induction on \( n \in \omega \), choose (using \( \mathsf{DC} \)) a closed nowhere dense \( C_n \subseteq B'_n = B_n \setminus \bigcup_{i < n} C_i \) such that \( \diam(C_n) \leq 2^{-n} \): such a \( C_n \) exists by the observation above since, by the inductive hypothesis applied to the \( C_i \)'s, \( B'_n \) is a nonempty open set, and  if necessary \( B'_n \) can obviously be further shrunk to a nonempty open set of diameter \( \leq 2^{-n} \). It is straightforward to check that the sequence of the \( C_n \)'s constructed in this way has the desired properties.

Now choose sets $A_n \in {\bf\Sigma}^0_\alpha(C_n) \setminus {\bf\Pi}^0_\alpha(C_n)$ (using \( \mathsf{DC} \) again). We claim that $A=\bigcup_{n\in\omega} A_n$ is as required. First notice that each \( A_n \in \boldsymbol{\Sigma}^0_\alpha(X) \) since \( \alpha \geq 2 \) and \( C_n \) is closed, hence \( A \in \boldsymbol{\Sigma}^0_\alpha(X) \) as well. Now assume towards a contradiction that there is an open \( U \subseteq X \) such that \( A \cap U \in \boldsymbol{\Pi}^0_\alpha(X) \). Let \( x \in X \) and \( 0 < \varepsilon \in \RR^+ \) be such that \( B_d(x, \varepsilon) \subseteq U \). Let \( N \in \omega \) be such that  \( \diam(C_n) < \frac{\varepsilon}{2} \) for every \( n \geq N \). Since each \( C_n \) was assumed nowhere dense, \( V = B_d(x, \frac{\varepsilon}{2}) \setminus \bigcup_{n < N} C_n \) is a nonempty open set. By density of \( \bigcup_{n \in \omega} C_n \) and the definition of \( V \), there is \( n \geq N \) such that \( C_n \cap V \neq \emptyset \). Hence, by the choice of \( N \) and \( V \) we have \( C_n \subseteq U  \). By the assumption \( A \cap U \in \boldsymbol{\Pi}^0_\alpha(X) \), \( A_n = A \cap C_n  = (A \cap U) \cap C_n\in \boldsymbol{\Pi}^0_\alpha(C_n) \), contradicting the choice of the \( A_n \)'s.
\end{proof}

\begin{proposition}
There are uncountable antichains in the $\mathsf{D}_2$-hierarchy on $[0,1]$. \end{proposition}

\begin{proof}
Let $W_\alpha\subseteq\mathcal{C}$ denote the set of codes of well-orders on $\omega$ of order type at most $\alpha$. Let $W=\bigcup_{\alpha<\omega_1} W_\alpha$. Since each $W_\alpha$ is Borel and $W$ is $\boldsymbol{\Pi}^1_1$-complete, the Borel ranks of the sets $W_\alpha$ are unbounded in $\omega_1$. Hence we obtain an unbounded set $C\subseteq\omega_1$ and a collection $(A_\alpha)_{\alpha\in C}$ of subsets of $\mathcal{C}$ such that $A_\alpha$ is $\boldsymbol{\Sigma}^0_\alpha(\mathcal{C})$-complete.
It follows from the previous Lemma (using $\mathsf{DC}$) that there is a collection $(A_\alpha)_{\alpha\in C}$ of Borel subsets of $[0,1]$ such that $A_\alpha\cap (a,b)$ is a proper ${\bf\Sigma}^0_\alpha([0,1])$ set for all nondegenerate open intervals $(a,b)\subseteq [0,1]$.

Suppose $\alpha,\beta\in C$ and $1<\alpha<\beta$.
Then \( A_\beta \nleq_{\mathsf{D}_2} A_\alpha \) because the pointclass \( \boldsymbol{\Sigma}^0_\alpha(X) \) is closed under \( \mathsf{D}_2 \)-preimages. Conversely, assume towards a contradiction  that there is a reduction $f\in\mathsf{D}_2([0,1])$ of $A_\alpha$ to $A_\beta$.
Let $[a,b]$ be a nondegenerate closed interval such that $f\restriction[a,b]$ is continuous (which exists by the observation preceding Proposition~\ref{propnotverygood}). Then there is an interval $[c,d]$ with $f([a,b])= [c,d]$, and since $f\restriction [a,b]$ cannot be constant (as otherwise either \( A_\alpha \cap [a,b] = [a,b] \) or \( A_\alpha \cap [a,b] = \emptyset \), contradicting the choice of the \( A_\alpha \)'s), $[c,d]$ is nondegenerate. Since $(f\restriction [a,b])^{-1} [A_\beta\cap [c,d]]=A_\alpha\cap [a,b]\in{\bf\Sigma}^0_\alpha([0,1])$ and $A_\beta\cap [c,d]\notin \boldsymbol{\Sigma}^0_\alpha([0,1]) \subseteq {\bf\Pi}^0_\beta([0,1])$, this contradicts Lemma~\ref{lemsaintraymond}.
\end{proof}

\begin{corollary}\label{corD2bad}
\begin{enumerate}[(1)]
\item
Suppose $X$ is an Hausdorff quasi-Polish space with $[0,1]\subseteq X$. Then there are uncountable antichains in the $\mathsf{D}_2$-hierarchy on $X$.
\item
The $\mathsf{D}_2$-hierarchy on $\mathbb{R}^n$ and $[0,1]^n$ (for $1\leq n\leq \omega$) is bad.
\end{enumerate}
\end{corollary}

\begin{proof}
This follows from the fact that a compact subset of an Hausdorff quasi-Polish  space is closed and from Corollary~\ref{corfr2}.
\end{proof}

Corollary~\ref{corD2bad} leaves open the following problem:

\begin{question}
Is the $\mathsf{D}_2$-hierarchy on $\mathbb{R}$ \emph{very} bad?
\end{question}

We do not know the answer to the above question, but we are at least able to show that
the $(\mathbf{B}, \mathsf{D}_2)$-hierarchy on $\RR^2$ (and hence, since \( \RR^2 \) is \(\sigma\)-compact, on any Hausdorff quasi-Polish space containing \( \RR^2 \), like the spaces \( [0,1]^n \) and \( \RR^n \) for \( 2 \leq n \leq \omega \)) is very bad. Unfortunately, our argument cannot be adapted to $\mathbb{R}$.

\begin{theorem} \label{theorR2}
The quasi-order $(\mathscr{P}(\omega),\subseteq^*)$ of inclusion modulo finite sets on $\mathscr{P}(\omega)$ embeds into $(\boldsymbol{\Sigma}^0_2(\mathbb{R}^2),\leq_{\mathsf{D}_2})$.
\end{theorem}

Before proving Theorem~\ref{theorR2}, we present some constructions and prove some technical lemmas which will be needed later. Call a map $f \colon \RR \to \RR$ \emph{weakly increasing} if $a\leq b$ implies $f(a)\leq f(b)$ for all $a,b\in\RR$.
Let $\{ A_x \mid x \subseteq \omega \}$ be a collection of ${\bf\Sigma}^0_2$ subsets of $[0,1]$ with the following properties:
\begin{enumerate}[({A}1)]
\item
\( \inf(A_x) =0 \), \( \sup (A_x ) = 1 \), and both \( A_x \cap (0,1) \) and \( \overline{A_x} \cap (0,1) \) are nonempty;
\item $x\subseteq^* y$ if and only if $A_x\leq_{\mathsf{D}_1}^\mathbb{R} A_y$;
\item If $x\subseteq^* y$, then there is a weakly increasing uniformly continuous surjective reduction $f \colon \RR \rightarrow \RR$ of $A_x$ to $A_y$ with $f(0)=0$ and $f(1)=1$.
\end{enumerate}
Such a collection exists by~\cite{sc10}.

\begin{definition}
A map \( h \colon \RR^2 \to \RR^2 \) is called \emph{special} if it is uniformly continuous, surjective, and such that for every \( a,a',b,b' \in \RR \)
\end{definition}
\begin{enumerate}[({S}1) ]
\item
if \( a = a' \) then \( \pi_0(h(a,b)) = \pi_0(h(a',b')) \),
\item
 \( a \leq a' \iff \pi_0(h(a,b)) \leq \pi_0(h(a',b')) \), and
\item
if \(b\leq b'\) then \(\pi_1(h(a,b))\leq \pi_1(h(a,b')).\)
\end{enumerate}

For \( a \in \RR \), let $\RR_a$ denote $\{ a \} \times \RR$.
Notice that for every map \( h \colon \RR^2 \to \RR^2 \) satisfying (S1) and (S2)   there is an order-preserving (hence injective) map \( \pi(h) \colon \RR \to \RR \) such that \( h(\RR_a) \subseteq \RR_{\pi(h)(a)} \).

Let $\langle \beta_n\mid n\in\omega \rangle$ be a decreasing sequence of reals with $\lim_{n\rightarrow\infty}\beta_n=0$ and $\beta_0<\frac{1}{12}$.
Let $\{q_n \mid n\in\omega\}$ be an enumeration without repetitions of $\QQ$, and
let $\{p_n\mid n\in\omega\}$ be a collection in $\QQ$ such that \( \{ (q_n, p_n) \in \QQ^2 \mid n \in \omega  \} \) is dense in \( \RR^2 \).
For $x\subseteq\omega$ and $n\in\omega$,
let $B_{x,n} \subseteq\mathbb{R}$ be an affine image of $A_x$ (with the same orientation as $A_x$) such that $\inf(B_{x,n})=p_n$ and $\sup(B_{x,n})=p_n+\beta_n$.
Let $I$ be a countable dense subset of $\mathbb{R} \setminus \mathbb{Q}$ (so that \( \RR \setminus (\QQ \cup I) \) is dense in \( \RR \) as well),
and let $\{i_k \mid k\in\omega\}$ be an enumeration without repetitions of $I$.
Finally, define
\begin{equation} \label{eqCx}
\tag{{$*$}} C_x= (I\times \mathbb{R}) \cup \bigcup_{\substack{n\in\omega }} (\{q_n\}\times B_{x,n})\subseteq\RR^2.
\end{equation}
Thus \( C_x \cap \RR_a \) is empty if \( a \in \RR \setminus (I \cup \QQ) \),   equals \( \RR_a \) if \( a \in I \), and is an affine copy of \( A_x \) if \( a \in \QQ \).

In what follows we will tacitly assume that every \emph{partial} function $f \colon \RR^2 \rightarrow \RR^2$ is such that \( \dom(f) = D_f \times \RR \) and \( \range(f) = R_f \times \RR \) for some  (possibly empty) \( D_f,R_f \subseteq \RR \), and moreover that it satisfies (S1)--(S3) when restricted to \( a, a' \in D_f \). Note that necessarily \( D_f \) and \( R_f \)  have the same cardinality. Given such an \( f \), we set \( f_a \colon \RR \to \RR \colon b \mapsto \pi_1(f(a,b)) \) for every \( a \in D_f \).

 Given a partial function \( f \colon \RR^2 \to \RR^2 \) with finite \( D_f \), we
define a canonical extension $\bar{f} \colon \RR^2\rightarrow\RR^2$ of \( f \) as
follows.   If $D_f =\emptyset$, let $\bar{f}= \id_{\RR^2}$. Suppose now
that $D_f =\{a_0, \dotsc ,a_n\}$ with $a_0<a_1< \dotsc <a_n$ (for some \( n < \omega \)), and pick \( (a,b) \in \RR^2 \). Then we set set
$\bar{f}(a,b)=(a,f_{a_0}(b))$ if $a<a_0$,
$\bar{f}(a,b)=(a,f_{a_n}(b))$ if $a>a_n$, and
\begin{equation} \label{eqextension}
\tag{$\dagger$} \bar{f}(a,b)=\frac{a_{m+1}-a}{a_{m+1}-a_m} \cdot f(a_m,b)+ \frac{a-a_m}{a_{m+1}-a_m} \cdot f(a_{m+1},b)
\end{equation}
if $a_m<a<a_{m+1}$ for some \( m < n \), where \( + \) and \( \cdot\) denote the usual operations of vector addition and multiplication by a scalar on \( \RR^2 \); in other words, \( \bar{f}(a,b) \) is the linear combination of \( f(a_m,b) \) and \( f(a_{m+1},b) \).

Let \( d \) denote the usual Euclidean distance on \( \RR^2 \), and for \( f,g \colon \RR^2 \to \RR^2 \) write
 \[
\Vert f - g \Vert=\sup\{d( f(a,b),g(a,b)) \mid  a,b\in\RR\}.
\]

\begin{definition}
Given a (partial) function \( f \colon D_f \times \RR \to R_f \times \RR \), we say that $(\delta,\varepsilon)\in(\RR^+)^2$ is a \emph{modulus} (of uniform continuity) for $f$ if
\[
\forall a \in D_f \, \forall b,b' \in \RR \, (|b-b'|<\delta \Rightarrow
d( f(a,b),f(a,b'))  <\varepsilon).
\]
\end{definition}

Notice that in fact \( d( f(a,b),f(a,b')) = |f_a(b) - f_a(b') | \) by our assumptions on \( f \), and that if $(\delta,\varepsilon)$ is a modulus for a partial function $f \colon D_f \times \RR \to R_f \times \RR$ with finite \( D_f \), then $(\delta,\varepsilon)$ is also a modulus for $\bar{f}$.

\begin{definition} \label{defY}
We let $\mathcal{Y}$ be the the class of all partial functions $f \colon D_f \times \RR \rightarrow R_f \times \RR$ such that
\end{definition}
\begin{enumerate}[(Y1)]
\item
\( D_f \) is finite;
\item
\( f \) satisfies (S1)--(S3) restricted to \( a,a' \in D_f \);
\item
$f$ is uniformly continuous (by (Y1) this is equivalent to require that \( f_a \) is uniformly continuous for every \( a \in D_f \));
\item
there is \( K \in \omega \) such that \( | f_a(b) - f_{a'}(b) | < K \) for every \( a,a' \in D_f \) and \( b \in \RR \);
\item
$f$ is a partial reduction of $C_x$ to $C_y$, i.e.\ for every \( a \in D_f \) and every \( b \in \RR \), \( (a,b) \in C_x \iff f(a,b) \in C_y \).
\end{enumerate}

\begin{lemma} \label{lemmaextensionspecial}
For every \( f \in \mathcal{Y} \), \( \bar{f} \) is special.
\end{lemma}

\begin{proof}
All the computations in the proof rely on the very specific definition~\eqref{eqextension} of \( \bar{f} \).
It is not hard to check that \( \bar{f} \) is surjective since \( \range(f) = R_f \times \RR \), and that \( \bar{f} \) satisfies (S1)--(S3) by (Y2), so that \( \pi(\bar{f}) \) is well-defined. It remains to check that \( \bar{f} \) is uniformly continuous.
Fix \( K \in \omega \) as in (Y4), and notice that from this property we get \( |\bar{f}_a(b) - \bar{f}_{a'}(b)| < \frac{K}{\rho} \cdot |a-a'| \) for every \( a, a',b \in \RR \), where \( \rho = \min \{ |a - a' | \mid a,a' \in D_f \} \). Given \( \varepsilon > 0 \), let \( \delta' \) be such that \( ( \delta', \frac{\varepsilon}{2} ) \) is a modulus for \( f \) (such a \(\delta'\) exists by (Y3)), and \( \delta'' \) be such that \( \frac{K}{\rho} \cdot |a-a'| < \frac{\varepsilon}{4} \) and \( |\pi(\bar{f})(a)-\pi(\bar{f})(a')| < \frac{\varepsilon}{4} \) whenever \( |a-a'| < \delta'' \) (such a \( \delta'' \) exists by (Y1)). Finally, let \( \delta = \min \{ \delta',\delta'' \} \). Fix \( (a,b) \in \RR^2 \), and let \( (a',b' ) \in \RR^2 \) be such that \( d((a,b),(a',b')) < \delta \), so that in particular \( |a-a'| < \delta'' \) and \( |b-b'| < \delta' \). Since \( (\delta',\frac{\varepsilon}{2}) \) is a modulus for \( \bar{f} \) as well, we have that
\begin{multline*}
d(\bar{f}(a,b) , \bar{f}(a',b')) \leq d(\bar{f}(a,b),\bar{f}(a',b)) + d(\bar{f}(a',b),\bar{f}(a',b')) \\ \leq (| \pi(\bar{f})(a)- \pi(\bar{f})(a')| + |\bar{f}_a(b) - \bar{f}_{a'}(b)|) + |\bar{f}_{a'}(b) - \bar{f}_{a'}(b')| < \left(\frac{\varepsilon}{4} + \frac{\varepsilon}{4} \right)+\frac{\varepsilon}{2} = \varepsilon.
 \end{multline*}
Therefore \( \bar{f} \) is uniformly continuous, and hence special.
\end{proof}

\begin{lemma} \label{lemmaextension}
Suppose that $f\in\mathcal{Y}$ and that $(\delta,\varepsilon)$ and $(\delta',\varepsilon')$ are moduli for $f$. Let $\{k_i \mid i<M\}\subseteq\omega$ and $\{l_j \mid j<N\}\subseteq\omega$ be such that $k_i \leq k_{i'} \) for \( i \leq i' < M \), and assume that $\beta_{k_0}<\min\{ \delta, \varepsilon,\varepsilon'\}$. Then there is an extension $g\in\mathcal{Y}$ of $f$ such that
\begin{enumerate}[(1)]
\item
$\{ q_{k_i}, i_{l_j} \mid i < M , j < N \}\subseteq D_g \cap R_g$;
\item
$(\delta',3\varepsilon')$ is a modulus  for $g$;
\item
$\Vert \bar{f}-\bar{g}\Vert \leq 2\varepsilon$.
\end{enumerate}
\end{lemma}

\begin{proof}
We will define \( g \) through some intermediate extensions \( f \subseteq g' \subseteq g'' \subseteq g''' \subseteq g \).
Let
\[
H = \{ q_{k_i}, \pi(\bar{f})(q_{k_i}), \pi(\bar{f})^{-1}(q_{k_i}), i_{l_j} , \pi(\bar{f})(i_{l_j}), \pi(\bar{f})^{-1}(i_{l_j}) \mid i < M, j<N \},
\]
and let $\theta=\frac{1}{2}\min\{d(a,a')\colon a,a'\in H,\ a\neq a'\}$.

First we extend $f$ to a partial function \( g' \) in such a way that \( D_{g'} = D_f \cup  \{ q_{k_i} \mid i<M \}\). Fix \( i < M \): if $\bar{f} \restriction \RR_{q_{k_i}}$ is already a partial reduction of $C_x$ to $C_y$, we let $g' \restriction \RR_{q_{k_i}}=\bar{f} \restriction \RR_{q_{k_i}}$.
Otherwise, to simplify the notation let $\beta=\beta_{k_i}$, $a_0= p_{k_i} = \inf(B_{x,k_i})$, and $a_1= p_{k_i}+\beta_{k_i} = \sup(B_{x,k_i})$. Let $l \in \omega$ be such that
\begin{enumerate}[- ]
\item $q_l\leq \pi(\bar{f})(q_{k_i})$,
\item $|\pi(\bar{f})(q_{k_i})-q_l|<\min\{\varepsilon,\theta\}$,
\item $\beta_l<\varepsilon'$, and
\item $\bar{f}_{q_{k_i}}(a_0)< p_l <p_l + \beta_l <\bar{f}_{q_{k_i}}(a_1)$.
\end{enumerate}
Notice that such an \( l \) exists by the choice of the \( \beta_n \)'s and the \( p_n \)'s.
Then we define $g' \restriction \RR_{q_{k_i}}$ from \( \RR_{q_{k_i}} \) onto \( \RR_{q_l} \) as follows. First we ``translate'' the values of \( \bar{f}_{q_i} \) preceding \( a_0 \) and following \( a_1 \) by suitable fixed constants so that \( g'_{q_{k_i}}(a_0) = p_l \) and \( g'_{q_{k_i}}(a_1) = p_l + \beta_l \). More precisely, let \( u_0 = p_l - \bar{f}_{q_{k_i}}(a_0) \) and \( u_1 = p_l + \beta_l - \bar{f}_{q_{k_i}}(a_1) \), and set \( g'(q_{k_i},b) = (q_l, \bar{f}_{q_{k_i}}(b) + u_0) \) for \( b \leq a_0 \), and  \( g'(q_{k_i},b) = (q_l, \bar{f}_{q_{k_i}}(b) + u_1) \) for \( b \geq a_1 \).
Finally, we extend \( g' \) to the whole \( \RR_{q_{k_i}} \) by mapping $\{q_{k_i}\}\times[a_0,a_1]$ onto $\{q_l\}\times[p_l,p_l+\beta_l]$ by a uniformly continuous partial reduction of $C_x$ to $C_y$ which is weakly increasing in the second coordinate and maps \( a_0 \) to \( p_l \) and \( a_1 \) to \( p_l + \beta_l \) (this is possible by (A3)).

By construction, \( g' \) is well-defined and  satisfies (Y1)--(Y3) and (Y5) (for (Y2) notice that \( g' \) still satisfies (S2) by the condition $|\pi(\bar{f})(q_{k_i})-q_l|< \theta$, while for (Y3) use the fact that \( g' \restriction (- \infty, a_0] \), \( g' \restriction [a_0, a_1] \) and \( g' \restriction [a_1, + \infty) \) are all uniformly continuous).
We now check that \( (\delta', 3\varepsilon') \) is a modulus for \( g' \). Obviously, it is enough to consider \( a \in D_{g'} \setminus D_f \), i.e.\ \( a = q_{k_i} \) for some \( i < M \). Let \( b < b' \) be such that \( |b'-b| < \delta' \). First assume that \( b, b' \leq a_0 \). Then \( d(g'(q_{k_i},b), g'(q_{k_i},b')) = | \bar{f}_{q_{k_i}}(b') + u_0 - (\bar{f}_{q_{k_i}}(b) + u_0)| = d(\bar{f}(q_{k_i},b), \bar{f}(q_{k_i},b')) < \varepsilon' \) since \( ( \delta', \varepsilon' ) \) is a modulus for \( f \) (and hence also for \( \bar{f} \)).
The case \( b,b' \geq a_1 \) is treated in a similar way and gives that \( d(g'(q_{k_i},b), g'(q_{k_i},b'))  < \varepsilon' \) again.  If \( b \leq a_0 \leq b' \leq a_1 \) then  \( | a_0 - b | < \delta' \): by the previous computation, it follows that \( d(g'(q_{k_i},b), g'(q_{k_i},b')) \leq  d(g'(q_{k_i},b), g'(q_{k_i},a_0)) +  d(g'(q_{k_i},a_0), g'(q_{k_i},b')) = |g'_{q_{k_i}}(b) -  g'_{q_{k_i}}(a_0)| +  |g'_{q_{k_i}}(a_0) -  g'_{q_{k_i}}(b')| < \varepsilon' + \varepsilon' = 2 \varepsilon' \) since, by construction, \( p_l = g'_{q_{k_i}}(a_0) \leq  g'_{q_{k_i}}(b') \leq p_l + \beta_l \) and \( \beta_l < \varepsilon' \). The case \( a_0 \leq b \leq a_1 \leq b' \) is treated similarly, so let us consider the case \( b \leq a_0 \leq a_1 \leq b' \). Since in this case \( |a_0 - b | < \delta' \) and \( | b' - a_1 | < \delta' \), by the previous computations we get \( d(g'(q_{k_i},b), g'(q_{k_i},b')) \leq  d(g'(q_{k_i},b), g'(q_{k_i},a_0)) +  d(g'(q_{k_i},a_0), g'(q_{k_i},a_1)) + d(g'(q_{k_i},a_1), g'(q_{k_i},b')) < 3 \varepsilon' \) since \( d(g'(q_{k_i},a_0), g'(q_{k_i},a_1)) = |g'_{q_{k_i}}(a_0) - g'_{q_{k_i}}(a_1)| = | p_l + \beta_l - p_l| = \beta_l < \varepsilon' \). Finally, the case \( a_0 \leq b \leq b' \leq a_1 \) is trivial since in this situation we have \( d(g'(q_{k_i},b), g'(q_{k_i},b')) \leq d(g'(q_{k_i},a_0), g'(q_{k_i},a_1)) = \beta_l < \varepsilon' \). Hence \( ( \delta', 3 \varepsilon' ) \) is a modulus for \( g' \), as required.

Finally, we check that \( \sup \{ d(\bar{f}(a,b) , g'(a,b)) \mid a \in D_{g'}, b \in \RR \} \leq  2 \varepsilon \). It is obviously enough to consider the case \( a = q_{k_i} \) and show that \( d(\bar{f}(q_{k_i},b) , g'(q_{k_i},b)) < 2\varepsilon \) for every \( b \in \RR \). Since \( \beta_{k_i} \leq \beta_{k_0} < \delta \) it follows that \( | \bar{f}_{q_{k_i}}(a_1) - \bar{f}_{q_{k_i}}(a_0) | = d(\bar{f}(q_{k_i},a_0), \bar{f}(q_{k_i},a_1)) < \varepsilon \) (because \( ( \delta, \varepsilon ) \) is a modulus for \( \bar{f} \)), therefore \( u_0, u_1 < \varepsilon \) by the choice of \( l \). This implies that \( d(\bar{f}(q_{k_i},b) , g'(q_{k_i},b)) \leq | \pi(\bar{f})(q_{k_i}) - q_l | + |\bar{f}_{q_{k_i}}(b)-g'_{q_{k_i}}(b) |< 2\varepsilon \) for \( b \leq a_0 \) or \( b \geq a_1 \). If instead \( a_0 \leq b \leq a_1 \), notice that both \( \bar{f}_{q_{k_i}}(b) \) and \( g'_{q_{k_i}}(b) \) belong to the interval \( [\bar{f}_{q_{k_i}}(a_0) , \bar{f}_{q_{k_i}}(a_1)] \), whence \(  d(\bar{f}(q_{k_i},b) , g'(q_{k_i},b)) \leq | \pi(\bar{f})(q_{k_i}) - q_l | + | \bar{f}_{q_{k_i}}(a_1) - \bar{f}_{q_{k_i}}(a_0) | < 2\varepsilon \).

\medskip

We now extend $g'$ to $g''$ in such a way that $D_{g''} = D_{g'} \cup \{ i_{l_j} \mid j<N \}$. Fix \( j < n \).
If $\pi(\bar{f})(i_{l_j})\in I$, then $\bar{f} \restriction \RR_{i_{l_j}}$ is already a partial reduction of $C_x$ to $C_y$ and we define $g'' \restriction \RR_{i_{l_j}}=\bar{f} \restriction \RR_{i_{l_j}}$.
Otherwise, we choose $l \in \omega$ such that
\begin{enumerate}[- ]
\item $i_l<\pi(\bar{f})(i_{l_j})$ and
\item $|\pi(\bar{f})(i_{l_j})-i_l|<\min\{\varepsilon,\theta\}$,
\end{enumerate}
and define $g''(i_{l_j},b)=(i_l,\bar{f}_{i_{l_j}}(b))$ for every $b\in\RR$. Then \( g'' \) satisfies (Y1)--(Y3) and (Y5) (for (Y2) use  $|\pi(\bar{f})(i_{l_j})-i_l|<\theta$),  $(\delta',3 \varepsilon')$ is a modulus for $g''$ (since it is a modulus for \( g' \) and \( ( \delta', \varepsilon') \) is a modulus for \( \bar{f} \)) and \( \sup \{ d(\bar{f}(a,b), g''(a,b)) \mid a \in D_{g''}, b \in \RR \} \leq 2 \varepsilon \) by the analogous property for \( g' \) and $|\pi(\bar{f})(i_{l_j})-i_l|<\varepsilon$.

\medskip

We then extend $g''$ to \( g''' \) in such a way that $R_{g'''} = R_{g''} \cup \{q_{k_i} \mid i<M \}$. Fix \( i < M \) and let $r=\pi(\bar{f})^{-1}(q_{k_i})$ (such an \( r \) exists because \( \bar{f} \), and hence \( \pi(\bar{f}) \), is surjective). If $\bar{f} \restriction \RR_r$ is already a partial reduction of $C_x$ to $C_y$, we let $g''' \restriction\RR_r=\bar{f} \restriction\RR_r$.
Otherwise, to simplify the notation let \( \beta = \beta_{k_i} \), $b_0= p_{k_i} = \inf(B_{x,k_i})$,  and $b_1= p_{k_i} + \beta_{k_i} = \sup(B_{x,k_i})$. Let
$l \in \omega$ be such that
\begin{enumerate}[- ]
\item $q_l \leq r $,
\item $|r-q_l|< \theta$,
\item $d(\bar{f}(q_l,b),\bar{f}(r,b))< \varepsilon$ for all $b\in\RR$, and
\item $\pi_1(\bar{f}^{-1}(q_{k_i},b_0))< p_l < p_l + \beta_l <\pi_1(\bar{f}^{-1}(q_{k_i},b_1))$.
\end{enumerate}
Notice that the third requirement is possible since $\bar{f}$ is uniformly continuous. Then we define $g''' \restriction \RR_{q_l}$ from \( \RR_{q_l} \) onto \( \RR_{q_{k_i}} \) similarly to the case of \( g' \restriction \RR_{q_{k_i}} \). More precisely, let \( u_0 = b_0 -\bar{f}_r(p_l) \) and \( u_1 = b_1- \bar{f}_r(p_l + \beta_l) \), and set $g(q_l,b)=(q_{k_i},\bar{f}_r(b)+u_0)$ for $b\leq p_l$
and $g(q_l,b)=(q_{k_i}, \bar{f}_r(b) +u_1)$ for $b\geq p_l + \beta_l$.
Finally, we extend \( g''' \) to the entire \( \RR_l \) by mapping  $\{q_l\}\times[p_l,p_l + \beta_l]$ onto $\{q_{k_i}\}\times[b_0,b_1]$ via a uniformly continuous partial reduction of $C_x$ onto $C_y$ which is weakly increasing in the second coordinate and maps \( p_l \) to \( b_0 \) and \( p_l + \beta_l \) to \( b_1 \) (this is possible by (A3)).

Arguing as for \( g' \), it is not hard to check that \( g''' \) satisfies (Y1)--(Y3) and (Y5), and that \( ( \delta', 3 \varepsilon') \) is a modulus for \( g''' \) since it is a modulus for \( g'' \) and \( | b_1 - b_0 | = \beta_{k_i} \leq \beta_{k_0} < \varepsilon' \). Finally, we check that \( \sup \{ d(\bar{f}(a,b), g'''(a,b)) \mid a \in D_{g'''}, b \in \RR \} \leq 2 \varepsilon\). Since the analogous property for \( g'' \) holds, it is enough to check that given \( i < M \) and \( l \in \omega \) as above, \( d(\bar{f}(q_l,b), g'''(q_l,b)) < 2 \varepsilon \) for every \( b \in \RR \). First notice that \( d(\bar{f}(q_l,b), g'''(q_l,b)) \leq d(\bar{f}(q_l,b), \bar{f}(r,b)) + d(\bar{f}(r,b), g'''(q_l,b)) \): since \( d(\bar{f}(q_l,b), \bar{f}(r,b)) < \varepsilon \) by the third property above, it is enough to check that \( d(\bar{f}(r,b), g'''(q_l,b)) = | \bar{f}_r(b)- g'''_{q_l}(b)| < \varepsilon\). Since \( | b_1 - b_0 | = \beta_{k_i} \leq \beta_{k_0} < \varepsilon \), we get \( u_0, u_1 < \varepsilon \), whence \( | \bar{f}_r(b)- g'''_{q_l}(b)| < \varepsilon\) for \( b \leq p_l \) and \( b \geq p_l + \beta_l \).  If instead \( p_l \leq b \leq p_l + \beta_l \) we have that both \( \bar{f}_r(b) \) and \( g'''_{q_l}(b) \) belong to the interval \( [b_0,b_1] \) which has length \( \beta_{k_i} \leq \beta_{k_0} < \varepsilon \), whence \( | \bar{f}_r(b)- g'''_{q_l}(b)| < \varepsilon\) again.

\medskip

Finally, we extend \( g''' \) to \( g \) in such a way that \( R_g = R_{g'''} \cup \{ i_{l_j} \mid j<N \} \). Let $r=\pi(\bar{f})^{-1}(i_{l_j})$. If $r\in I$, then $\bar{f} \restriction \RR_r$ is already a partial reduction of $C_x$ to $C_y$ and we simply define $g \restriction \RR_r=\bar{f} \restriction\RR_r$.
Otherwise, we choose $l \in \omega$ with
\begin{enumerate}[- ]
\item $i_l \leq r$,
\item $d(i_l,r)<\theta$, and
\item $d(\bar{f}(i_l,b),\bar{f}(r,b)<\varepsilon$ for all $b\in\RR$ .
\end{enumerate}
As for the definition of \( g''' \), the last requirement is possible by the uniform continuity of $\bar{f}$. Then we define $g(i_l,b)=(i_{l_j},\bar{f}_r(b))$ for all $b\in\RR$. As for the previous steps, $g$ satisfies (Y1)--(Y3) and (Y5), and $(\delta',3\varepsilon')$ is a modulus for $g$ since it is a modulus for \( g''' \) and \( ( \delta', \varepsilon') \) is a modulus for \( \bar{f} \). Finally,
\begin{equation} \label{eqddagger}
\tag{$\ddagger$}
\sup \{ d(\bar{f}(a,b),g(a,b)) \mid a \in D_g, b \in \RR \} \leq 2 \varepsilon
\end{equation}
by the last requirement in the definition of \( l \in \omega \) and since the analogous property holds for \( g''' \). A straightforward computation shows that~\eqref{eqddagger} implies \( \Vert \bar{f} - \bar{g} \Vert \leq 2 \varepsilon \) by~\eqref{eqextension}, therefore \( g \) satisfies (1)--(3). Hence it remains to check that \( g  \in \mathcal{Y}\). Let \( K \in \omega \) be a witness of the fact that \( f \) satisfies (Y4). As already observed, this implies that \( |\bar{f}_a(b) - \bar{f}_{a'}(b)| < \frac{K}{\rho}\cdot |a-a'| \) for every \( a,a',b \in \RR \), where \( \rho = \min \{ |a - a' | \mid a,a' \in D_f \} \). Let \( K' = \frac{K}{\rho} \cdot \max \{ |a-a'| \mid a,a' \in D_g \} \). Then for every \( a,a' \in D_g \) and \( b \in \RR \) we have \begin{multline*}
|g_a(b) - g_{a'}(b)| \leq |g_a(b) - \bar{f}_a(b)| + |\bar{f}_a(b) - \bar{f}_{a'}(b)| + |\bar{f}_{a'}(b) - g_{a'}(b)| \\ \leq d(g(a,b),\bar{f}(a,b)) + K' + g(\bar{f}(a',b),g(a',b)) \leq 2 \varepsilon + K' + 2 \varepsilon = K' + 4 \varepsilon
 \end{multline*}
by (3). This shows that \( K' + 4 \varepsilon \) witnesses that \( g \) satisfies (Y4), and hence \( g \in \mathcal{Y} \) by Definition~\ref{defY}, as required.
\end{proof}

\begin{proof} [Proof of Theorem~\ref{theorR2}]
We will show that the map
\( \pow(\omega) \to \pow(\RR^2) \colon x \mapsto C_ x \),
where \( C_x \) is defined as in
\eqref{eqCx}, is the desired embedding.

Assume first that $x\subseteq^* y$: we claim that $C_x\leq_{\mathsf{D}_1} C_y$ (hence, in particular, \( C_x \leq_{\mathsf{D}_2} C_y \)).
In fact, we will construct a continuous reduction
of $C_x$ to $C_y$ as a \emph{uniform} limit of a sequence of special maps $\langle \bar{f}_n \mid f_n \in \mathcal{Y}, n\in\omega \rangle$ with the property that  for every \( k \in \omega \) there is \( N_k \in \omega \) such that:
\begin{enumerate}[(1) ]
\item
\( \bar{f}_{N_k} \restriction \RR_{q_k} = \bar{f}_n \restriction \RR_{q_k} \) and \( \bar{f}_{N_k} \restriction \RR_{i_k} = \bar{f}_n \restriction \RR_{i_k} \) for every \( n \geq N_k \) (so that, in particular, \( \pi(\bar{f}_{N_k})(q_k) = \pi(\bar{f}_n)(q_k) \) and \( \pi(\bar{f}_{N_k})(i_k) = \pi(\bar{f}_n)(i_k) \));
\item
\( \bar{f}_{N_k} \restriction \RR_{q_k} \) is a partial reduction of  \( C_x \)  to \( C_y \): in particular,
\( \pi(\bar{f}_{N_k})(q_k) \in \QQ \) by property (A1);
\item
\( \pi(\bar{f}_{N_k})(i_k) \in I \);
\item
for every \( l \in \omega \) there is \( m \in \omega \) such that \( \pi(\bar{f}_{N_m})(q_m) = q_l \);
\item
for every \( l \in \omega \) there is \( m \in \omega \) such that \( \pi(\bar{f}_{N_m})(i_m) = i_l \).
\end{enumerate}
Assume that such a sequence exists, and let \( \bar{f} = \lim_{n \to \infty}
\bar{f}_n \), so that, in particular, \( \bar{f} \) is continuous because the \( \bar{f}_n \)'s are continuous and are assumed to converge uniformly:
then \( \bar{f} \) witnesses \( C_x
\leq_{\mathsf{D}_1} C_y \). In fact, let \( (a,b) \in \RR^2 \). If \( a =
q_k \) or \( a = i_k \) for some \( k \in \omega \), then \( \bar{f}(a,b) =
\bar{f}_{N_k}(a,b) \) by (1), and hence \( (a,b) \in C_x \iff \bar{f}(a,b) \in
C_y \) by, respectively, (2) and (3) (depending on whether \( a\in \QQ \) or \(
a \in I \)). It remains to consider the case  \( a \in \RR \setminus (\QQ \cup
I) \). Notice that since \( \bar{f} \) is the limit of a sequence of special maps,
\( \bar{f} \) still satisfies (S1) and (S2), hence \( \pi(\bar{f})
\colon \RR \to \RR \) is a well-defined injective map. By (4), (5) and (1), \(
\QQ \cup I \subseteq \pi(\bar{f})(\QQ \cup I) \) (in fact, \( \QQ \cup I =
\pi(\bar{f}) (\QQ \cup I) \) by (2) and (3)). By injectivity of \( \pi(\bar{f}) \), this
implies \( \pi(\bar{f})(a) \in \RR \setminus (\QQ \cup I) \), whence \( (a,b) \in
C_x \iff \bar{f}(a,b) \in C_y \) because, by construction, \( (a,b) \notin C_x \)
and \( \bar{f}(a,b) \notin C_y \).

It remains to show that the above sequence of  maps exists: the sequence \( \langle \bar{f}_n \mid f_n \in \mathcal{Y}, n \in \omega \rangle \) will be defined recursively using a back-and-forth construction. Towards this aim, we will also define an auxiliary sequence $\langle\delta_n\mid n\in\omega\rangle$ of positive real numbers, and require that
\begin{enumerate}[(i) ]
\item
\( f_{n+1} \) extends \( f_n \);
\item
\( q_m \in D_{f_n} \cap R_{f_n} \) for all $m$ with $\beta_m\geq\min\{\delta_n, \frac{1}{3 \cdot 2^{n+2}}\}$,
\item
\( i_m \in D_{f_n} \cap R_{f_n} \) for all $m<n$,
\item
$(\delta_n,\frac{1}{2^{n+1}})$ is a modulus for $f_n$, and
\item
\( \Vert \bar{f}_{n+1} - \bar{f}_n \Vert \leq \frac{1}{2^n} \).
\end{enumerate}

Set $f_0= \emptyset$ and $\delta_0=1$. Notice that  $f_0$ fulfills all the requirements since $\beta_0<\frac{1}{12}$.
To define $f_{n+1}$, we first choose $\delta_{n+1}$ such that
$(\delta_{n+1},\frac{1}{3\cdot2^{n+2}})$ is a modulus for $\bar{f_n}$:
this is possible since $\bar{f_n}$ is uniformly continuous by
Lemma~\ref{lemmaextensionspecial}. Next we extend $f_n$ to $f_{n+1}$  by
applying Lemma~\ref{lemmaextension} with \( f = f_n \), \( \delta = \delta_n \),
\( \varepsilon = \frac{1}{2^{n+1}} \), \( \delta' = \delta_{n+1} \), and
\( \varepsilon' = \frac{1}{3 \cdot 2^{n+2}} \) in such a way that setting $f_{n+1} = g$ we get a function which fulfills
(ii) and (iii). More precisely, we let \( \{ k_i \mid i < M \} \) be an increasing enumeration of those \( m \in \omega \) such that
\[
\min \left \{ \delta_{n+1}, \frac{1}{3 \cdot 2^{n+3}} \right \}  \leq \beta_m < \min \left \{ \delta_n, \frac{1}{3\cdot 2^{n+2}} \right \},
\]
 and set \( N=1 \) and \( l_0 = n \). Notice that Lemma~\ref{lemmaextension} can be applied with these parameters because \( \beta_{k_0} < \min \{ \delta_n ,\frac{1}{3 \cdot 2^{n+2}} \} = \min \{ \delta, \varepsilon, \varepsilon' \} \), and that the resulting \(  f_{n+1} = g \) satisfies (ii) and (iii) because of the choice of the \( k_i \)'s and \( l_0 \), together with the fact that \( f_n \) satisfies such conditions by inductive hypothesis.
Then $(\delta_{n+1},\frac{1}{2^{n+2}})$ is a modulus for $f_{n+1}$ by
Lemma~\ref{lemmaextension}(ii), and $\Vert \bar{f}_{n+1} - \bar{f}_{n} \Vert \leq \frac{1}{2^n}$ by
Lemma~\ref{lemmaextension}(iii), hence \( f_{n+1} \) satisfies also (iv) and (v). This
completes the recursive definition. It is immediate to check that conditions (1)--(5) are satisfied by construction. Moreover, using standard arguments one can easily show
that (v)  implies that the sequence \( \langle \bar{f}_n \mid n \in \omega \rangle \) uniformly converges to some \( \bar{f} \colon \RR^2 \to \RR^2 \),
hence we are done.

\medskip

To complete the proof of Theorem~\ref{theorR2}, it remains to show that if
 $C_x\leq_{\mathsf{D}_2} C_y$ then $x\subseteq^* y$.
Suppose that $f$ witnesses $C_x \leq_{\mathsf{D}_2 }C_y$. Then there is a nonempty open set $U\subseteq\mathbb{R}^2$ such that $f\restriction U$ is continuous by the observation preceding Proposition~\ref{propnotverygood}. We may assume that $U$ is an open rectangle.
Fix \( a \in I \). If \( b < b' \) are such that \( (a,b), (a,b') \in U \), then $\{ a \} \times [b,b'] \subseteq C_x \cap U$ is homeomorphic to \( [0,1]  \subseteq \RR \), hence  $f(\{ a \} \times [b,b'])$ is a (possibly degenerate) path totally contained in \( C_y \) (since \( f \restriction U \) is continuous and \( f \) reduces \( C_x \) to \( C_y \)). By construction of \( C_y \) and the choice of \( I \), this implies that $f(\{ a \} \times [b,b'])$ is contained in a single vertical line, so that in particular \( \pi_0(f(a,b)) = \pi_0(f(a,b')) \). We have thus shown that for every \( a \in I \), \( f(\RR_a \cap U) \subseteq \RR_d \) for some \( d \in \RR \).
It follows that the same holds also for an arbitrary \( a \in \RR \), i.e.\ that \( \pi_0(f(a,b)) = \pi_0(f(a,b')) \) for every $(a,b),(a,b')\in U$, because otherwise, by continuity of \( f \restriction U \), we would have \( \pi_0(f(a',b)) \neq \pi_0(f(a',b')) \) for some \( a' \in I \) sufficiently close to \( a \).
Now notice that by the choice of the $ p_n$'s and the \( \beta_n \)'s there is an index $n \in \omega $ with $\{q_n\}\times [p_n,p_n + \beta_n]\subseteq U$. Let \( s \in \RR \) be such that \( f(\RR_{q_n} \cap U)  \subseteq \RR_s\).
Since $f$ is a reduction of $C_x$ to $C_y$ and \( C_x \cap \RR_{q_n} = \{ q_n \} \times B_{x,n} \subseteq U \) is neither empty nor the entire \( \RR_{q_n} \cap U \) by (A1),  \( s = q_m \) for some \( m \in \omega \). Therefore, \( C_y \cap \RR_{q_m} = \{ q_m \} \times B_{y,m} \). Let \( \varepsilon > 0 \) be such that \( \{ q_n \} \times [p_n - \varepsilon, p_n + \beta_n + \varepsilon] \subseteq U \). Then the map \( g \colon \RR \to \RR \) defined by \( g(x) = f(p_n-\varepsilon) \) if \( x \leq p_n - \varepsilon \), \( g(x) = f(p_n + \beta_n + \varepsilon) \) if \( x \geq p_n+\beta_n+\varepsilon \), and \( g(x) = f(x) \) if \( p_n - \varepsilon \leq x \leq p_n + \beta_n + \varepsilon \) is a continuous reduction of \( B_{x,n} \) to \( B_{y,m} \).
Since \( B_{x,n} \) is an affine image of $A_x$ and $B_{y,m}$ is an affine image of $A_y$, it follows that \( A_x \leq^\RR_{\mathsf{D}_1} A_y \), and hence $x\subseteq^* y$ by (A2).
\end{proof}

\begin{corollary}
Suppose that \( X \) is an Hausdorff quasi-Polish space such that  either $\RR^2 \subseteq X$ (in particular, we can take e.g.\
 \( X = \RR^n \) or $X=[0,1]^n$ for $2\leq n<\omega$), or \( X \) is an \( n \)-dimensional $\sigma$-compact Polish
space which is embeddable into \( \RR^n \)  (for some \( n\geq 2 \)), or \( X \) is locally Euclidean and of dimension \( n \geq 2 \).
\begin{enumerate}[(1)]
\item Any countable partial order embeds into the $( \mathbf{B},\mathsf{D}_2)$-hierarchy on $X$.
\item ($\mathsf{ZFC}$) Any partial order of size $\omega_1$ embeds into the $( \mathbf{B},\mathsf{D}_2)$-hierarchy on $X$.
\item The $( \mathbf{B},\mathsf{D}_2)$-hierarchy on $X$ is very bad.
\end{enumerate}
\end{corollary}

\begin{proof}
First we consider the case \( X = \RR^2 \). Assuming the axiom of choice, every partial order of size $\omega_1$ embeds into $(\mathscr{P}(\omega),\subseteq^*)$ by Parovi\v{c}enko's Theorem~\cite{par}. The embedding of countable partial orders is constructed by diagonalizing over finitely many subsets of $\omega$ in each step, hence \( \mathsf{AC} \) is not necessary.
In particular, $(\mathscr{P}(\omega),\subseteq^*)$ contains infinite antichains and infinite decreasing sequences. Hence the result follows from Theorem~\ref{theorR2}.

Now consider the case of an Hausdorff quasi-Polish space \( X \) such that \( \RR^2 \subseteq X \). Since \( \RR^2 \) is (quasi-)Polish, \( \RR^2 \in \mathbf{\Pi}^0_2(X) \) by Proposition~\ref{propsubspace}. Since \( \RR^2 \) is \(\sigma\)-compact and \( X \) is Hausdorff, \( \RR^2 \in \mathbf{\Sigma}^0_2(X) \), hence \( \RR^2 \in \mathbf{\Delta}^0_2(X) \). Therefore the result follows from Corollary~\ref{corfr2}.

Finally, let \( X \) be either an \( n \)-dimensional $\sigma$-compact Polish
space which is embeddable into \( \RR^n \)  (for some \( n\geq 2 \)), or a locally Euclidean Polish space of dimension \( n \geq 2 \). Then \( X \cong_2 \RR^n \) by Remark~\ref{remJR}(2)(a-b), and hence the desired claim follows from Proposition~\ref{propiso} and the fact that we already proved the analogous result for \( \RR^n \).
\end{proof}

Unfortunately, for what concerns the \( \mathsf{D}_n \)-hierarchies (\( n \geq 3 \)) on uncountable quasi-Polish spaces of dimension \( \infty \), the situation is still completely unclear.

\begin{question} Are the $\mathsf{D}_3$-hierarchies on the Hilbert cube $[0,1]^\omega$ and on the Scott domain \( P  \omega \) (very) good? What about the \( \mathsf{D}_n \)-hierarchies (for larger \( n \in \omega \))? What about other (quasi-)Polish spaces of dimension \( \infty \)?
\end{question}

\subsection{Degree structures in countable spaces}\label{subsectioncountable}

Here  we consider the case of countable quasi-Polish spaces and show that even when considering the very restricted class of
scattered countably based spaces, the Wadge hierarchy may not be very good.

It follows from Proposition~\ref{propgeneralhomeoctbl}(1) and Proposition~\ref{propiso} that  all countable countably based $T_0$-spaces \( X \) have isomorphic \( \F \)-hierarchies whenever \( \F \supseteq \mathsf{D}^\mathsf{W}_3\) is a family of reducibilities. Moreover, it follows from parts (2) and (3) of  Proposition~\ref{propgeneralhomeoctbl} and Proposition~\ref{propiso} that  all countable  $T_1$-spaces and all scattered countably based spaces \( X \) have isomorphic \( \F \)-hierarchies whenever \( \F \supseteq \mathsf{D}^\mathsf{W}_2 \). In fact, the proof of Proposition~\ref{propgeneralhomeoctbl} shows that in all the above mentioned cases, the class \( \F(X) \) coincides with the class of all functions from \( X \) to itself: hence the resulting \( \F \)-hierarchy is formed by two incomparable degrees consisting of \( \emptyset \) and the whole space \( X \), together with a single degree above them containing all other subsets of \( X \) (this in particular means that the \( \F \)-hierarchy on \( X \) is very good).
Therefore only the Wadge reducibility and the $\mathsf{D}_2$-reducibility are of
interest when considering countable countably based \( T_0 \)-spaces \( X \), and when \( X \) is a countable \( T_1 \) space or a countably based scattered space, then only the
Wadge reducibility needs to be considered.

The next result identifies two classes of countable spaces with a very good structure of Wadge degrees.

\begin{proposition}\label{propco1}
Let $X$ be a countable Polish space or  a finite $T_0$-space. Then
the \( \mathsf{D}_1 \)-structure on \( X \) is very good.
 \end{proposition}

\begin{proof}  If \( X \) is countable  Polish then it is
zero-dimensional, hence the claim follows from Theorem~\ref{theorcoruncountable} and the fact that  \(
\mathscr{P}(X) \subseteq \mathbf{\Delta}^0_2(X) \).

Let now \( X \) be a finite \( T_0 \)-space, and let \( \leq \) be the \emph{specialization order} on \( X \) defined by setting \( x \leq y \iff  x \) belongs to the closure of \( \{ y \} \) (for every \( x,y \in X \)).  Then \( X \) coincides with the \( \omega \)-algebraic domain $(X, \leq)$ (endowed with the Scott or, equivalently, the Alexandrov topology), and the continuous functions on \( X \)
coincide with the monotone functions on $(X,\leq)$. Moreover, any subset of $X$ is a
finite Boolean combination of open sets, hence it is in \( \bigcup_{n \in \omega} \boldsymbol{\Sigma}^{-1}_n = \bigcup_{n \in \omega} \boldsymbol{\Pi}^{-1}_n = \bigcup_{n \in \omega } \boldsymbol{\Delta}^{-1}_n \). By (the second part of) Theorem~\ref{t-dh}, this implies that every
subset of \( X \) is Wadge complete in one of
$\mathbf{\Sigma}^{-1}_n,\mathbf{\Pi}^{-1}_n,\mathbf{\Delta}^{-1}_n$,
$n \in \omega$: therefore,  \( (\mathscr{P}(X),\leq_\mathsf{W}) \) is
semi-well-ordered.
 \end{proof}

\begin{remark}
In both the cases considered in Proposition~\ref{propco1}, the space \( X \) falls in one of the cases mentioned at the beginning of this subsection: if \( X \) is Polish then it is also \( T_1 \), while if \( X \) is finite then it is automatically scattered. Therefore the \( \mathsf{D}_2 \)-hierarchy on such an \( X \) is always (trivially) very good.
\end{remark}

We now show that Proposition~\ref{propco1} cannot be extended  to
scattered quasi-Polish spaces by showing that there is a scattered
$\omega$-continuous domain whose Wadge hierarchy is good but not very good. Note that by Proposition~\ref{propco2},
any scattered dcpo is in fact an algebraic domain and all of its elements
are compact.

\begin{proposition}\label{propco3}
There is a scattered  $\omega$-algebraic domain $(X, \leq)$ such that
$(\mathscr{P}(X), \leq_\mathsf{W})$ has four pairwise incomparable elements (hence it is not very good).
 \end{proposition}

\begin{proof}
For any $n \in \omega$, fix an \( \leq \)-chain $C_n=\{c^n_n<\dotsc<c^n_0\}$
with $n+1$-many elements.  Let $C=\bigsqcup_{n \in \omega} C_n$ be the
disjoint union of these chains (so the elements of different
chains are \( \leq \)-incomparable in $C$), and let $X$ be obtained from $C$
by adjoining a bottom element $\bot$ and a top element $\top$. By
Proposition~\ref{propco2}, $X$ is a scattered $\omega$-algebraic
domain.

We inductively define the sets \( D_k \subseteq X \), \( k \in \omega \), by letting $D_0=\{x \in X \mid \exists n (c^n_0\leq x)\}$ and
$D_{i+1}=D_i \cup \{x \in X \mid \exists n > i (c^n_{i+1} \leq
x) \}$. Then $D_0 \subseteq D_1 \subseteq \dotsc$ and any $D_k$ is
open in $X$, so the sets $A=D_0 \cup \bigcup_k (D_{2k+2}\setminus
D_{2k+1})$ and $B=A \setminus\{\top\}$ are in the pointclass
$\mathbf{\Sigma}^{-1}_\omega(X)$, i.e.\ in the \( \omega \)-th level of the Hausdorff difference
hierarchy over the open sets in $X$ (see Subsection~\ref{subsectionhier}).
Observe that for all \( k < n  \in \omega \) we have  $\bot \not \in D_k$, $c^n_0\in D_0$, and $c^n_{k+1} \in D_{k +1}\setminus D_k$. Therefore, for \( k \leq n \in \omega \) we have
$\bot\not\in A\cup B$, $c^n_k\in A\cap B$ if \( k \) is even,
and $c^n_k\not\in A\cup B$ if \( k \) is odd. This means that for each of \( A \) and \( B \) one can construct in the obvious way a \( 0 \)-alternating tree onto
$X \setminus \{ \top \} = \{\bot,c^n_k \mid k \leq n  \in \omega\}$,
so $A,B \notin \mathbf{\Pi}^{-1}_\omega$ by (the first part of) Theorem~\ref{t-dh}.

We claim that the sets $A,\overline{A},B,\overline{B}$ are
pairwise Wadge incomparable. Since any of \( \boldsymbol{\Sigma}^{-1}_\alpha, \boldsymbol{\Pi}^{-1}_\alpha \), being a family of pointclasses,   is closed under continuous preimages, the
fact that $A,B \in \boldsymbol{\Sigma}^{-1}_\omega \setminus \boldsymbol{\Pi}^{-1}_\omega$ implies that each of \( A,B \) is Wadge incomparable
with both $\overline{A}$ and $\overline{B}$. So it remains only to check that
$A$ is Wadge incomparable with $B$: we will just show that $A\not\leq_\mathsf{W} B$, as the fact that
$B\not\leq_\mathsf{W} A$ can be proved in the same way. Assume towards a
contradiction that $A=f^{-1}(B)$ for a continuous (i.e.\ monotone) function $f$ on $X$. Since $\top\in A$ and $\top \not \in
B$, $f(\top) \neq \top$, hence $f(\top) \in X \setminus \{ \top \} = \{\bot,c^n_k\mid k \leq n \in \omega \}$. Since $f(x) \leq f(\top)$ for all $x\in X$, the range of $f$
is contained in \( C_m \cup \{ \bot \} \) for some $m<\omega$.
Choose an alternating chain $a_0 < \dotsc <a_{m+2}$ for $A$ of length \( m+ 3 \) (e.g.\ we can take \( a_i = c^{m+2}_{m+2-i} \) for \( i \leq m+2 \)). Then, since \( f \) is supposed to be a reduction of \( A \) to \( B \), the image under \( f \) of this chain must be alternating for \( B \). Since, as already observed in Subsection~\ref{omega-cont}, this implies $f(a_0) < \dotsc < f(a_{m+2})$, we have \( | \{ f(a_i) \mid i \leq m+2 \} | = m+3 > m+2 = | C_m \cup \{ \bot \} | \geq | \{ f(x) \mid x \in X \} | \), a contradiction.
 \end{proof}

Notice that in the example above the four Wadge incomparable elements have (almost) the minimal possible complexity, as by (the second part of) Theorem~\ref{t-dh} the Wadge hierarchy on any \(\omega\)-algebraic domain is semi-well-ordered when restricted to \( \bigcup_{n \in \omega} \boldsymbol{\Sigma}^{-1}_n \). Moreover, since the structure of the poset $(X,\leq)$ in the previous proof
is very simple, it is possible to completely describe the
Wadge hierarchy on $X$, as it is shown in the next proposition.

Let $(P,\leq_P)$ and $(Q,\leq_Q)$ be arbitrary posets.  By $P+Q$ we denote the poset $(P\sqcup Q,\leq)$
where ${\leq} \restriction P$ (respectively, on ${\leq} \restriction Q$) coincides with $\leq_P$ (respectively, with $\leq_Q$), and
$p\leq q$ for all $p\in P$ and $q\in Q$. By $P\cdot Q$ we denote the
poset $(P\times Q,\leq)$ where  $(p_0,q_0)< (p_1,q_1)$ if and only if
$q_0 <_Q q_1$ or $q_0 = q_1\wedge p_0 <_P p_1$. For any integer $n\geq 2$, let
$\bar{n}$ denote the poset consisting of exactly one antichain with $n$ elements, and identify \( \omega \) with the poset \( (\omega, \leq) \).

\begin{proposition}\label{propco4}
The quotient-poset of $(\mathscr{P}(X),\leq_\mathsf{W})$ is
isomorphic to 
$(\bar{2} \cdot \omega)+\bar{4}$. Hence the \( \mathsf{W} \)-hierarchy on \( X \) is good but
not very good.
 \end{proposition}

\begin{proof}
Let us first make a basic observation which is intimately related to the very particular structure of the poset \( (X, \leq ) \). Given \( S \subseteq X \), call an alternating chain \( a_0 < \dotsc < a_n \) for \( S \) \emph{maximal} if for every \(  b < a_0 \) and \( b' > a_n \) none of \( b < a_0 < \dotsc < a_n \) and \( a_0 < \dotsc < a_n < b' \) is alternating for \( S \). Notice that if \( a_0 < \dotsc < a_n \) is a maximal alternating chain for \( S \), then \( a_0 \in S \iff \bot \in S \) and \( a_n \in S \iff \top \in S \) (otherwise one of \( \bot < a_0 < \dotsc < a_n \) or \( a_0 < \dotsc < a_n < \top \) would be alternating, contradicting the maximality of the chain). This also implies that all maximal alternating chains are compatible, i.e.\ that if \( a'_0 < \dotsc < a'_{n'} \) is another maximal alternating chain for \( S \) then \( a'_0 \in S \iff a_0 \in S \) and \( a'_{n'} \in S \iff a_n \in S \).

Let now $\mathcal{S}$ be the class of all $\emptyset \neq S \subsetneq X$ such that there is a natural number \( n \) bounding the lengths of all alternating chains for $S$. We will use Theorem~\ref{t-dh} to show that each \( S \in \mathcal{S} \) belongs to some finite level of the difference hierarchy over the open sets of \( X \). First notice that every \( S \in \mathcal{S} \) (in fact, every \( S\subseteq X \)) is approximable because all elements of \( X \) are compact. Given \( S \in \mathcal{S} \), let \( m(S)+1 \) be the maximal length of an alternating chain for \( S \), so that, in particular, there is an alternating tree for \( S \) of rank \( m(S) \) but no alternating tree for \( S \) of rank \( m(S) + 1 \). Notice that by definition of \( \mathcal{S} \), we necessarily have \( m(S) \geq 1 \). Let
$a_0<\dotsc <a_{m(S)}$ be an alternating chain for \( S \) of length \( m(S) +1 \). By definition of \( m(S) \), such a chain is necessarily maximal, hence by the compatibility of the maximal alternating chains for \( S \)
there is no alternating chain $b_0<\dotsc<b_{m(S)}$ for $S$ such that $a_0\in
S \iff b_0\notin S$. This means that all alternating trees for \( S \) of rank \( m(S) \) are of the same type, i.e.\ either they are all \( 1 \)-alternating or they are all \( 0 \)-alternating (depending on whether \( \bot \in S  \) or not). Then by Theorem~\ref{t-dh}, either $S \in \boldsymbol{\Pi}^{-1}_{m(S)}(X) \setminus \boldsymbol{\Sigma}^{-1}_{m(S)}(X)$ or $S \in \boldsymbol{\Sigma}^{-1}_{m(S)}(X) \setminus \boldsymbol{\Pi}^{-1}_{m(S)}(X)$, hence \( S \) is also Wadge complete in (exactly) one
of $\mathbf{\Sigma}^{-1}_{m(S)}(X)$ or $\mathbf{\Pi}^{-1}_{m(S)}(X)$.

Conversely, it is not hard to check that
all the possibilities are realized, i.e.\ that \( \boldsymbol{\Sigma}^{-1}_n(X) \setminus \boldsymbol{\Pi}^{-1}_n(X) \neq \emptyset \) for every \( 1 \leq n \in \omega \). In fact, let \( A,B \) be defined as in the proof of Proposition~\ref{propco3}, and put \( C'_n  =  C_n \cup \{ \bot,\top \} \) for every \( n \in \omega \). Then one can straightforwardly check using Theorem~\ref{t-dh} that \( \overline{A} \cap C'_{2i} \in \boldsymbol{\Sigma}^{-1}_{2i+1}(X) \setminus \boldsymbol{\Pi}^{-1}_{2i+1}(X) \), \( A \cap C'_{2i} \in \boldsymbol{\Pi}^{-1}_{2i+1}(X) \setminus \boldsymbol{\Sigma}^{-1}_{2i+1}(X) \), \( B \cap C'_{2i} \in \boldsymbol{\Sigma}^{-1}_{2i+2}(X) \setminus \boldsymbol{\Pi}^{-1}_{2i+2}(X) \), and \( \overline{B} \cap C'_{2i} \in \boldsymbol{\Pi}^{-1}_{2i+2}(X) \setminus \boldsymbol{\Sigma}^{-1}_{2i+2} (X) \) for every \( i \in \omega \).

Therefore, letting \( h(\emptyset) = (0,0) \), \( h(X) = (1,0 ) \), and, for \( S \in \mathcal{S} \), $h(S)=(0,m)$  (respectively, $h(S)=(1,m)$)
if and only if $S$ is Wadge complete in  $\mathbf{\Sigma}^{-1}_m(X)$ (respectively, in
$\mathbf{\Pi}^{-1}_m(X)$), we get that the function
$h \colon \mathcal{S} \cup \{ \emptyset, X \} \to\bar{2}\cdot\omega$ induces an isomorphism
between the quotient-poset of $(\mathcal{S} \cup \{ \emptyset,X \},\leq_\mathsf{W})$ and
$\bar{2}\cdot\omega$.

Now assume that $S \subseteq X$ has arbitrarily long finite alternating chains, i.e.\ that \( S \neq \emptyset,X \) and \( S \notin \mathcal{S} \).  Obviously, \( \emptyset,X \leq_\mathsf{W} S \). We claim that also every set
from $\mathcal{S}$ is  Wadge reducible to $S$. In fact, let \( S' \in \mathcal{S} \) and \( m = m(S') \). Let \( a_0 < \dotsc a_{m+1} \) be an alternating chain for \( S \) such that \( a_0 \in S \iff \bot \in S ' \) (such a chain exists by the choice of \( S \)). Consider the function \( f \colon X \to \{ a_0, \dotsc, a_{m+1} \} \) defined in the following way: first, \( f(\bot) = a_0 \) and \( f(\top) = a_{m+1} \). Then fix \( n \in \omega \), and define \( f \) on \( c^n_{n-i} \) by induction on \( i \leq n \) as follows. For \( i = 0 \), set \( f(c^n_n) = a_0 \) if \( c^n_n \in S' \iff \bot \in S' \) and \( f(c^n_n) = a_1 \) otherwise. For the inductive step, let \( f(c^n_{n-i}) = a_j \), and set \( f(c^n_{n-(i+1)}) = a_j \) if \( c^n_{n-i} \in S' \iff c^n_{n-(i+1)} \in S' \) and \( f(c^n_{n-(i+1)}) =  a_{j+1} \) otherwise. Then \( f \) is clearly monotone (hence continuous) and reduces \( S' \) to \( S \).

 Using essentially the same method as in the previous paragraph, for every \( C \in \{ A,B, \overline{A}, \overline{B} \} \) (where \( A,B \) are again defined as in the proof of Proposition~\ref{propco3}) one can easily
define two monotone functions $f',g' \colon X\setminus\{\bot,\top\} \to X\setminus\{\bot,\top\}$
such that $x\in C \iff f'(x)\in S$ and $x\in
S \iff g'(x)\in C$ for all $x \in
X \setminus \{\bot,\top\}$ (i.e.\ \( f' \) and \( g' \) are partial continuous reduction of, respectively, \( C \) to \( S \) and \( S \) to \( C \)). Extend \( f' \) and \( g' \), respectively,  to the functions \( f,g \colon X \to X \) by setting \( f(\bot) = g(\bot) = \bot \) and \( f(\top) = g(\top) = \top \). Then it is straightforward to check that \( f \) and \( g \) are continuous,  and that they witness exactly one of the  following four possibilities:
\begin{enumerate}[(1) ]
\item
$S\equiv_\mathsf{W}A$ (in case $\top\in S,\bot\not\in S$);
\item
$S\equiv_\mathsf{W}B$ (in case $\top\not\in S,\bot\not\in S$);
\item
$S\equiv_\mathsf{W}\overline{A}$ (in case $\top\not\in
S,\bot\in S$);
\item
$S\equiv_\mathsf{W}\overline{B}$ (in case
$\top\in S,\bot\in S$).
\end{enumerate}
Therefore, we can extend \( h \) in the obvious way to the desired
isomorphism between the quotient-poset of
$(\mathscr{P}(X),\leq_\mathsf{W})$ and the poset
$(\bar{2}\cdot\omega)+\bar{4}$.
 \end{proof}

By the previous proof, the \( n \)-th level of the Wadge hierarchy on \( X \) (for \( n \in \omega \)) is occupied by the pair of Wadge degrees \( (\boldsymbol{\Sigma}^{-1}_n(X) \setminus \boldsymbol{\Pi}^{-1}_n(X) , \boldsymbol{\Pi}^{-1}_n(X) \setminus \boldsymbol{\Sigma}^{-1}_n(X)) \), while  the four Wadge degrees on the top of the hierarchy are exactly \( [A]_\mathsf{W} \), \( [B]_\mathsf{W} \) , \( [\overline{A}]_\mathsf{W} \), and \( [\overline{B}]_\mathsf{W} \), where \( A,B \subseteq X \) are defined as in the proof of Proposition~\ref{propco3}.

\begin{remark}
Contrarily to the case of Polish spaces, there is no obvious  relation between the dimension of a scattered $\omega$-algebraic domain and the Wadge
hierarchy on it. For example, the dimension of the space $X$ considered in Propositions~\ref{propco3} and~\ref{propco4} is $\omega$, and the Wadge hierarchy on  $X$ is (good but) not very good. On the other hand, the
dimension of the space $\mathsf{C}_\infty$ from Example~\ref{exdim3}(2) is $\infty$, while the
Wadge hierarchy on $\mathsf{C}_\infty$ is very good,
as one can easily check using an argument similar to the one of Proposition~\ref{propco4}.
 \end{remark}

We end this subsection with some natural questions that are left
open by Propositions~\ref{propco3} and~\ref{propco4}.

\begin{question}
\begin{enumerate}[(1)]
\item
Is there a countable quasi-Polish space \( X \) with a (very) bad \( \mathsf{D}_1 \)-hierarchy?
\item
Is there a (necessarily non scattered) quasi-Polish space whose \( \mathsf{D}_2 \)-hierarchy is not very good? If yes, can it be even (very) bad?
\item
Is there a (necessarily uncountable) \emph{Polish} space whose \( \mathsf{D}_\alpha \)-hierarchy is good but not very good (for some \( 1 \leq \alpha < \omega_1 \))? Notice that by Theorem~\ref{theorbad} the \( \mathsf{D}_1 \)-hierarchy on any Polish space  \( X \) is either very good or bad, depending on whether \( \dim(X) = 0 \) or not.
\end{enumerate}
\end{question}

\end{document}